\documentclass[12pt]{amsart}
\usepackage{amsfonts}
\usepackage{amssymb, amstext, amscd, amsmath}
\usepackage{amsthm}
\usepackage{times}
\usepackage{fullpage}
\usepackage{txfonts} 
\usepackage{indentfirst}
\usepackage[all]{xy}
\usepackage{subfigure}
\usepackage{graphicx}
\usepackage{tikz}

\usepackage{nicematrix} 
\usepackage{amsfonts}
\usepackage{mathrsfs}
\usepackage{lscape}
\usepackage{setspace} 
\usepackage{graphicx}    

\usepackage{geometry}
\usepackage{subfigure} 
\usepackage{xcolor} 
\usepackage{color}              
\usepackage{indentfirst} 
\usepackage{url}
\usepackage[colorlinks = true,
linkcolor=blue,
urlcolor=blue,
citecolor=blue,
anchorcolor=blue,
backref=page]{hyperref}

\usepackage[nocompress]{cite}
\newtheoremstyle{noparens}%
{}{}%
{\itshape}{}%
{\bfseries}{.}%
{ }%
{\thmname{#1}\thmnumber{ #2}\mdseries\thmnote{ #3}}
\theoremstyle{noparens}
\allowdisplaybreaks
\newtheorem{theorem}{Theorem}[section]
\newtheorem{corollary}{Corollary}[section]
\newtheorem{proposition}{Proposition}[section]
\newtheorem{lemma}{Lemma}[section]
\newtheorem{remark}{Remark}[section]
\newtheorem{definition}{Definition}[section]
\newtheorem{example}{Example}[section]

\numberwithin{equation}{section}
\numberwithin{equation}{section}

\renewcommand{\baselinestretch}{1.2} 

\begin{document}
	
\title[Two-variable Parity Polynomial for Virtual Knotoids]{Two-variable Parity Polynomial for Virtual Knotoids}
	
\author{Siqi Ding}
\address{School of Mathematical Sciences, Dalian University of Technology, Dalian 116024, P. R. China}
\email{sqding@yeah.net}

\author{Suo Gao}
\address{School of Information Science and Engineering, Dalian Polytechnic University, Dalian 116034, China}
\email{gaosuo@dlpu.edu.cn}

\author{Fengchun Lei}
\address{Beijing Institute of Mathematical Sciences and Applications, Beijing 101408, P.R. China; School of Mathematical Sciences, Dalian University of Technology, Dalian 116024, P.R. China. }
\email{leifengchun@bimsa.cn}
	
\author{Fengling Li*}
\address{School of Mathematical Sciences, Dalian University of Technology, Dalian 116024, P. R. China}
\email{fenglingli@dlut.edu.cn}
	
\author{Andrei Vesnin}
\address{Sobolev Institute of Mathematics of the Siberian Branch of the Russian Academy of Sciences, Novosibirsk 630090, Russia; Regional Scientific and Educational Mathematical Center, Tomsk State University, Tomsk, 634050, Russia}
\email{vesnin@math.nsc.ru}

\thanks{F.\,Lei supported in part by a grant of NSFC (No. 12331003); F.\,Li supported by the Fundamental Research Funds for the Central Universities (No. DUT25LAB302); A.\,V. supported by the Ministry of Sciences and Higher Education of Russia (agreement no. 075-02-2026-1339).}

\date{\today} 
	
\subjclass{57K12, 57K16}

\keywords{Virtual knotoids; parity of crossings; Vassiliev invariant of order one; }
	
\begin{abstract}
In this paper, we introduce a two-variable parity polynomial invariant for virtual knotoids, defined on oriented virtual knotoid diagrams. The construction is based on the parity of classical crossings, where each crossing is classified as even or odd and treated accordingly in the definition of the invariant. We study several fundamental properties of this invariant. We demonstrate that the parity polynomial can distinguish pairs of virtual knotoids that are not distinguished by the odd writhe, prove that it is equivalent to the affine index polynomial, and establish that it is a Vassiliev invariant of order one. Finally, we give its relationship with the Petit gluing invariant. 
\end{abstract}
\maketitle

\section{Introduction}

Knotoid theory, introduced by Turaev~\cite{Tur1} in 2012, generalizes classical knot theory by considering open-ended knot diagrams in oriented surfaces. A knotoid diagram is a generic immersion of the unit interval into a surface with finitely many transversal double points endowed with over/under-crossing information. In contrast to classical knots or long knots, the endpoints of a knotoid are allowed to lie in different regions of the diagram. 

The construction and study of knotoid invariants have become a central topic in knotoid theory. 
In 2017, G\"{u}g\"{u}mc\"{u} and Kauffman~\cite{GK} developed the theory of virtual knotoids, and constructed several new invariants for knotoids, such as the odd writhe, the parity bracket polynomial, the affine index polynomial and the arrow polynomial. In~\cite{GN}, G\"{u}g\"{u}mc\"{u} and Nelson constructed enhanced biquandle counting invariants for knotoids in $\mathbb{R}^2$ or $S^2$ that can distinguish knotoids which the counting invariant fails to distinguish. In~\cite{FL}, Feng and Li introduced a polynomial invariant for knotoids, called the $F$-polynomial, and studied its properties. 

The notion of finite type (Vassiliev) invariants was introduced by Vassiliev~\cite{Vas1, Vas2}. A systematic history and theory of these invariants is provided in~\cite{CDM}. Goussarov, Polyak, and Viro~\cite{GPV} introduced a new notion of finite type invariant and a universal invariant of order $\leq n$ via a Gauss diagram formula. Im, Kim and Lee~\cite{IKL} introduced a two-variable parity polynomial invariant for oriented virtual knots, refining the odd writhe polynomial by means of a modified warping degree and verified that this invariant is a Vassiliev invariant.  In 2021, Manouras, Lambropoulou, and Kauffman~\cite{MLK} extended Vassiliev invariant theory to knotoids via two approaches: using knot closures to define knotoid invariants, and directly defining finite type invariants on knotoids via the Vassiliev skein relation. They proved the existence of non-trivial invariants of order one for spherical knotoids, classified linear chord diagrams of order one, and presented examples from the affine index polynomial and extended bracket polynomial. Feng, Li, and Vesnin~\cite{FLV} introduced a three-variable transcendental invariant of planar knotoids, defined via an index function of a Gauss diagram and proved that this invariant is a Vassiliev invariant of order one. 

Parity is a property that assigns an even or odd label to each crossing of a knot diagram via a preassigned rule. When the rule satisfies axioms compatible with Reidemeister moves, it enables the construction of knot invariants. Kauffman introduced invariants of virtual knots using odd crossings and defined the odd self-linking index~\cite{Kau2}. Manturov studied knot theories with a parity property for crossings, constructed related invariants, proved the non-triviality of free knots, and strengthened known invariants such as the Kauffman bracket by parity considerations~\cite{Man}. Based on Manturov's parity axioms, Cheng and Gao constructed polynomial invariants for virtual knots and links~\cite{CG}. In \cite{GK}, G\"ug\"umc\"u and Kauffman extended parity to knotoids and redefined the parity bracket polynomial for both classical and virtual knotoids. In \cite{GK2}, they further studied parity in planar and spherical knotoids, introduced a planar version of the parity bracket polynomial, and proved Turaev's conjecture that minimal diagrams of knot-type knotoids have zero height. For more results on parity, we refer the reader to~\cite{Afa,KK}.	

Based matrices for flat virtual knots corresponding to virtual strings were first introduced by Turaev in~\cite{Tur2}. Based matrices were later generalized by Henrich~\cite{Hen} to flat singular virtual knots with a single singular crossing. Petit~\cite{Pet} further extended the construction to framed virtual knots and to long virtual knots, both framed and unframed. In 2017, Cahn studied a generalized cobracket for virtual strings and proved that it yields a stronger minimal self-intersection bound than Turaev's virtual cobracket~\cite{Cahn}.
In 2020, Freund generalized Turaev's based matrices for virtual 1-strings to construct multistring based matrices for virtual $n$-strings~\cite{Fre}.

In this paper, for an oriented virtual  knotoid $K$, we define a two-variable parity polynomial $\mathbf{P}_D(x,y) \in \mathbb Z[x^{\pm 1},  y^{\pm 1}]$ constructed from a diagram $D$ of $K$, 
$$
\mathbf{P}_D(x,y)=\sum_{c \in E(D)} \operatorname{sgn}(c) i(c) x^{W_D(c)} + \sum_{c \in O(D)} \operatorname{sgn}(c) i(c) y^{W_D(c)}, 
$$
where $O(D)$ is the set of odd crossings in a diagram $D$ of an oriented virtual knotoid, and $E(D)$ is the set of even crossings in $D$, see Definition~\ref{def2.1} for details. The construction involves an integer labeling of arcs, analogous to the Cheng~\cite{Che} coloring for virtual knot diagrams, a crossing parity, analogous to parity, introduced by Manturov~\cite{Man} for virtual knot diagrams, and an intersection index $i(c)$, analogous to intersection index for virtual knot diagrams, introduced by Henrich~\cite{Hen}, see Section~\ref{section2}. In Theorem~\ref{th2.1} we prove that $\mathbf{P}_D(x,y)$ that it is an invariant of virtual knotoids and demonstrate in Example~\ref{exp2.1} that this invariant is non-trivial. In Section~\ref{section3} we demonstrate some properties of $\mathbf{P}_D(x,y)$. In Example~\ref{example3.0}, we present a pair of virtual knotoids that cannot be distinguished by the odd writhe invariant, whereas they are distinguished by $\mathbf{P}_K(x,y)$. Theorem~\ref{thm2} shows that $\textbf{P}_K(x,y)$ is equivalent to the affine index polynomial $AI_K(t)$. However, $\mathbf{P}_K(t,t)$ is, in general, not equal to $AI_K(t)$, see Example~\ref{exp3}. In Theorem~\ref{thm3} we prove  that $\textbf{P}_K(x,y)$ is a Vassiliev invariant of order one. In Proposition~\ref{prop3.1} we describe how this invariant changes when we consider  the inverse image and the mirror image of $K$. In particular, $\textbf{P}_K(x,y)$ vanishes if virtual knotoid $K$ is both invertible and amphichiral.  Finally, in Section~\ref{section4}, we compare $\textbf{P}_K(x,y)$ with Petit gluing invariant and prove that the latter is strictly stronger, see Theorem~\ref{prop4.1}.  

%%%%%
\section{Two-variable parity polynomial $\mathbf{P}_K(x,y)$ of virtual knotoids} \label{section2}

\subsection{Virtual knotoids and Reidemeister moves} \label{subsection2.1}

In this subsection we briefly recall the basic notions and results on virtual knotoids that will be used throughout the paper, see~\cite{GK,Tur1}.

A \textit{knotoid diagram} $D$ on an oriented surface $\Sigma$ is a generic immersion of $[0,1]$ into $\Sigma$ with finitely many transverse double points endowed with over/under-crossing information, called \textit{classical crossings}. The endpoints of $D$ are defined as the images of $0$ and $1$, called the \textit{tail} and the \textit{head}, respectively, these points are distinct from each other and from any crossing. The \textit{orientation} of $D$ is from the tail to the head.

A \textit{virtual knotoid diagram} is a knotoid diagram in $S^2$ equipped with \textit{virtual crossings}. Such a diagram has two types of crossings: classical crossings in Fig.~\ref{fig1}(a), and virtual crossings, represented by a small circle with no over/under-crossing information, see Fig.~\ref{fig1}(b). 
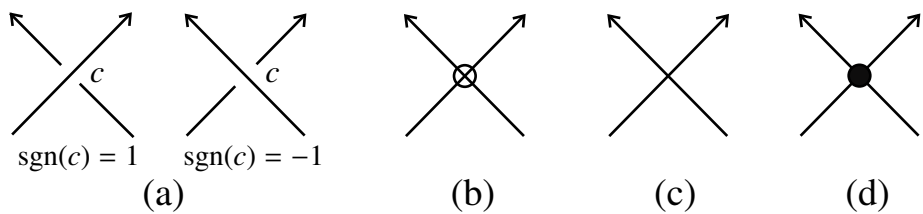
\begin{figure}[!ht]
	\centering
	\tikzset{every picture/.style={line width=1.0pt}} 
	\begin{tikzpicture}[x=0.75pt,y=0.75pt,yscale=-1,xscale=1]
		\draw    (201.26,306.05) -- (142.17,366.42) ;
		\draw   (142.41,311.9) -- (141.66,305.97) -- (147.43,307.17) ;
		\draw   (195.79,307.32) -- (201.51,305.91) -- (200.98,311.87) ;
		\draw    (176.09,340.88) -- (203.09,367.71) ;
		\draw    (288.81,306.61) -- (264.76,331.15) ;
		\draw    (229.6,306.67) -- (288.93,366.92) ;
		\draw   (230.25,312.46) -- (229.51,306.53) -- (235.27,307.73) ;
		\draw   (283.34,307.88) -- (289.06,306.47) -- (288.53,312.43) ;
		\draw    (254.09,342.15) -- (229.72,366.98) ;
		\draw    (398.82,306.85) -- (339.74,367.22) ;
		\draw    (339.62,306.91) -- (398.94,367.16) ;
		\draw   (340.27,312.7) -- (339.52,306.77) -- (345.29,307.97) ;
		\draw   (393.35,308.12) -- (399.07,306.71) -- (398.54,312.67) ;
		\draw    (141.17,305.66) -- (166.76,331.54) ;
		\draw    (500.82,306.83) -- (441.74,367.21) ;
		\draw    (441.62,306.9) -- (500.94,367.14) ;
		\draw   (442.27,312.69) -- (441.52,306.76) -- (447.29,307.95) ;
		\draw   (495.35,308.11) -- (501.07,306.7) -- (500.54,312.65) ;
		\draw   (363.86,337.04) .. controls (363.86,334.04) and (366.29,331.62) .. (369.28,331.62) .. controls (372.27,331.62) and (374.7,334.04) .. (374.7,337.04) .. controls (374.7,340.03) and (372.27,342.45) .. (369.28,342.45) .. controls (366.29,342.45) and (363.86,340.03) .. (363.86,337.04) -- cycle ;
		\draw    (596.82,306.65) -- (537.74,367.03) ;
		\draw    (537.62,306.72) -- (596.94,366.96) ;
		\draw   (538.27,312.51) -- (537.52,306.58) -- (543.29,307.77) ;
		\draw   (591.35,307.93) -- (597.07,306.52) -- (596.54,312.47) ;
		\draw  [fill={rgb, 255:red, 14; green, 13; blue, 13 }  ,fill opacity=1 ] (561.86,336.84) .. controls (561.86,333.85) and (564.29,331.42) .. (567.28,331.42) .. controls (570.27,331.42) and (572.7,333.85) .. (572.7,336.84) .. controls (572.7,339.83) and (570.27,342.26) .. (567.28,342.26) .. controls (564.29,342.26) and (561.86,339.83) .. (561.86,336.84) -- cycle ;
		\draw (179.67,331.25) node [anchor=north west][inner sep=0.75pt]    {$c$};
		\draw (267.6,331.25) node [anchor=north west][inner sep=0.75pt]    {$c$};
		\draw (143.32,370.02) node [anchor=north west][inner sep=0.75pt]  [font=\small]  {$\text{sgn}(c) =1$};
		\draw (226.44,370.02) node [anchor=north west][inner sep=0.75pt]  [font=\small]  {$\text{sgn}(c) =-1$};
		\draw (205.83,387.01) node [anchor=north west][inner sep=0.75pt]  [font=\large] [align=left] {{\fontfamily{ptm}\selectfont (a)}};
		\draw (360.53,387.01) node [anchor=north west][inner sep=0.75pt]  [font=\large] [align=left] {{\fontfamily{ptm}\selectfont (b)}};
		\draw (462.53,387.01) node [anchor=north west][inner sep=0.75pt]  [font=\large] [align=left] {{\fontfamily{ptm}\selectfont (c)}};
		\draw (558.53,387.01) node [anchor=north west][inner sep=0.75pt]  [font=\large] [align=left] {{\fontfamily{ptm}\selectfont (d)}};
	\end{tikzpicture}
	\caption{Crossings: (a) classical, (b) virtual, (c) flat and (d) singular.} \label{fig1}
\end{figure}

A \textit{flat knotoid diagram} is obtained from a knotoid diagram by forgetting the over/under-crossing information at each classical crossing, so that all classical crossings are replaced by flat crossings without crossing data, see Fig.~\ref{fig1}(c). It can be viewed as the shadow (or, the flattening) of the corresponding knotoid diagram.

For later use, we also illustrate singular crossings in Fig.~\ref{fig1}(d). The notion of singular knotoids will be introduced in  subsection~\ref{subsection3.3} and will play a role in the study of Vassiliev invariants.

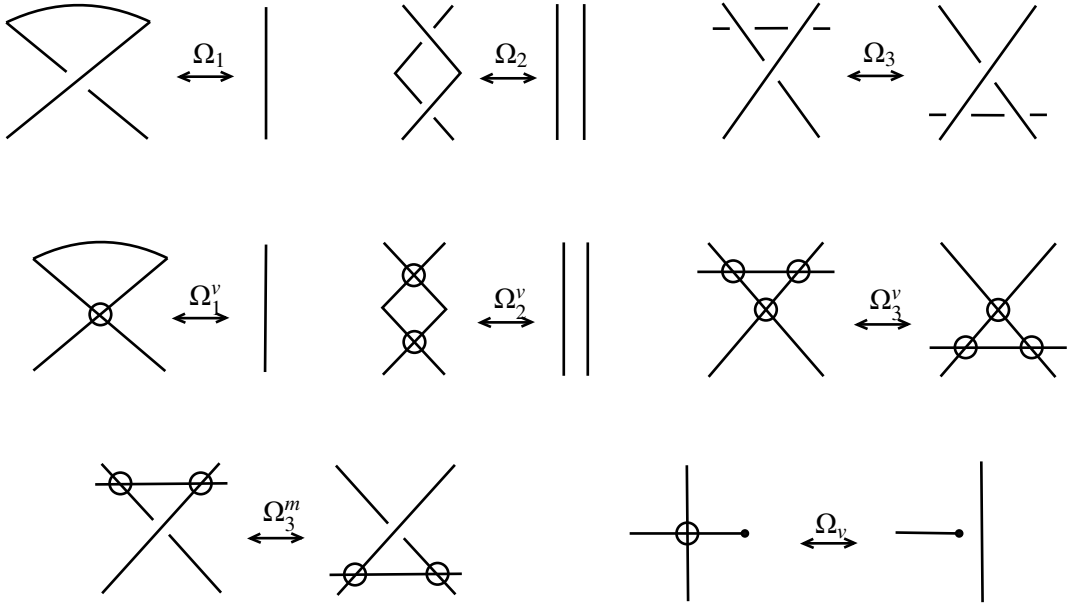
\begin{figure}[!htbp]
	\centering
		\tikzset{every picture/.style={line width=1.0pt}} 
		\begin{tikzpicture}[x=0.75pt,y=0.75pt,yscale=-0.9,xscale=0.9]	
			\draw    (48.88,161.75) -- (122.56,224.35) ;
			\draw    (178.71,153.8) -- (178.16,224.88) ;
			\draw    (243.86,226.37) -- (279.02,189.89) -- (244.33,152.77) ;
			\draw    (278.56,152.81) -- (243.58,189) -- (278.35,226.37) ;
			\draw    (344.68,152.61) -- (344.68,226.97) ;
			\draw    (358.06,152.61) -- (358.06,226.97) ;
			\draw    (489.31,152.61) -- (425.44,227.57) ;
			\draw    (425.25,152.96) -- (489.5,227.21) ;
			\draw    (419.14,169.02) -- (495.59,169.02) ;
			\draw    (152.79,275.73) -- (87.34,348.27) ;
			\draw    (83.47,287.55) -- (157.2,286.95) ;
			\draw    (87.01,276.09) -- (115.81,307.39) ;
			\draw    (124.6,315.58) -- (153.63,347.9) ;
			\draw    (48.88,161.75) .. controls (73.94,149.37) and (98.85,149.37) .. (123.76,161.4) ;
			\draw    (49.02,224.06) -- (123.76,161.4) ;
			\draw   (150.76,192.15) -- (156.47,194.67) -- (150.76,197.2) ;
			\draw    (155.76,194.83) -- (128.65,194.83) ;
			\draw   (133.76,197.38) -- (127.94,194.97) -- (133.53,192.33) ;
			\draw   (321.08,194.5) -- (326.79,197.03) -- (321.08,199.55) ;
			\draw    (326.07,197.19) -- (298.97,197.19) ;
			\draw   (304.08,199.73) -- (298.26,197.32) -- (303.84,194.68) ;
			\draw    (618.79,227.72) -- (554.91,152.76) ;
			\draw    (554.73,227.36) -- (618.97,153.12) ;
			\draw    (548.62,211.31) -- (625.07,211.31) ;
			\draw   (531.01,195.64) -- (536.72,198.17) -- (531.01,200.7) ;
			\draw    (536.01,198.33) -- (508.9,198.33) ;
			\draw   (514.01,200.87) -- (508.19,198.46) -- (513.78,195.82) ;
			\draw    (218.52,349.21) -- (283.98,276.67) ;
			\draw    (287.84,337.39) -- (214.11,337.99) ;
			\draw    (284.3,348.85) -- (255.5,317.55) ;
			\draw    (246.72,309.36) -- (217.69,277.04) ;
			\draw   (192.75,314.83) -- (198.45,317.36) -- (192.75,319.89) ;
			\draw    (198.74,317.52) -- (171.64,317.52) ;
			\draw   (175.75,320.06) -- (169.93,317.66) -- (175.51,315.01) ;
			\draw   (80.37,192.73) .. controls (80.37,189.4) and (83.07,186.71) .. (86.39,186.71) .. controls (89.72,186.71) and (92.41,189.4) .. (92.41,192.73) .. controls (92.41,196.05) and (89.72,198.75) .. (86.39,198.75) .. controls (83.07,198.75) and (80.37,196.05) .. (80.37,192.73) -- cycle ;
			\draw   (255.05,170.9) .. controls (255.05,167.58) and (257.74,164.88) .. (261.07,164.88) .. controls (264.39,164.88) and (267.09,167.58) .. (267.09,170.9) .. controls (267.09,174.23) and (264.39,176.93) .. (261.07,176.93) .. controls (257.74,176.93) and (255.05,174.23) .. (255.05,170.9) -- cycle ;
			\draw   (255.42,208.13) .. controls (255.42,204.81) and (258.12,202.11) .. (261.44,202.11) .. controls (264.77,202.11) and (267.46,204.81) .. (267.46,208.13) .. controls (267.46,211.46) and (264.77,214.15) .. (261.44,214.15) .. controls (258.12,214.15) and (255.42,211.46) .. (255.42,208.13) -- cycle ;
			\draw   (433.09,168.71) .. controls (433.09,165.39) and (435.79,162.69) .. (439.11,162.69) .. controls (442.44,162.69) and (445.13,165.39) .. (445.13,168.71) .. controls (445.13,172.04) and (442.44,174.73) .. (439.11,174.73) .. controls (435.79,174.73) and (433.09,172.04) .. (433.09,168.71) -- cycle ;
			\draw   (469.49,168.71) .. controls (469.49,165.39) and (472.19,162.69) .. (475.51,162.69) .. controls (478.84,162.69) and (481.53,165.39) .. (481.53,168.71) .. controls (481.53,172.04) and (478.84,174.73) .. (475.51,174.73) .. controls (472.19,174.73) and (469.49,172.04) .. (469.49,168.71) -- cycle ;
			\draw   (451.35,190.09) .. controls (451.35,186.76) and (454.05,184.07) .. (457.37,184.07) .. controls (460.7,184.07) and (463.4,186.76) .. (463.4,190.09) .. controls (463.4,193.41) and (460.7,196.11) .. (457.37,196.11) .. controls (454.05,196.11) and (451.35,193.41) .. (451.35,190.09) -- cycle ;
			\draw   (580.83,190.24) .. controls (580.83,186.92) and (583.52,184.22) .. (586.85,184.22) .. controls (590.18,184.22) and (592.87,186.92) .. (592.87,190.24) .. controls (592.87,193.57) and (590.18,196.26) .. (586.85,196.26) .. controls (583.52,196.26) and (580.83,193.57) .. (580.83,190.24) -- cycle ;
			\draw   (562.69,211.11) .. controls (562.69,207.79) and (565.39,205.09) .. (568.71,205.09) .. controls (572.04,205.09) and (574.73,207.79) .. (574.73,211.11) .. controls (574.73,214.44) and (572.04,217.13) .. (568.71,217.13) .. controls (565.39,217.13) and (562.69,214.44) .. (562.69,211.11) -- cycle ;
			\draw   (599.49,211.11) .. controls (599.49,207.79) and (602.19,205.09) .. (605.51,205.09) .. controls (608.84,205.09) and (611.53,207.79) .. (611.53,211.11) .. controls (611.53,214.44) and (608.84,217.13) .. (605.51,217.13) .. controls (602.19,217.13) and (599.49,214.44) .. (599.49,211.11) -- cycle ;
			\draw   (91.49,286.99) .. controls (91.49,283.67) and (94.19,280.97) .. (97.51,280.97) .. controls (100.84,280.97) and (103.53,283.67) .. (103.53,286.99) .. controls (103.53,290.32) and (100.84,293.01) .. (97.51,293.01) .. controls (94.19,293.01) and (91.49,290.32) .. (91.49,286.99) -- cycle ;
			\draw   (136.29,286.99) .. controls (136.29,283.67) and (138.99,280.97) .. (142.31,280.97) .. controls (145.64,280.97) and (148.33,283.67) .. (148.33,286.99) .. controls (148.33,290.32) and (145.64,293.01) .. (142.31,293.01) .. controls (138.99,293.01) and (136.29,290.32) .. (136.29,286.99) -- cycle ;
			\draw   (222.29,337.79) .. controls (222.29,334.47) and (224.99,331.77) .. (228.31,331.77) .. controls (231.64,331.77) and (234.33,334.47) .. (234.33,337.79) .. controls (234.33,341.12) and (231.64,343.81) .. (228.31,343.81) .. controls (224.99,343.81) and (222.29,341.12) .. (222.29,337.79) -- cycle ;
			\draw   (268.29,337.39) .. controls (268.29,334.07) and (270.99,331.37) .. (274.31,331.37) .. controls (277.64,331.37) and (280.33,334.07) .. (280.33,337.39) .. controls (280.33,340.72) and (277.64,343.41) .. (274.31,343.41) .. controls (270.99,343.41) and (268.29,340.72) .. (268.29,337.39) -- cycle ;
			\draw    (272.34,83.03) -- (283.26,95.48) ;
			\draw    (265.42,41.06) -- (250.7,58.33) -- (265.42,75.18) ;
			\draw    (282.83,20.53) -- (272.24,33.03) ;
			\draw    (33.88,29.97) .. controls (60.71,17.15) and (87.39,17.15) .. (114.07,29.61) ;
			\draw    (34.03,94.51) -- (114.07,29.61) ;
			\draw    (33.88,29.97) -- (67.52,57.52) ;
			\draw   (153.43,57.49) -- (159.13,60.02) -- (153.43,62.54) ;
			\draw    (158.42,60.18) -- (131.32,60.18) ;
			\draw   (136.43,62.72) -- (130.61,60.31) -- (136.19,57.67) ;
			\draw    (178.74,21.38) -- (178.7,95.15) ;
			\draw    (254.51,95.3) -- (287.07,58.15) -- (254.94,20.35) ;
			\draw    (341.03,19.71) -- (341.03,95.48) ;
			\draw    (356.07,19.71) -- (356.07,95.48) ;
			\draw    (487.16,19.24) -- (433.54,94.81) ;
			\draw    (427.73,33.39) -- (437.22,33.39) ;
			\draw    (79.65,67.64) -- (112.78,94.81) ;
			\draw    (433.2,19.24) -- (456.29,51.31) ;
			\draw    (464.33,62.57) -- (487.12,94.1) ;
			\draw    (451.1,33.21) -- (469.63,33.14) ;
			\draw    (483.72,33.37) -- (493.2,33.37) ;
			\draw   (322.32,58.4) -- (328.02,60.93) -- (322.32,63.46) ;
			\draw    (327.31,61.09) -- (300.21,61.09) ;
			\draw   (305.32,63.64) -- (299.5,61.23) -- (305.08,58.58) ;
			\draw   (527.03,57.07) -- (532.74,59.6) -- (527.03,62.13) ;
			\draw    (532.02,59.76) -- (504.92,59.76) ;
			\draw   (510.03,62.3) -- (504.21,59.89) -- (509.79,57.25) ;
			\draw    (607.94,20.58) -- (554.14,95.48) ;
			\draw    (548.31,81.46) -- (557.83,81.45) ;
			\draw    (553.79,20.58) -- (576.97,52.37) ;
			\draw    (585.03,63.53) -- (607.9,94.77) ;
			\draw    (571.75,81.63) -- (590.35,81.7) ;
			\draw    (604.49,81.47) -- (614.01,81.47) ;
			\draw    (381.4,314.9) -- (445.66,314.9) ;
			\draw    (413.28,276.87) -- (413.98,352.67) ;
			\draw    (529.64,314.38) -- (565.11,314.9) ;
			\draw    (577.43,274.73) -- (578.13,353.22) ;
			\draw   (500.98,317.72) -- (506.69,320.25) -- (500.98,322.78) ;
			\draw    (505.98,320.41) -- (478.87,320.41) ;
			\draw   (483.98,322.95) -- (478.16,320.54) -- (483.75,317.9) ;
			\draw  [fill={rgb, 255:red, 13; green, 13; blue, 13 }  ,fill opacity=1 ] (444,314.9) .. controls (444,313.98) and (444.75,313.24) .. (445.66,313.24) .. controls (446.58,313.24) and (447.32,313.98) .. (447.32,314.9) .. controls (447.32,315.81) and (446.58,316.55) .. (445.66,316.55) .. controls (444.75,316.55) and (444,315.81) .. (444,314.9) -- cycle ;
			\draw  [fill={rgb, 255:red, 13; green, 13; blue, 13 }  ,fill opacity=1 ] (563.45,314.9) .. controls (563.45,313.98) and (564.19,313.24) .. (565.11,313.24) .. controls (566.02,313.24) and (566.76,313.98) .. (566.76,314.9) .. controls (566.76,315.81) and (566.02,316.55) .. (565.11,316.55) .. controls (564.19,316.55) and (563.45,315.81) .. (563.45,314.9) -- cycle ;
			\draw   (407.51,314.9) .. controls (407.51,311.57) and (410.2,308.87) .. (413.53,308.87) .. controls (416.85,308.87) and (419.55,311.57) .. (419.55,314.9) .. controls (419.55,318.22) and (416.85,320.92) .. (413.53,320.92) .. controls (410.2,320.92) and (407.51,318.22) .. (407.51,314.9) -- cycle ;
			\draw (134.19,174.41) node [anchor=north west][inner sep=0.75pt]  [font=\small]  {$\Omega_{1}^{v}$};
			\draw (513.38,176.98) node [anchor=north west][inner sep=0.75pt]  [font=\small]  {$\Omega_{3}^{v}$};
			\draw (304.01,176.26) node [anchor=north west][inner sep=0.75pt]  [font=\small]  {$\Omega_{2}^{v}$};
			\draw (175.26,294.73) node [anchor=north west][inner sep=0.75pt]  [font=\small]  {$\Omega_{3}^{m}$};
			\draw (136.31,39.62) node [anchor=north west][inner sep=0.75pt]  [font=\small]  {$\Omega_{1}$};
			\draw (509.9,38.9) node [anchor=north west][inner sep=0.75pt]  [font=\small]  {$\Omega_{3}$};
			\draw (304.66,39.56) node [anchor=north west][inner sep=0.75pt]  [font=\small]  {$\Omega_{2}$};
			\draw (483.19,301.86) node [anchor=north west][inner sep=0.75pt]  [font=\small]  {$\Omega_{v}$};
		\end{tikzpicture}
	\caption{Generalized Reidemeister moves.\label{fig2}} 
\end{figure}
A \textit{virtual knotoid} is defined as an equivalence class of virtual knotoid diagrams modulo the generalized Reidemeister moves $\textit{g}\mathcal{R}$, namely the classical moves $\Omega_1, \Omega_2, \Omega_3$, the virtual moves $\Omega_1^v, \Omega_2^v, \Omega_3^v$, the mixed move $\Omega_3^m$, and $\Omega_{v}$-move as shown in Fig.~\ref{fig2}. The moves $\Phi_{+}$ and $\Phi_{-}$, illustrated in Fig.~\ref{fig3}, are called the \textit{forbidden moves} for knotoid diagrams, since any knotoid diagram in $\Sigma$ can be transformed into the trivial diagram by a finite sequence of such moves.
\begin{figure}[htbp]
	\centering
	\tikzset{every picture/.style={line width=1.0pt}} 
	\begin{tikzpicture}[x=0.75pt,y=0.75pt,yscale=-1.0,xscale=1.0]  	
		\draw    (42.82,349.32) -- (117.64,348.82) ;
		\draw    (80.96,304.68) -- (81.26,339.78) ;
		\draw    (81.26,358.86) -- (81.26,390.55) ;
		\draw    (205.82,350.39) -- (238.57,350.39) ;
		\draw    (255.07,305.75) -- (255.64,392.55) ;
		\draw    (401.07,350.18) -- (431.11,350.28) ;
		\draw    (392.04,306.06) -- (392.57,391.55) ;
		\draw    (349.64,350.29) -- (380.38,350.28) ;
		\draw   (170.33,351.01) -- (176.04,353.54) -- (170.33,356.06) ;
		\draw    (175.32,353.7) -- (148.22,353.7) ;
		\draw   (153.33,356.24) -- (147.51,353.83) -- (153.1,351.19) ;
		\draw   (317,350.34) -- (322.7,352.87) -- (317,355.4) ;
		\draw    (321.99,353.03) -- (294.89,353.03) ;
		\draw   (300,355.58) -- (294.18,353.17) -- (299.76,350.52) ;
		\draw  [fill={rgb, 255:red, 13; green, 13; blue, 13 }  ,fill opacity=1 ] (115.98,348.82) .. controls (115.98,347.9) and (116.72,347.16) .. (117.64,347.16) .. controls (118.55,347.16) and (119.29,347.9) .. (119.29,348.82) .. controls (119.29,349.73) and (118.55,350.47) .. (117.64,350.47) .. controls (116.72,350.47) and (115.98,349.73) .. (115.98,348.82) -- cycle ;
		\draw  [fill={rgb, 255:red, 13; green, 13; blue, 13 }  ,fill opacity=1 ] (236.91,350.39) .. controls (236.91,349.48) and (237.65,348.73) .. (238.57,348.73) .. controls (239.48,348.73) and (240.23,349.48) .. (240.23,350.39) .. controls (240.23,351.31) and (239.48,352.05) .. (238.57,352.05) .. controls (237.65,352.05) and (236.91,351.31) .. (236.91,350.39) -- cycle ;
		\draw  [fill={rgb, 255:red, 13; green, 13; blue, 13 }  ,fill opacity=1 ] (429.45,350.28) .. controls (429.45,349.36) and (430.2,348.62) .. (431.11,348.62) .. controls (432.03,348.62) and (432.77,349.36) .. (432.77,350.28) .. controls (432.77,351.19) and (432.03,351.93) .. (431.11,351.93) .. controls (430.2,351.93) and (429.45,351.19) .. (429.45,350.28) -- cycle ;
		\draw (152.97,333.09) node [anchor=north west][inner sep=0.75pt]  [font=\small]  {$\Phi _{+}$};
		\draw (299.63,333.63) node [anchor=north west][inner sep=0.75pt]  [font=\small]  {$\Phi _{-}$};
	\end{tikzpicture}
	\caption{Forbidden knotoid moves $\Phi_{+}$ and $\Phi_{-}$.\label{fig3}}	
\end{figure}
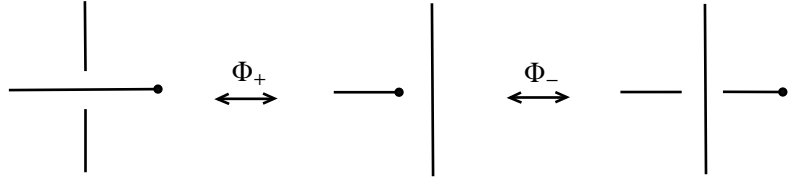
For each classical crossing $c$, a sign $\operatorname{sgn}(c) \in \{1,-1\}$ is defined as shown in Fig.~\ref{fig1}(a). The \textit{writhe} of a (classical or virtual) knotoid diagram $D$ is defined by
\begin{equation}
	\operatorname{wr}(D) = \sum_{c \in D} \operatorname{sgn}(c),
\end{equation}
where the sum runs over all classical crossings. 

We recall the notion of multi-knotoids, following Turaev~\cite{Tur1}, and extend it by admitting virtual and flat versions. A \textit{multi-knotoid diagram} in $S^2$ is an immersion of a single oriented segment together with several oriented circles in $S^2$, where each double point is endowed with over/under-crossing information. A \textit{multi-knotoid} is the equivalence class of such diagrams under the equivalence relation for knotoids. A \textit{virtual multi-knotoid diagram} is a multi-knotoid diagram with virtual crossings. A \textit{virtual multi-knotoid} is the equivalence class of such diagrams under the equivalence relation for virtual knotoids. A \textit{flat virtual multi-knotoid diagram} is a virtual multi-knotoid diagram in which the over/under information at each classical crossing is ignored. 

%%%
\subsection{Definition of $\mathbf{P}_K(x,y)$ and its invariance} \label{subsection2.2}

Now, let us introduce a two-variable parity polynomial invariant for virtual knotoids. The construction of this polynomial is based on assigning integer labels to the arcs of a virtual knotoid diagram and using the parity of crossings. The key idea is to distinguish classical crossings into two types and use this distinction to define the invariant. Before giving the definition of the polynomial, we introduce several necessary notions.

We define the \textit{integer labeling} of arcs, where an arc is taken to be a curve segment connecting two crossings.  Let $D$ be an oriented (virtual) knotoid diagram. Each arc $\alpha$ of $D$ is assigned an integer label $\ell(\alpha)$ starting from the tail, and the labels change only at classical crossings: the label increases by $1$ when passing through a crossing from down right to top left and decreases by $1$ when passing from down left to top right, while remaining unchanged at virtual crossings. The labeling rule is illustrated in Fig.~\ref{fig11}. This integer labeling is analogous to the Cheng coloring for virtual knot diagrams, see~\cite{Che}.
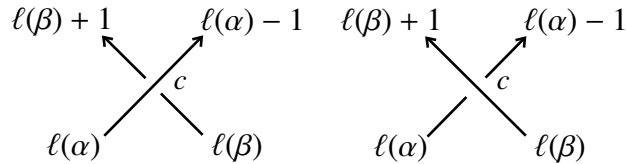
\begin{figure}[!ht]
	\begin{center}
		\tikzset{every picture/.style={line width=1.0pt}}  
		\begin{tikzpicture}[x=0.75pt,y=0.75pt,yscale=-1,xscale=1]	
			\draw    (339.04,32.33) -- (388.49,81.42) ;
			\draw    (225.66,32.01) -- (176.99,81.45) ;
			\draw    (176.89,31.91) -- (197.71,52.56) ;
			\draw   (177.19,36.8) -- (176.57,31.94) -- (181.32,32.93) ;
			\draw   (221.15,33.05) -- (225.86,31.9) -- (225.42,36.77) ;
			\draw    (205.39,60.22) -- (226.34,81) ;
			\draw   (339.33,37.22) -- (338.72,32.36) -- (343.47,33.34) ;
			\draw   (383.3,33.47) -- (388.01,32.31) -- (387.57,37.19) ;
			\draw    (387.8,32.42) -- (368.27,52.3) ;
			\draw    (358.67,62) -- (339.14,81.87) ;
			\draw (210.22,50) node [anchor=north west][inner sep=0.75pt]  [font=\small]  {$c$};
			\draw (372.22,50) node [anchor=north west][inner sep=0.75pt]  [font=\small]  {$c$};
			\draw (146.49,78.53) node [anchor=north west][inner sep=0.75pt]  [font=\normalsize]  {$\ell(\alpha)$};
			\draw (229.1,78.57) node [anchor=north west][inner sep=0.75pt]  [font=\normalsize]  {$\ell(\beta)$};
			\draw (128.85,15.31) node [anchor=north west][inner sep=0.75pt]  [font=\normalsize]  {$\ell(\beta)+1$};
			\draw (223.6,15.75) node [anchor=north west][inner sep=0.75pt]  [font=\normalsize]  {$\ell(\alpha)-1$};
			\draw (310.64,78.95) node [anchor=north west][inner sep=0.75pt]  [font=\normalsize]  {$\ell(\alpha)$};
			\draw (391.25,78.98) node [anchor=north west][inner sep=0.75pt]  [font=\normalsize]  {$\ell(\beta)$};
			\draw (294,15.73) node [anchor=north west][inner sep=0.75pt]  [font=\normalsize]  {$\ell(\beta)+1$};
			\draw (385.74,15.75) node [anchor=north west][inner sep=0.75pt]  [font=\normalsize]  {$\ell(\alpha)-1$};
		\end{tikzpicture}
		\caption{Labeling around a classical crossing $c$ of $D$.} \label{fig11}
	\end{center}
\end{figure}
For a classical crossing $c$ of $D$, let $\ell(\alpha)$ and $\ell(\beta)$ denote the labels of the incoming arcs $\alpha$ and $\beta$, respectively. The \textit{weight} of $c$ is denoted by $W_D(c)$ and defined by
$$
W_D(c) = \operatorname{sgn}(c)\,\big(\ell(\alpha) - (\ell(\beta)+1)\big),
$$
where $\operatorname{sgn}(c)$ is the sign of the crossing. 

The parity of classical crossings in virtual knotoid diagrams is defined in the same way as for virtual knots, see~\cite{Man}. We recall the definition adapted to our setting.

Let $D$ be an oriented virtual knotoid diagram and let $c$ be a classical crossing of $D$. Consider the diagram obtained by applying the orientation-preserving smoothing, also known as  $1$-smoothing, at the crossing $c$. This operation produces a two-component diagram, denoted by $D_1 \cup D_2$. An ordering ($D_1, D_2$) for components $D_1$ and $D_2$ of $D_1 \cup D_2$ is chosen according to the sign of $c$ as shown in Fig.~\ref{fig12}. Let $I(c)$ denote the number of classical crossings between the two components $D_1$ and $D_2$. If $I(c)$ is even, then the crossing $c$ is called \textit{even}, otherwise, it is called \textit{odd}.  This definition is independent of the choice of orientation. We note that this parity is the same as parity of the degree $d(c)$ of crossing $c$ considered in~\cite{FLV}. 
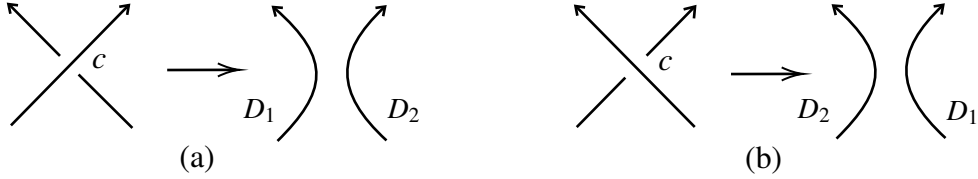
\begin{figure}[!ht]
	\begin{center}
		\tikzset{every picture/.style={line width=1.0pt}}  
		\begin{tikzpicture}[x=0.75pt,y=0.75pt,yscale=-1,xscale=1]	
			\draw    (154.26,40.4) -- (95.17,100.78) ;
			\draw   (149.49,41.25) -- (154.28,40.44) -- (153.47,45.3) ;
			\draw    (129.09,75.23) -- (156.09,102.06) ;
			\draw    (438.11,41.26) -- (414.06,65.8) ;
			\draw    (378.9,41.33) -- (438.23,101.57) ;
			\draw   (432.64,42.53) -- (438.36,41.12) -- (437.83,47.08) ;
			\draw    (403.39,76.8) -- (379.02,101.64) ;
			\draw    (94.17,40.01) -- (119.76,65.9) ;
			\draw    (228.89,108.6) .. controls (253.17,84.5) and (257.06,67.8) .. (227.37,41.93) ;
			\draw    (282.27,41.65) .. controls (257.88,64.79) and (257.34,84.13) .. (283.57,108.32) ; 
			\draw    (509.35,107.4) .. controls (533.64,83.3) and (537.53,66.6) .. (507.84,40.73) ; 
			\draw    (562.73,40.45) .. controls (538.34,63.59) and (537.8,82.93) .. (564.04,107.12) ;
			\draw    (173.55,72.98) -- (207.91,72.9) ;
			\draw [shift={(209.9,72.9)}, rotate = 179.87] [color={rgb, 255:red, 0; green, 0; blue, 0 }  ][line width=0.75]    (10.93,-3.29) .. controls (6.95,-1.4) and (3.31,-0.3) .. (0,0) .. controls (3.31,0.3) and (6.95,1.4) .. (10.93,3.29)   ;
			\draw   (278.04,41.81) -- (282.85,41.15) -- (281.9,45.98) ;
			\draw   (558.09,41.05) -- (562.89,40.28) -- (562.04,45.14) ;
			\draw   (512.57,41.03) -- (507.73,40.65) -- (508.96,45.42) ;
			\draw   (383.21,41.58) -- (378.4,40.89) -- (379.33,45.73) ;
			\draw   (231.59,41.78) -- (226.75,41.4) -- (227.98,46.17) ;
			\draw   (98.71,40.52) -- (93.89,39.91) -- (94.89,44.74) ;
			\draw    (456.42,75.01) -- (490.78,74.93) ;
			\draw [shift={(492.78,74.93)}, rotate = 179.87] [color={rgb, 255:red, 0; green, 0; blue, 0 }  ][line width=0.75]    (10.93,-3.29) .. controls (6.95,-1.4) and (3.31,-0.3) .. (0,0) .. controls (3.31,0.3) and (6.95,1.4) .. (10.93,3.29)   ;
			\draw (134.17,63) node [anchor=north west][inner sep=0.75pt]    {$c$};
			\draw (418.4,64) node [anchor=north west][inner sep=0.75pt]    {$c$};
			\draw (178,109.2) node [anchor=north west][inner sep=0.75pt]    {(a)};
			\draw (461.8,109.8) node [anchor=north west][inner sep=0.75pt]    {(b)};
			\draw (210,85.65) node [anchor=north west][inner sep=0.75pt]  [font=\small]  {$D_{1}$};
			\draw (562.73,86.65) node [anchor=north west][inner sep=0.75pt]  [font=\small]  {$D_{1}$};
			\draw (282.27,85.65) node [anchor=north west][inner sep=0.75pt]  [font=\small]  {$D_{2}$};
			\draw (488,85.65) node [anchor=north west][inner sep=0.75pt]  [font=\small]  {$D_{2}$};
		\end{tikzpicture}
		\caption{Applying the orientation-preserving smoothing at the crossing $c$.} \label{fig12} 
	\end{center}
\end{figure}

The intersection index of an ordered virtual multi-knotoid associated with a smoothed crossing $c$ in virtual knotoid diagrams is defined in the same way as for virtual knots, see~\cite[Def 3.1]{Hen}. We now present the definition adjusted for virtual knotoids. The intersection index $i(c)$ is given by $i(c) = \sum_{e \in D_1 \cap D_2} \operatorname{ind}(e)$, where $\operatorname{ind}(e)$ is defined for classical crossings $e \in D_1 \cap D_2$ as shown in Fig.~\ref{fig511}. In case (a), when $D_1$ is over, we suppose $\operatorname{ind}(e)=\operatorname{sgn}(e)$; \and in case (b), when $D_2$ is over, we suppose $\operatorname{ind}(e)=-\operatorname{sgn}(e)$. 

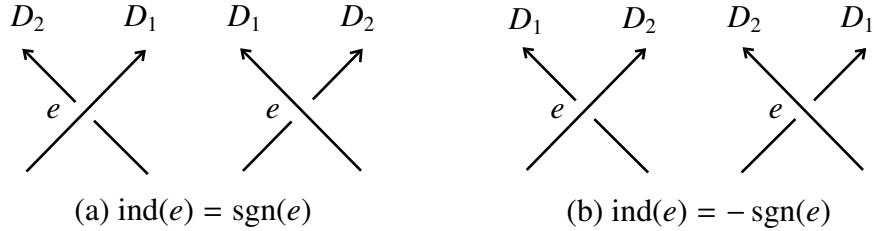
\begin{figure}[htbp]
	\begin{center}
		\tikzset{every picture/.style={line width=1.0pt}} 
		\begin{tikzpicture}[x=0.75pt,y=0.75pt,yscale=-1,xscale=1]
			\draw    (153.66,169.2) -- (94.57,229.58) ;
			\draw   (148.89,170.05) -- (153.68,169.24) -- (152.87,174.1) ; 
			\draw    (128.49,204.03) -- (155.49,230.86) ;
			\draw    (93.57,168.81) -- (119.16,194.7) ;
			\draw   (98.11,169.32) -- (93.29,168.71) -- (94.29,173.54) ;
			\draw    (404.66,168.4) -- (345.57,228.78) ;
			\draw   (399.89,169.25) -- (404.68,168.44) -- (403.87,173.3) ;
			\draw    (379.49,203.23) -- (406.49,230.06) ;
			\draw    (344.57,168.01) -- (370.16,193.9) ;
			\draw   (349.11,168.52) -- (344.29,167.91) -- (345.29,172.74) ;
			\draw    (262.31,168.86) -- (238.26,193.4) ;
			\draw    (203.1,168.93) -- (262.43,229.17) ;
			\draw   (256.84,170.13) -- (262.56,168.72) -- (262.03,174.68) ;
			\draw    (227.59,204.4) -- (203.22,229.24) ;
			\draw   (207.41,169.18) -- (202.6,168.49) -- (203.53,173.33) ;
			\draw    (455.2,168.4) -- (514.29,228.78) ;
			\draw   (510.29,169.25) -- (515.08,168.44) -- (514.27,173.3) ;
			\draw    (480.37,203.23) -- (453.37,230.06) ;
			\draw    (515.29,168.01) -- (489.7,193.9) ;
			\draw   (459.51,168.52) -- (454.69,167.91) -- (455.69,172.74) ;
			\draw (334.6,145.8) node [anchor=north west][inner sep=0.75pt]    {$D_{1}$};
			\draw (391.47,145.65) node [anchor=north west][inner sep=0.75pt]    {$D_{2}$};
			\draw (444.2,145) node [anchor=north west][inner sep=0.75pt]    {$D_{2}$};
			\draw (501.07,144.85) node [anchor=north west][inner sep=0.75pt]    {$D_{1}$};
			\draw (193.4,144.6) node [anchor=north west][inner sep=0.75pt]    {$D_{1}$};
			\draw (250.27,144.45) node [anchor=north west][inner sep=0.75pt]    {$D_{2}$};
			\draw (84.6,144.6) node [anchor=north west][inner sep=0.75pt]    {$D_{2}$};
			\draw (141.47,144.45) node [anchor=north west][inner sep=0.75pt]    {$D_{1}$};
			\draw (103.02,194) node [anchor=north west][inner sep=0.75pt]    {$e$};
			\draw (213.02,194) node [anchor=north west][inner sep=0.75pt]    {$e$};
			\draw (355.82,194) node [anchor=north west][inner sep=0.75pt]    {$e$};
			\draw (465.44,194) node [anchor=north west][inner sep=0.75pt]    {$e$};
			\draw (117,241.2) node [anchor=north west][inner sep=0.75pt]    {(a) $\operatorname{ind}(e)=\operatorname{sgn}(e)$};
			\draw (363.8,241.8) node [anchor=north west][inner sep=0.75pt]    {(b) $\operatorname{ind}(e)=-\operatorname{sgn}(e)$};
		\end{tikzpicture}
		\caption{The $\operatorname{ind}(e)$ for $e \in D_1 \cap D_2$.} \label{fig511}
	\end{center}
\end{figure}

Let $D$ be an oriented knotoid diagram. Denote by $C(D)$ the set of classical crossings of $D$. Using the parity of crossings, we split $C(D)$ into two subsets: the set $O(D)$ of odd crossings and the set $E(D)$ of even crossings. Clearly, $C(D) = O(D) \cup E(D)$ is a disjoint union. Moreover, both sets $O(D)$ and $E(D)$ are invariant under reversing the orientation of $D$.

In \cite[Thm. 2.6]{Lee}, Lee presented an invariant of virtual knots. We extend its parity-enhanced form to virtual knotoids. 

\begin{definition} \label{def2.1}
	{\rm 
		With the above notation, \textit{the two-variable parity polynomial} $\mathbf{P}_D(x,y) \in \mathbb Z[x^{\pm 1}, y^{\pm 1}]$ of $D$ is defined by	
		$$
		\mathbf{P}_D(x,y)=\sum_{c \in E(D)} \operatorname{sgn}(c) i(c) x^{W_D(c)} + \sum_{c \in O(D)} \operatorname{sgn}(c) i(c) y^{W_D(c)}. 
		$$
	}
\end{definition}

\begin{theorem} \label{th2.1}
	Let $K$ be a virtual knotoid with diagram $D$. Then $\mathbf{P}_D(x,y)$ is an invariant of $K$.
\end{theorem}

\begin{proof}
	Let $D'$ be an oriented diagram of $K$ obtained from $D$ by applying one $M$-move, where $M$ belongs to the set of generalized Reidemeister moves. We consider the following moves.
	
	\noindent
	\textbf{Case (i):} $M =\Omega_1$. Assume that $c^*$ is the new classical crossing generated by the $\Omega_1$-move and $C(D')=C(D)\cup \{c^*\}$. Let $\alpha$ be an arc of $D$ with its integer labeling $\ell(\alpha)$, and  $\alpha_1, \alpha_2, \alpha_3$ be three arcs in $D'$ corresponding to $\alpha$. Then, from Fig.~\ref{fig1122}, the integer labels are  $\ell(\alpha_1) = \ell(\alpha)$, $\ell(\alpha_2) = \ell(\alpha)+1$, $\ell(\alpha_3) = \ell(\alpha)$. For any other arc $\gamma$ of $D$, denote by $\gamma'$ its corresponding arc in $D'$, then $\ell(\gamma') =  \ell(\gamma)$. Moreover, $W_{D'}(c^*)=0$, regardless of whether $\operatorname{sgn}(c^*)=1$ or $\operatorname{sgn}(c^*)=-1$. Obviously, $I(c^*)=0$ and $i(c^*)=0$. The set of odd crossings remains unchanged, while the new crossing $c^*$ is even. Thus, $O(D') = O(D)$ and $E(D') = E(D) \cup \{c^*\}$. 
	\begin{figure}[!ht]
		\begin{center}
			\tikzset{every picture/.style={line width=1.0pt}}  
			\begin{tikzpicture}[x=0.75pt,y=0.75pt,yscale=-1,xscale=1]	
				\draw    (188.51,242.37) .. controls (188.45,201.75) and (124.57,201.96) .. (124.83,242.58) ;
				\draw    (358.43,225.98) .. controls (303.62,180.58) and (471.22,202.41) .. (331.55,244.72) ;
				\draw    (371.6,236.78) .. controls (379.23,243.58) and (388.79,248.5) .. (400.94,250.73) ;
				\draw    (549.39,227.26) .. controls (604.51,182.22) and (436.77,202.94) .. (576.15,246.17) ;
				\draw    (536.15,237.97) .. controls (528.47,244.72) and (518.88,249.58) .. (506.72,251.72) ;
				\draw    (236.26,226.03) -- (286.48,225.69) ;
				\draw [shift={(288.48,225.68)}, rotate = 179.62] [color={rgb, 255:red, 0; green, 0; blue, 0 }  ][line width=0.75]    (10.93,-3.29) .. controls (6.95,-1.4) and (3.31,-0.3) .. (0,0) .. controls (3.31,0.3) and (6.95,1.4) .. (10.93,3.29)   ;
				\draw [shift={(234.26,226.04)}, rotate = 359.62] [color={rgb, 255:red, 0; green, 0; blue, 0 }  ][line width=0.75]    (10.93,-3.29) .. controls (6.95,-1.4) and (3.31,-0.3) .. (0,0) .. controls (3.31,0.3) and (6.95,1.4) .. (10.93,3.29)   ;
				\draw   (397.09,246.85) -- (401.14,250.66) -- (396,252.73) ;
				\draw   (572.71,241.87) -- (576.13,246.26) -- (570.73,247.51) ;
				\draw   (191.47,238.41) -- (188.4,243.05) -- (185.49,238.33) ;
				\draw (448.33,210.67) node [anchor=north west][inner sep=0.75pt]   [align=left] {or};
				\draw (353.67,234) node [anchor=north west][inner sep=0.75pt]    [font=\small]  {$c^{*}$};
				\draw (534.33,236) node [anchor=north west][inner sep=0.75pt]    [font=\small]  {$c^{*}$};
				\draw (107,236.67) node [anchor=north west][inner sep=0.75pt]    {$\alpha $};
				\draw (311.33,243) node [anchor=north west][inner sep=0.75pt]    {$\alpha _{1}$};
				\draw (355.33,188) node [anchor=north west][inner sep=0.75pt]    {$\alpha _{2}$};
				\draw (579,246) node [anchor=north west][inner sep=0.75pt]    {$\alpha _{3}$};
				\draw (489,246) node [anchor=north west][inner sep=0.75pt]    {$\alpha _{1}$};
				\draw (529,188.67) node [anchor=north west][inner sep=0.75pt]    {$\alpha _{2}$};
				\draw (402.67,248) node [anchor=north west][inner sep=0.75pt]    {$\alpha _{3}$};
			\end{tikzpicture}
			\caption{$\Omega_1$-move.} \label{fig1122}
		\end{center}
	\end{figure}
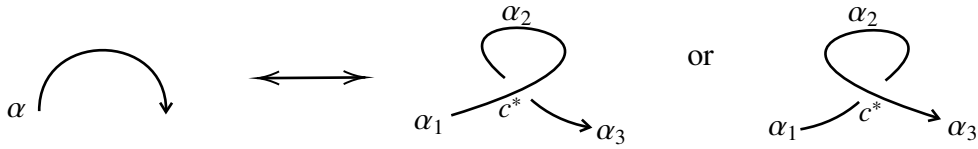
	
	Let $c\in C(D)$ be any classical crossing of $D$, and $c'\in C(D')$ the classical crossing of $D'$ corresponding to $c$. It is clear that $\operatorname{sgn}(c)=\operatorname{sgn}(c')$, and $W_D(c)=W_{D'}(c')$. Then $i(c^*)=0$ implies 
	$$
	\mathbf{P}_{D'}(x,y) - \mathbf{P}_D(x,y) = \operatorname{sgn}(c^*)i(c^*)x^0 =0.
	$$
	
	\noindent
	\textbf{Case (ii):} $M =\Omega_2$. Assume that $c_1$, $c_2$ are the new classical crossings generated by the $\Omega_2$-move and $C(D')=C(D)\cup \{c_1,c_2\}$. Let $\alpha$ and $\beta$ be arcs in $D$ with  integer labels $\ell(\alpha)$ and $\ell(\beta)$, respectively.  Denote the corresponding six arcs in $D'$ by $\alpha_1, \alpha_2, \alpha_3$ and $\beta_1, \beta_2, \beta_3$.  Then, from Fig.~\ref{fig1133} we get the integer labels $\ell(\alpha_1) = \ell(\alpha)$, $\ell(\alpha_2) = \ell(\alpha) - 1$, $\ell(\alpha_3) = \ell(\alpha)$ and $\ell(\beta_1) = \ell(\beta)$, $\ell(\beta_2) = \ell(\beta) + 1$, $\ell(\beta_3) = \ell(\beta)$. 
	For any other arc $\gamma$ of $D$, denote by $\gamma'$ its corresponding arc in $D'$, we have $\ell(\gamma') = \ell(\gamma)$. Moreover, $W_{D'}(c_1)=W_{D'}(c_2)=\ell(\beta)-\ell(\alpha)+1$. Since $I(c_1)=I(c_2)$, see Fig.~\ref{fig2233}, the crossings $c_1$ and $c_2$ have the same parity, that is, either $c_1, c_2 \in O(D')$ or $c_1, c_2 \in E(D')$. Equivalently, $\bigl(O(D'), E(D')\bigr) = \bigl(O(D) \cup \{c_1, c_2\}, E(D)\bigr)$ or $\bigl(O(D'), E(D')\bigr) = \bigl(O(D), E(D) \cup \{c_1, c_2\}\bigr)$. Moreover, $\operatorname{sgn}(c_1)=- \operatorname{sgn}(c_2)$ and $i(c_1)=i(c_2)$.
	
	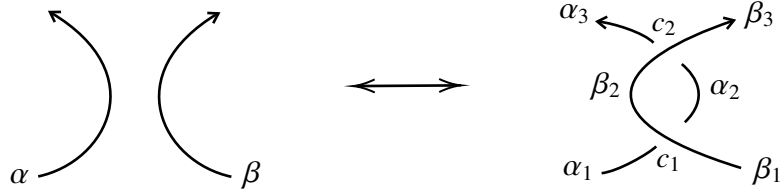
\begin{figure}[!ht]
		\begin{center}
			\tikzset{every picture/.style={line width=1.0pt}}  
			\begin{tikzpicture}[x=0.75pt,y=0.75pt,yscale=-1,xscale=1]		
				\draw    (441.35,405.47) .. controls (357.73,379.88) and (379.24,351.14) .. (437.51,331.5) ;
				\draw    (371.37,408.07) .. controls (379.43,407.31) and (395.05,398.27) .. (398.76,394.64) ;
				\draw    (411.59,353.32) .. controls (419.23,359.66) and (425.32,370.81) .. (413.12,382.51) ;
				\draw    (370.19,332.07) .. controls (377.2,332.81) and (392.49,337.49) .. (397.42,342.64) ;
				\draw    (88.68,409.73) .. controls (123.38,402.34) and (146.12,357.94) .. (95.2,327.59) ;
				\draw    (248.59,364.36) -- (298.81,364.03) ;
				\draw [shift={(300.81,364.01)}, rotate = 179.62] [color={rgb, 255:red, 0; green, 0; blue, 0 }  ][line width=0.75]    (10.93,-3.29) .. controls (6.95,-1.4) and (3.31,-0.3) .. (0,0) .. controls (3.31,0.3) and (6.95,1.4) .. (10.93,3.29)   ;
				\draw [shift={(246.59,364.37)}, rotate = 359.62] [color={rgb, 255:red, 0; green, 0; blue, 0 }  ][line width=0.75]    (10.93,-3.29) .. controls (6.95,-1.4) and (3.31,-0.3) .. (0,0) .. controls (3.31,0.3) and (6.95,1.4) .. (10.93,3.29)   ;
				\draw    (185.62,409.88) .. controls (150.92,402.49) and (128.18,358.09) .. (179.1,327.74) ;
				\draw   (97.03,332.24) -- (94.62,327.22) -- (100.17,327.15) ;
				\draw   (432.32,330.06) -- (437.72,331.41) -- (434.23,335.72) ;
				\draw   (373.54,335.49) -- (369.22,331.99) -- (374.2,329.55) ;
				\draw   (173.53,327.66) -- (179.09,327.93) -- (176.52,332.84) ;
				\draw (397.23,396.68) node [anchor=north west][inner sep=0.75pt]  [font=\small,rotate=-0.11]  {$c_{1}$};
				\draw (394.99,328.57) node [anchor=north west][inner sep=0.75pt]  [font=\small,rotate=-0.11]  {$c_{2}$};
				\draw (73.04,404.27) node [anchor=north west][inner sep=0.75pt]  [font=\normalsize]  {$\alpha $};
				\draw (190.33,400.43) node [anchor=north west][inner sep=0.75pt]  [font=\normalsize]  {$\beta $};
				\draw (424.15,359.45) node [anchor=north west][inner sep=0.75pt]  [font=\normalsize,rotate=-0.11]  {$\alpha _{2}$};
				\draw (348.52,321.36) node [anchor=north west][inner sep=0.75pt]  [font=\normalsize,rotate=-0.11]  {$\alpha _{3}$};
				\draw (350.48,400.47) node [anchor=north west][inner sep=0.75pt]  [font=\normalsize,rotate=-0.11]  {$\alpha _{1}$};
				\draw (442.57,319.89) node [anchor=north west][inner sep=0.75pt]  [font=\normalsize,rotate=-0.11]  {$\beta _{3}$};
				\draw (365.31,355.5) node [anchor=north west][inner sep=0.75pt]  [font=\normalsize,rotate=-0.11]  {$\beta _{2}$};
				\draw (446.75,398.14) node [anchor=north west][inner sep=0.75pt]  [font=\normalsize,rotate=-0.11]  {$\beta _{1}$};	
			\end{tikzpicture}
			\caption{$\Omega_2$-move.} \label{fig1133}
		\end{center}
	\end{figure}
	
	\begin{figure}[!ht]
		\begin{center}
			\tikzset{every picture/.style={line width=1.0pt}}  
			\begin{tikzpicture}[x=0.75pt,y=0.75pt,yscale=-1,xscale=1]	
				\draw    (419.35,125.79) .. controls (320.03,95.22) and (420.88,79.93) .. (348.19,52.4) ;
				\draw    (349.37,128.39) .. controls (357.43,127.63) and (373.05,118.6) .. (376.76,114.96) ;
				\draw    (389.59,73.65) .. controls (391.45,85.08) and (403.32,91.13) .. (391.12,102.84) ;
				\draw    (389.59,73.65) .. controls (388.31,62.79) and (402.3,54.3) .. (415.86,51.65) ;
				\draw   (410.63,49.87) -- (415.92,51.58) -- (412.15,55.65) ;
				\draw   (351.14,56.68) -- (347.82,52.21) -- (353.25,51.09) ;
				\draw    (203.37,128.39) .. controls (265.32,119.77) and (182.58,80.33) .. (269.51,51.82) ;
				\draw    (245.12,102.84) .. controls (241.7,125.63) and (255.23,124.66) .. (274.43,127.77) ;
				\draw    (243.59,73.65) .. controls (254.03,85.86) and (246.3,95.16) .. (245.12,102.84) ;
				\draw    (202.19,52.4) .. controls (209.2,53.14) and (224.49,57.81) .. (229.42,62.97) ;
				\draw   (264.32,50.39) -- (269.72,51.74) -- (266.23,56.05) ;
				\draw   (205.54,55.82) -- (201.22,52.32) -- (206.2,49.88) ;
				\draw (251,101) node [anchor=north west][inner sep=0.75pt]  [font=\small]  {$D_{1}$};
				\draw (197.27,100.65) node [anchor=north west][inner sep=0.75pt]  [font=\small]  {$D_{2}$};
				\draw (349,64) node [anchor=north west][inner sep=0.75pt] [font=\small]   {$D_{1}$};
				\draw (395.27,64.65) node [anchor=north west][inner sep=0.75pt]  [font=\small]  {$D_{2}$};
			\end{tikzpicture}
			\caption{The diagrams obtained by orientation-preserving smoothing at $c_1$ and $c_2$.} \label{fig2233}
		\end{center}
	\end{figure}
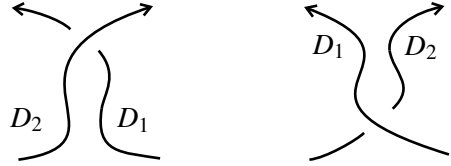
	
	Let $c\in C(D)$ be any classical crossing of $D$, and $c'\in C(D')$ the classical crossing of $D'$ corresponding to $c$. It is clear that $\operatorname{sgn}(c)=\operatorname{sgn}(c')$ and $W_D(c)=W_{D'}(c')$. Then
	\begin{flalign*}
		&\mathbf{P}_{D'}(x,y)-\mathbf{P}_D(x,y) = \\
		&\begin{cases}
			\operatorname{sgn}(c_1)i(c_1)x^{\ell(\beta)-\ell(\alpha)+1} + \operatorname{sgn}(c_2)i(c_2)x^{\ell(\beta)-\ell(\alpha)+1} = 0, & \text{if } c_1,c_2\in E(D),\\[6pt]
			\operatorname{sgn}(c_1)i(c_1)y^{\ell(\beta)-\ell(\alpha)+1} + \operatorname{sgn}(c_2)i(c_2)y^{\ell(\beta)-\ell(\alpha)+1} = 0, & \text{if } c_1,c_2\in O(D).
		\end{cases}
	\end{flalign*}
	The case of $\Omega_2$ when $\beta$ has the opposite orientation can be shown in a similar way.
	
	\noindent
	\textbf{Case (iii):} $M =\Omega_3$. Assume that $c_1, c_2,c_3\in D$ are the classical crossings involved in $\Omega_3$-move. Denote by $c'_1, c'_2,c'_3$ the corresponding crossings in $D'$. Let $\alpha_1$, $\alpha_2$, $\alpha_3$, $\beta_1$, $\beta_2$, $\beta_3$, $\delta_1$, $\delta_2$, and $\delta_3$ be nine arcs of $D$, participating in the $\Omega_3$-move. Let $\alpha'_1$, $\alpha'_2$, $\alpha'_3$, $\beta'_1$, $\beta'_2$, $\beta'_3$, $\delta'_1$, $\delta'_2$, and $\delta'_3$ be the corresponding arcs in $D'$. For any other arc $\gamma$ of $D$, denote by $\gamma'$ its corresponding arc in $D'$. It is easy to see that  $\ell(\gamma') = \ell(\gamma)$. 
	\begin{figure}[!ht]
		\begin{center}
			\tikzset{every picture/.style={line width=1.0pt}}  
			\begin{tikzpicture}[x=0.75pt,y=0.75pt,yscale=-1,xscale=1]	
				\draw    (153.24,488.05) -- (86.62,549.38) ;
				\draw    (138.78,513.53) -- (177.4,548.59) ;
				\draw    (110.25,487.34) -- (125.17,501) ;
				\draw    (86.04,465.34) -- (100.96,478.79) ;
				\draw    (162.54,479) -- (177.47,465.34) ;
				\draw    (83.04,499.96) .. controls (123.04,469.96) and (150.45,470.19) .. (183.04,499.96) ;
				\draw    (397.84,527.66) -- (464.46,466.34) ; 
				\draw    (412.3,502.19) -- (373.68,467.13) ;
				\draw    (440.83,528.38) -- (425.91,514.72) ;
				\draw    (465.04,550.37) -- (450.12,536.93) ;
				\draw    (388.54,536.71) -- (373.61,550.38) ;
				\draw    (468.04,515.76) .. controls (428.04,545.76) and (400.63,545.53) .. (368.04,515.76) ;
				\draw    (252.26,503.93) -- (302.48,503.6) ;
				\draw [shift={(304.48,503.59)}, rotate = 179.62] [color={rgb, 255:red, 0; green, 0; blue, 0 }  ][line width=0.75]    (10.93,-3.29) .. controls (6.95,-1.4) and (3.31,-0.3) .. (0,0) .. controls (3.31,0.3) and (6.95,1.4) .. (10.93,3.29)   ;
				\draw [shift={(250.26,503.95)}, rotate = 359.62] [color={rgb, 255:red, 0; green, 0; blue, 0 }  ][line width=0.75]    (10.93,-3.29) .. controls (6.95,-1.4) and (3.31,-0.3) .. (0,0) .. controls (3.31,0.3) and (6.95,1.4) .. (10.93,3.29)   ;
				\draw   (181.54,494.62) -- (183.06,499.97) -- (177.58,499.09) ;
				\draw   (87.34,470.48) -- (85.63,465.19) -- (91.13,465.86) ;
				\draw   (172.44,465.8) -- (177.94,464.94) -- (176.41,470.27) ;
				\draw   (466.46,520.71) -- (468.39,515.49) -- (462.86,515.93) ;
				\draw   (374.96,472.18) -- (373.25,466.89) -- (378.75,467.56) ;
				\draw   (459.24,467.14) -- (464.7,466.07) -- (463.37,471.45) ;
				\draw (68.13,543.54) node [anchor=north west][inner sep=0.75pt]    {$\alpha _{1}$};
				\draw (143.62,496.71) node [anchor=north west][inner sep=0.75pt]    {$\alpha _{2}$};
				\draw (180.8,455.21) node [anchor=north west][inner sep=0.75pt]    {$\alpha _{3}$};
				\draw (179.47,541.38) node [anchor=north west][inner sep=0.75pt]    {$\beta _{1}$};
				\draw (102.29,494.55) node [anchor=north west][inner sep=0.75pt]    {$\beta _{2}$};
				\draw (64.8,452.71) node [anchor=north west][inner sep=0.75pt]    {$\beta _{3}$};
				\draw (65.8,495.38) node [anchor=north west][inner sep=0.75pt]    {$\delta _{1}$};
				\draw (121.96,460.22) node [anchor=north west][inner sep=0.75pt]    {$\delta _{2}$};
				\draw (184.47,493.71) node [anchor=north west][inner sep=0.75pt]    {$\delta _{3}$};
				\draw (354.51,542.54) node [anchor=north west][inner sep=0.75pt]    {$\alpha'_{1}$};
				\draw (386.66,500.71) node [anchor=north west][inner sep=0.75pt]    {$\alpha'_{2}$};
				\draw (467.84,453.54) node [anchor=north west][inner sep=0.75pt]    {$\alpha'_{3}$};
				\draw (464.84,541.71) node [anchor=north west][inner sep=0.75pt]    {$\beta'_{1}$};
				\draw (432.66,500.55) node [anchor=north west][inner sep=0.75pt]    {$\beta'_{2}$};
				\draw (352.51,455.38) node [anchor=north west][inner sep=0.75pt]    {$\beta'_{3}$};
				\draw (349.17,500.04) node [anchor=north west][inner sep=0.75pt]    {$\delta'_{1}$};
				\draw (405.66,540.22) node [anchor=north west][inner sep=0.75pt]    {$\delta'_{2}$};
				\draw (473.17,502.04) node [anchor=north west][inner sep=0.75pt]    {$\delta'_{3}$};
				\draw (84.33,477.04) node [anchor=north west][inner sep=0.75pt]    [font=\small] {$c_{1}$};
				\draw (124.33,516.5) node [anchor=north west][inner sep=0.75pt]    [font=\small] {$c_{3}$};
				\draw (165.33,476.83) node [anchor=north west][inner sep=0.75pt]    [font=\small] {$c_{2}$};
				\draw (453,523.83) node [anchor=north west][inner sep=0.75pt]    [font=\small] {$c'_{1}$};
				\draw (410.67,480.83) node [anchor=north west][inner sep=0.75pt]    [font=\small] {$c'_{3}$};
				\draw (364.33,523.83) node [anchor=north west][inner sep=0.75pt]    [font=\small] {$c'_{2}$};
			\end{tikzpicture}
			\caption{$\Omega_3$-move.} \label{fig1144}
		\end{center}
	\end{figure}
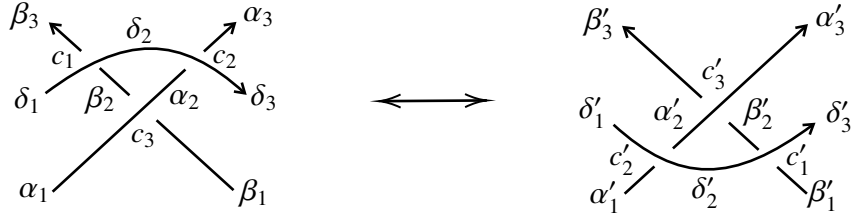
	
	From Fig.~\ref{fig1144}, we have 
	\begin{equation}
		\operatorname{sgn}(c'_i)=\operatorname{sgn}(c_i), \label{2.2}
	\end{equation}
	where $i=1,2,3$. Moreover, for $D$ we have 
	$$
	\begin{gathered}
		\ell(\alpha_2)=\ell(\alpha_1)-1, \quad \ell(\alpha_3)=\ell(\alpha_1), \quad 
		\ell(\beta_2)=\ell(\beta_1)+1,  \cr 
		\ell(\beta_3)=\ell(\beta_1)+2, \quad  \ell(\delta_2)=\ell(\delta_1)-1, \quad
		\ell(\delta_3)=\ell(\delta_1)-2, 
	\end{gathered}
	$$
	and for $D'$ we have
	$$
	\begin{gathered}	
		\ell(\alpha'_1)=\ell(\alpha_1), \quad 
		\ell(\alpha'_2)=\ell(\alpha_1)+1, \quad  
		\ell(\alpha'_3)=\ell(\alpha_1), \cr 
		\ell(\beta'_1)=\ell(\beta_1), \quad
		\ell(\beta'_2)=\ell(\beta_1)+1, \quad   
		\ell(\beta'_3)=\ell(\beta_1)+2, \cr
		\ell(\delta'_1)=\ell(\delta_1), \quad 
		\ell(\delta'_2)=\ell(\delta_1)-1, \quad 
		\ell(\delta'_3)=\ell(\delta_1)-2.
	\end{gathered}
	$$
	Therefore,
	\begin{align}
		W_{D'}(c'_1) &= W_D(c_1) = \ell(\delta_1)-\ell(\beta_1)-2, \label{2.3} \\
		W_{D'}(c'_2) &= W_D(c_2) = \ell(\delta_1)-\ell(\alpha_1)-1, \label{2.4} \\
		W_{D'}(c'_3) &= W_D(c_3) = \ell(\alpha_1)-\ell(\beta_1)-1.  \label{2.5} 
	\end{align}
	
	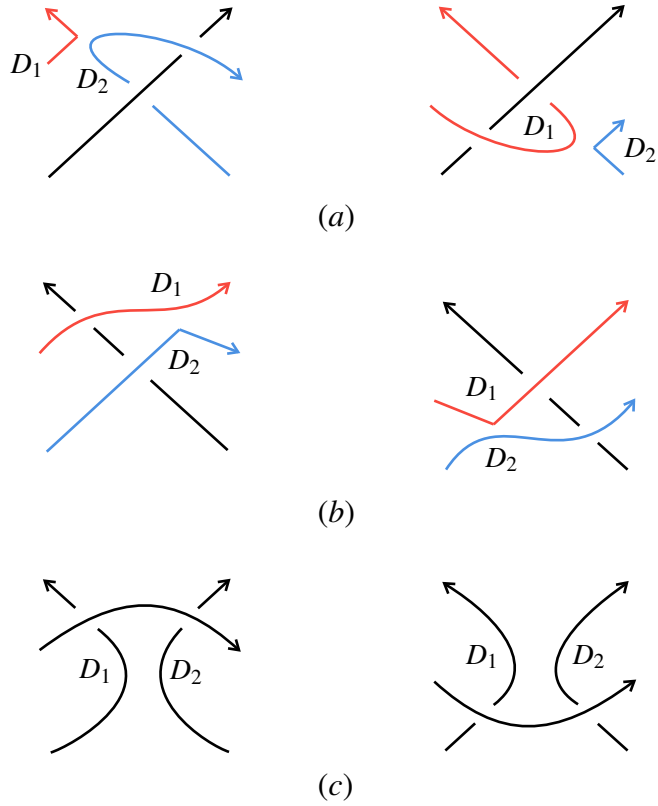
\begin{figure}[!ht]
		\centering
		\tikzset{every picture/.style={line width=1.0pt}}  
		\begin{tikzpicture}[x=0.75pt,y=0.75pt,yscale=-1.0,xscale=1.0]	
			\draw    (114.24,269.68) -- (47.62,331) ;
			\draw [color={rgb, 255:red, 74; green, 144; blue, 226 }  ,draw opacity=1 ]   (99.78,295.15) -- (138.4,330.21) ;
			\draw [color={rgb, 255:red, 248; green, 72; blue, 61 }  ,draw opacity=1 ]   (61.96,260.41) -- (47.03,274.08) ;
			\draw [color={rgb, 255:red, 248; green, 72; blue, 61 }  ,draw opacity=1 ]   (47.04,246.97) -- (61.96,260.41) ;
			\draw    (123.54,260.63) -- (138.47,246.96) ;
			\draw [color={rgb, 255:red, 74; green, 144; blue, 226 }  ,draw opacity=1 ]   (88.04,283.52) .. controls (32.04,250.72) and (111.45,251.82) .. (144.04,281.58) ;
			\draw  [color={rgb, 255:red, 74; green, 144; blue, 226 }  ,draw opacity=1 ] (142.54,276.24) -- (144.06,281.59) -- (138.58,280.71) ;
			\draw  [color={rgb, 255:red, 248; green, 72; blue, 61 }  ,draw opacity=1 ] (48.34,252.1) -- (46.63,246.81) -- (52.13,247.48) ;
			\draw   (133.44,247.42) -- (138.94,246.56) -- (137.41,251.89) ;
			\draw    (268.84,306.95) -- (335.46,245.63) ;
			\draw [color={rgb, 255:red, 248; green, 72; blue, 61 }  ,draw opacity=1 ]   (283.3,281.48) -- (244.68,246.42) ;
			\draw [color={rgb, 255:red, 74; green, 144; blue, 226 }  ,draw opacity=1 ]   (336.04,329.66) -- (321.12,316.22) ;
			\draw    (259.54,316) -- (244.61,329.67) ;
			\draw [color={rgb, 255:red, 248; green, 72; blue, 61 }  ,draw opacity=1 ]   (299.24,293.81) .. controls (339.24,327.15) and (271.63,324.81) .. (239.04,295.05) ;
			\draw  [color={rgb, 255:red, 74; green, 144; blue, 226 }  ,draw opacity=1 ] (334.62,308.17) -- (335.86,302.75) -- (330.44,303.9) ;
			\draw  [color={rgb, 255:red, 248; green, 72; blue, 61 }  ,draw opacity=1 ] (245.96,251.47) -- (244.25,246.17) -- (249.75,246.85) ;
			\draw   (330.24,246.43) -- (335.7,245.36) -- (334.37,250.74) ;
			\draw [color={rgb, 255:red, 74; green, 144; blue, 226 }  ,draw opacity=1 ]   (321.12,316.22) -- (336.04,302.77) ;
			\draw    (46.04,533.7) -- (60.96,547.15) ;
			\draw    (122.54,547.36) -- (137.47,533.7) ;
			\draw    (43.04,568.32) .. controls (83.04,538.32) and (110.45,538.55) .. (143.04,568.32) ;
			\draw    (338.04,618.73) -- (323.12,605.28) ;
			\draw    (261.54,605.07) -- (246.61,618.73) ;
			\draw    (341.04,584.11) .. controls (301.04,614.11) and (273.63,613.88) .. (241.04,584.11) ;
			\draw   (141.54,562.97) -- (143.06,568.33) -- (137.58,567.45) ;
			\draw   (47.34,538.83) -- (45.63,533.54) -- (51.13,534.21) ;
			\draw   (132.44,534.15) -- (137.94,533.29) -- (136.41,538.62) ;
			\draw   (339.46,589.06) -- (341.39,583.84) -- (335.86,584.29) ;
			\draw   (248.08,539.98) -- (246.05,534.79) -- (251.59,535.14) ;
			\draw   (332.24,535.49) -- (337.7,534.43) -- (336.37,539.81) ;
			\draw    (72.25,557.69) .. controls (111.51,588.59) and (57.51,617.05) .. (48.62,619.73) ;
			\draw    (114.45,557.36) .. controls (83.17,589.93) and (129.19,616.71) .. (138.07,619.4) ;
			\draw    (246.68,535.32) .. controls (276.3,553.72) and (293.9,578.12) .. (270.84,595.85) ;
			\draw    (337.33,534.99) .. controls (309.83,553.93) and (287.9,579.32) .. (313.17,595.52) ;
			\draw [color={rgb, 255:red, 74; green, 144; blue, 226 }  ,draw opacity=1 ]   (113.24,407.41) -- (46.62,468.73) ;
			\draw    (98.78,432.88) -- (137.4,467.94) ;
			\draw    (70.25,406.69) -- (85.17,420.36) ;
			\draw    (46.04,384.7) -- (60.96,398.15) ;
			\draw [color={rgb, 255:red, 74; green, 144; blue, 226 }  ,draw opacity=1 ]   (113.24,407.41) -- (142.65,418.74) ;
			\draw [color={rgb, 255:red, 248; green, 72; blue, 61 }  ,draw opacity=1 ]   (43.04,419.32) .. controls (76.99,380.41) and (106.32,412.41) .. (137.47,384.7) ;
			\draw [color={rgb, 255:red, 248; green, 72; blue, 61 }  ,draw opacity=1 ]   (270.84,455.02) -- (337.46,393.7) ;
			\draw    (285.3,429.55) -- (246.68,394.49) ;
			\draw    (313.83,455.73) -- (298.91,442.07) ;
			\draw    (338.04,477.73) -- (323.12,464.28) ;
			\draw [color={rgb, 255:red, 74; green, 144; blue, 226 }  ,draw opacity=1 ]   (341.37,443.11) .. controls (305.72,487.03) and (272.72,438.37) .. (246.95,477.73) ;
			\draw  [color={rgb, 255:red, 74; green, 144; blue, 226 }  ,draw opacity=1 ] (139.57,414.08) -- (142.63,418.73) -- (137.14,419.54) ;
			\draw   (47.34,389.83) -- (45.63,384.54) -- (51.13,385.21) ;
			\draw  [color={rgb, 255:red, 248; green, 72; blue, 61 }  ,draw opacity=1 ] (132.44,385.15) -- (137.94,384.29) -- (136.41,389.62) ;
			\draw  [color={rgb, 255:red, 74; green, 144; blue, 226 }  ,draw opacity=1 ] (341,448.55) -- (341.18,442.99) -- (336.08,445.17) ;
			\draw   (247.96,399.53) -- (246.25,394.24) -- (251.75,394.91) ;
			\draw  [color={rgb, 255:red, 248; green, 72; blue, 61 }  ,draw opacity=1 ] (332.24,394.49) -- (337.7,393.43) -- (336.37,398.81) ;
			\draw [color={rgb, 255:red, 248; green, 72; blue, 61 }  ,draw opacity=1 ]   (270.84,455.02) -- (241.04,443.11) ;
			\draw (181,342.91) node [anchor=north west][inner sep=0.75pt]    {$(a)$};
			\draw (181,628.51) node [anchor=north west][inner sep=0.75pt]    {$(c)$};
			\draw (26.6,267.35) node [anchor=north west][inner sep=0.75pt]  [font=\small]  {$D_{1}$};
			\draw (59.67,275.2) node [anchor=north west][inner sep=0.75pt]  [font=\small]  {$D_{2}$};
			\draw (284,298.35) node [anchor=north west][inner sep=0.75pt]  [font=\small]  {$D_{1}$};
			\draw (334.47,309.8) node [anchor=north west][inner sep=0.75pt]  [font=\small]  {$D_{2}$};
			\draw (61,571.35) node [anchor=north west][inner sep=0.75pt]  [font=\small]  {$D_{1}$};
			\draw (255.08,563.9) node [anchor=north west][inner sep=0.75pt]  [font=\small]  {$D_{1}$};
			\draw (106.47,571.8) node [anchor=north west][inner sep=0.75pt]  [font=\small]  {$D_{2}$};
			\draw (308.47,563.9) node [anchor=north west][inner sep=0.75pt]  [font=\small]  {$D_{2}$};
			\draw (181,491.51) node [anchor=north west][inner sep=0.75pt]    {$(b)$};
			\draw (97,377.35) node [anchor=north west][inner sep=0.75pt]  [font=\small]  {$D_{1}$};
			\draw (255,429.35) node [anchor=north west][inner sep=0.75pt]  [font=\small]  {$D_{1}$};
			\draw (105.67,416.2) node [anchor=north west][inner sep=0.75pt]  [font=\small]  {$D_{2}$};
			\draw (264.67,465.2) node [anchor=north west][inner sep=0.75pt]  [font=\small]  {$D_{2}$};	
		\end{tikzpicture}
		\caption{(a) Orientation-preserving smoothings at crossings $c_1$ and $c'_1$. (b) Orientation-preserving smoothings at crossings $c_2$ and $c'_2$.(c) Orientation-preserving smoothings at crossings $c_3$ and $c'_3$.} \label{fig2244}
	\end{figure}
	
	Let us discuss the parity of crossings $c_i$ and $c'_i$, where $i=1, 2, 3$. For $c_1$ and $c'_1$, consider the diagrams obtained by applying the orientation-preserving smoothings at $c_1$ and $c'_1$, respectively, see Fig.~\ref{fig2244}(a). Then there are two cases:
	\begin{itemize}
		\item[{(1)}] If the red part belongs to one component, say $D_1$, and the blue and black parts belong to the other component , say $D_2$, then $I(c'_1) - I(c_1) = 2$;
		\item[{(2)}] If the blue part belongs to one component, say $D_2$, and the red and black parts belong to the other, say $D_1$, then $I(c'_1) - I(c_1) = -2$.
	\end{itemize}
	
	For $c_2$ and $c'_2$, consider the diagrams obtained by applying the orientation-preserving smoothings at $c_2$ and $c'_2$, respectively, see Fig.~\ref{fig2244}(b). Then there are two cases:
	\begin{itemize}
		\item[{(1)}] If the red part belongs to one component, say $D_1$, and the blue and black parts belong to the other component , say $D_2$, then $I(c'_2) - I(c_2) = 0$;
		\item[{(2)}] If the blue part belongs to one component, say $D_2$, and the red and black parts belong to the other, say $D_1$, then $I(c'_2) - I(c_2) = 0$.
	\end{itemize}
	
	For $c_3$ and $c'_3$, from Fig.~\ref{fig2244}(c), one can easily obtain $I(c'_3) = I(c_3)$. Therefore, crossings $c'_i$ and $c_i$ have the same parity, for $i=1, 2, 3$. Moreover, $i(c'_i) = i(c_i)$ for $i=1, 2, 3$. Combining formulas \eqref{2.2}--\eqref{2.5}, we obtain that the contribution of $c'_i$ in $\mathbf{P}_{D'}(x,y)$ is the same as that of $c_i$ in $\mathbf{P}_{D}(x,y)$, where $i=1, 2, 3$. 
	
	Let $c\in C(D)$ be any classical crossing of $D$ not involved in $\Omega_3$-move, and $c'\in C(D')$ the classical crossing of $D'$ corresponding to $c$. It is clear that $\operatorname{sgn}(c)=\operatorname{sgn}(c')$ and $W_D(c)=W_{D'}(c')$. Therefore, the contribution of $c$ in $\mathbf{P}_D(x,y)$ is the same as that of $c'$ in $\mathbf{P}_{D'}(x,y)$. Hence, $\mathbf{P}_{D'}(x,y)=\mathbf{P}_D(x,y).$ The oriented $\Omega_3$-move has eight possible orientation configurations. Up to simultaneous reversal of all orientations, these reduce to four distinct cases. The remaining three cases can be proved individually by the same method.
	
	\noindent
	\textbf{Case (iv):} $M = \Omega^{m}_3$. The classical crossings involved in $\Omega^{m}_3$ remain unchanged, so this move has no effect on $\mathbf{P}_D(x,y)$. Hence $\mathbf{P}_{D'}(x,y) = \mathbf{P}_D(x,y)$.
	
	\noindent
	\textbf{Case (v):} $M = \Omega^{v}_i$ $(i=1,2,3)$ or $M = \Omega_v$. These moves do not involve classical crossings, so they have no effect on the definition of $\mathbf{P}_D(x,y)$. Hence, $\mathbf{P}_{D'}(x,y)=\mathbf{P}_D(x,y)$.
	
	It follows from the above cases that $\mathbf{P}_D(x,y)$ is an invariant of an oriented virtual knotoid $K$.
\end{proof}
Thus, for a virtual knotoid $K$ with diagram $D$, we may write $\mathbf{P}_K(x,y)=\mathbf{P}_D(x,y)$. 

\begin{remark}
	{\rm
		The invariant $P_K(x,y)$ is defined using the sign of classical crossing together with the parity of crossing, and the computation of $P_K(x,y)$ requires a single traversal of all crossings in a virtual knotoid diagram. Therefore, the overall computational complexity is linear in the number of classical crossings.
	}	
\end{remark}

\subsection{An example of calculation of $\mathbf{P}_K(x,y)$}

The following example shows that the invariant $\mathbf{P}(x,y)$ is non-trivial. 
\begin{example} \label{exp2.1}{\rm
		Consider the virtual knotoid $K_+$, whose diagram $D_+$ is presented in Fig.~\ref{fig1119} (a).  The diagram $D_+$ contains two virtual crossings and four classical crossings $c_1$, $c_2$, $c_3$ and $c_+$.  We will demonstrate that $\mathbf{P}_{D_+}(x,y) \neq 0$. 
		\begin{figure}[!ht]
			\begin{center}
				\tikzset{every picture/.style={line width=1.0pt}}  
				\begin{tikzpicture}[x=0.75pt,y=0.75pt,yscale=-0.85,xscale=0.85]	
					\draw    (144.72,101.18) .. controls (154.23,102.16) and (157.28,114.59) .. (157.77,118.86) ;
					\draw    (162.91,140.27) .. controls (166.93,157.69) and (164.49,199.46) .. (199.76,175.52) ; 
					\draw    (157.95,134.8) .. controls (105.46,204.41) and (95.68,88.39) .. (133.53,97.89) ; 
					\draw    (145.94,85.38) .. controls (117.92,133.68) and (120.53,171.93) .. (141.94,174.58) ;
					\draw    (141.94,174.58) .. controls (196.24,177.63) and (162.84,126.75) .. (202.88,128.64) ;
					\draw    (217.56,127.69) .. controls (240.24,126.27) and (226.48,84.67) .. (206.01,84.6) ; 
					\draw    (195.76,98.34) .. controls (220.22,121.54) and (218.71,162.91) .. (199.76,175.52) ;
					\draw    (168.76,114.09) .. controls (177.08,99.49) and (191.84,85.55) .. (206.01,84.6) ;
					\draw    (164.18,72.76) .. controls (166.85,73.71) and (178.86,78.45) .. (183.87,85.23) ;
					\draw    (157.95,134.8) .. controls (163.24,125.82) and (166.34,119.12) .. (168.76,114.09) ; 
					\draw   (119.31,155.97) .. controls (119.31,152.32) and (122.27,149.37) .. (125.92,149.37) .. controls (129.57,149.37) and (132.52,152.32) .. (132.52,155.97) .. controls (132.52,159.62) and (129.57,162.58) .. (125.92,162.58) .. controls (122.27,162.58) and (119.31,159.62) .. (119.31,155.97) -- cycle ;
					\draw   (161.31,168.37) .. controls (161.31,164.72) and (164.27,161.77) .. (167.92,161.77) .. controls (171.57,161.77) and (174.52,164.72) .. (174.52,168.37) .. controls (174.52,172.02) and (171.57,174.98) .. (167.92,174.98) .. controls (164.27,174.98) and (161.31,172.02) .. (161.31,168.37) -- cycle ;
					\draw  [fill={rgb, 255:red, 0; green, 0; blue, 0 }  ,fill opacity=1 ] (142.9,85.38) .. controls (142.9,83.7) and (144.26,82.34) .. (145.94,82.34) .. controls (147.62,82.34) and (148.98,83.7) .. (148.98,85.38) .. controls (148.98,87.06) and (147.62,88.42) .. (145.94,88.42) .. controls (144.26,88.42) and (142.9,87.06) .. (142.9,85.38) -- cycle ;
					\draw  [fill={rgb, 255:red, 0; green, 0; blue, 0 }  ,fill opacity=1 ] (161.14,72.76) .. controls (161.14,71.09) and (162.5,69.73) .. (164.18,69.73) .. controls (165.85,69.73) and (167.21,71.09) .. (167.21,72.76) .. controls (167.21,74.44) and (165.85,75.8) .. (164.18,75.8) .. controls (162.5,75.8) and (161.14,74.44) .. (161.14,72.76) -- cycle ;
					\draw   (208.43,108.41) -- (207.87,114.79) -- (202.18,111.86) ;
					\draw    (355.72,97.11) .. controls (365.23,98.09) and (368.28,110.52) .. (368.77,114.79) ;
					\draw    (373.91,136.2) .. controls (377.93,153.62) and (375.49,195.39) .. (410.76,171.45) ;
					\draw    (368.95,130.73) .. controls (316.46,200.34) and (306.68,84.32) .. (344.53,93.82) ; 
					\draw    (356.94,81.31) .. controls (328.92,129.61) and (331.53,167.86) .. (352.94,170.51) ;
					\draw    (352.94,170.51) .. controls (407.24,173.56) and (373.84,122.68) .. (413.88,124.57) ;
					\draw    (428.56,123.62) .. controls (451.24,122.2) and (437.48,80.6) .. (417.01,80.53) ;
					\draw    (406.76,94.27) .. controls (431.22,117.47) and (429.71,158.84) .. (410.76,171.45) ; 
					\draw    (379.76,110.02) .. controls (388.08,95.42) and (402.84,81.48) .. (417.01,80.53) ;
					\draw    (375.18,68.69) .. controls (377.85,69.64) and (389.86,74.38) .. (394.87,81.16) ;
					\draw    (368.95,130.73) .. controls (374.24,121.75) and (377.34,115.05) .. (379.76,110.02) ;
					\draw   (330.31,151.9) .. controls (330.31,148.25) and (333.27,145.3) .. (336.92,145.3) .. controls (340.57,145.3) and (343.52,148.25) .. (343.52,151.9) .. controls (343.52,155.55) and (340.57,158.51) .. (336.92,158.51) .. controls (333.27,158.51) and (330.31,155.55) .. (330.31,151.9) -- cycle ;
					\draw   (372.31,164.3) .. controls (372.31,160.65) and (375.27,157.7) .. (378.92,157.7) .. controls (382.57,157.7) and (385.52,160.65) .. (385.52,164.3) .. controls (385.52,167.95) and (382.57,170.91) .. (378.92,170.91) .. controls (375.27,170.91) and (372.31,167.95) .. (372.31,164.3) -- cycle ;
					\draw  [fill={rgb, 255:red, 0; green, 0; blue, 0 }  ,fill opacity=1 ] (353.9,81.31) .. controls (353.9,79.63) and (355.26,78.27) .. (356.94,78.27) .. controls (358.62,78.27) and (359.98,79.63) .. (359.98,81.31) .. controls (359.98,82.99) and (358.62,84.35) .. (356.94,84.35) .. controls (355.26,84.35) and (353.9,82.99) .. (353.9,81.31) -- cycle ;
					\draw  [fill={rgb, 255:red, 0; green, 0; blue, 0 }  ,fill opacity=1 ] (372.14,68.69) .. controls (372.14,67.02) and (373.5,65.66) .. (375.18,65.66) .. controls (376.85,65.66) and (378.21,67.02) .. (378.21,68.69) .. controls (378.21,70.37) and (376.85,71.73) .. (375.18,71.73) .. controls (373.5,71.73) and (372.14,70.37) .. (372.14,68.69) -- cycle ;
					\draw   (419.43,104.34) -- (418.87,110.72) -- (413.18,107.79) ;
					\draw (164.01,122.84) node [anchor=north west][inner sep=0.75pt]    {$c_+$};
					\draw (155.73,187.18) node [anchor=north west][inner sep=0.75pt]    {(a)};
					\draw (184,73.07) node [anchor=north west][inner sep=0.75pt]    {$c_{1}$};
					\draw (213.5,129.07) node [anchor=north west][inner sep=0.75pt]    {$c_{2}$};
					\draw (125.5,83.07) node [anchor=north west][inner sep=0.75pt]    {$c_{3}$};
					\draw (372.73,187.18) node [anchor=north west][inner sep=0.75pt]    {(b)};
					\draw (384.88,59.5) node [anchor=north west][inner sep=0.75pt]    {$0$};
					\draw (437.38,88.79) node [anchor=north west][inner sep=0.75pt]    {$2$};
					\draw (356.81,102.93) node [anchor=north west][inner sep=0.75pt]    {$1$};
					\draw (391.88,100.64) node [anchor=north west][inner sep=0.75pt]    {$-1$};
					\draw (342.53,76.21) node [anchor=north west][inner sep=0.75pt]    {$0$};
					\draw (341.81,171.07) node [anchor=north west][inner sep=0.75pt]    {$1$};
					\draw (418.67,159.36) node [anchor=north west][inner sep=0.75pt]    {$0$};
					\draw (309.24,123.5) node [anchor=north west][inner sep=0.75pt]    {$2$};
					\draw (373.81,88.21) node [anchor=north west][inner sep=0.75pt]    {$1$};
				\end{tikzpicture}
				\caption{A virtual knotoid diagram $D_+$ and its integer labeling.} \label{fig1119}
			\end{center}
		\end{figure}
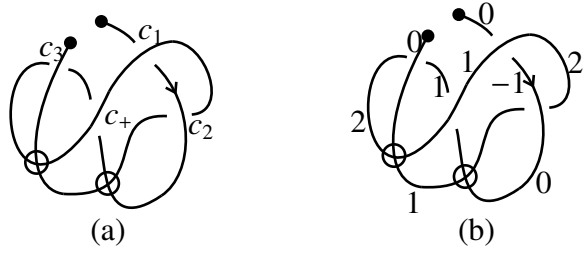  
		
		As can be seen from Fig.~\ref{fig1119}(a), $\operatorname{sgn}(c_1)=\operatorname{sgn}(c_2) =-1$ and $ \operatorname{sgn}(c_3) = \operatorname{sgn}(c_+)= 1$.
		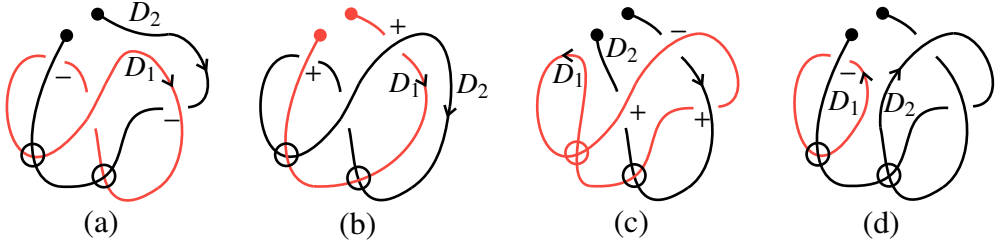
\begin{figure}[!ht]
			\begin{center}
				\tikzset{every picture/.style={line width=1.0pt}}  
				\begin{tikzpicture}[x=0.75pt,y=0.75pt,yscale=-0.85,xscale=0.85]	
					\draw    (204.55,71.94) .. controls (214.06,72.92) and (217.12,85.35) .. (217.61,89.62) ; 
					\draw    (222.74,111.04) .. controls (226.76,128.46) and (224.32,170.22) .. (259.6,146.29) ;
					\draw    (217.78,105.56) .. controls (165.3,175.17) and (155.51,59.15) .. (193.36,68.66) ;
					\draw [color={rgb, 255:red, 248; green, 72; blue, 61 }  ,draw opacity=1 ]   (205.77,56.14) .. controls (177.75,104.44) and (180.36,142.7) .. (201.77,145.34) ;
					\draw [color={rgb, 255:red, 248; green, 72; blue, 61 }  ,draw opacity=1 ]   (201.77,145.34) .. controls (256.07,148.39) and (288.61,91.74) .. (254.27,67.08) ;
					\draw    (282.27,98.08) .. controls (283.27,93.08) and (286.31,55.44) .. (265.84,55.36) ;
					\draw    (282.27,98.08) .. controls (280.27,108.74) and (278.54,133.68) .. (259.6,146.29) ;
					\draw    (228.59,84.85) .. controls (236.92,70.26) and (251.68,56.31) .. (265.84,55.36) ;
					\draw [color={rgb, 255:red, 248; green, 72; blue, 61 }  ,draw opacity=1 ]   (224.01,43.53) .. controls (226.68,44.47) and (238.69,49.21) .. (243.71,55.99) ;
					\draw    (217.78,105.56) .. controls (223.07,96.58) and (226.18,89.88) .. (228.59,84.85) ;
					\draw   (179.15,126.73) .. controls (179.15,123.09) and (182.1,120.13) .. (185.75,120.13) .. controls (189.4,120.13) and (192.36,123.09) .. (192.36,126.73) .. controls (192.36,130.38) and (189.4,133.34) .. (185.75,133.34) .. controls (182.1,133.34) and (179.15,130.38) .. (179.15,126.73) -- cycle ;
					\draw   (221.81,140.47) .. controls (221.81,136.82) and (224.77,133.86) .. (228.42,133.86) .. controls (232.07,133.86) and (235.02,136.82) .. (235.02,140.47) .. controls (235.02,144.12) and (232.07,147.07) .. (228.42,147.07) .. controls (224.77,147.07) and (221.81,144.12) .. (221.81,140.47) -- cycle ;
					\draw  [color={rgb, 255:red, 248; green, 72; blue, 61 }  ,draw opacity=1 ][fill={rgb, 255:red, 248; green, 72; blue, 61 }  ,fill opacity=1 ] (202.74,56.14) .. controls (202.74,54.47) and (204.1,53.11) .. (205.77,53.11) .. controls (207.45,53.11) and (208.81,54.47) .. (208.81,56.14) .. controls (208.81,57.82) and (207.45,59.18) .. (205.77,59.18) .. controls (204.1,59.18) and (202.74,57.82) .. (202.74,56.14) -- cycle ;
					\draw  [color={rgb, 255:red, 248; green, 72; blue, 61 }  ,draw opacity=1 ][fill={rgb, 255:red, 248; green, 72; blue, 61 }  ,fill opacity=1 ] (220.97,43.53) .. controls (220.97,41.85) and (222.33,40.49) .. (224.01,40.49) .. controls (225.69,40.49) and (227.05,41.85) .. (227.05,43.53) .. controls (227.05,45.2) and (225.69,46.56) .. (224.01,46.56) .. controls (222.33,46.56) and (220.97,45.2) .. (220.97,43.53) -- cycle ;
					\draw   (268.26,79.18) -- (267.7,85.56) -- (262.01,82.62) ;
					\draw [color={rgb, 255:red, 248; green, 72; blue, 61 }  ,draw opacity=1 ]   (54.88,72.12) .. controls (64.4,73.11) and (67.45,85.53) .. (67.94,89.8) ;
					\draw [color={rgb, 255:red, 248; green, 72; blue, 61 }  ,draw opacity=1 ]   (73.07,111.22) .. controls (77.1,128.64) and (74.65,170.41) .. (109.93,146.47) ;
					\draw [color={rgb, 255:red, 248; green, 72; blue, 61 }  ,draw opacity=1 ]   (68.12,105.74) .. controls (15.63,175.35) and (5.84,59.34) .. (43.7,68.84) ;
					\draw    (56.11,56.32) .. controls (28.08,104.63) and (30.69,142.88) .. (52.1,145.52) ;
					\draw    (52.1,145.52) .. controls (106.4,148.57) and (73.01,97.69) .. (113.04,99.59) ;
					\draw    (127.72,98.64) .. controls (150.41,97.22) and (136.64,55.62) .. (116.17,55.55) ;
					\draw [color={rgb, 255:red, 248; green, 72; blue, 61 }  ,draw opacity=1 ]   (105.93,69.28) .. controls (130.39,92.48) and (128.87,133.86) .. (109.93,146.47) ;
					\draw [color={rgb, 255:red, 248; green, 72; blue, 61 }  ,draw opacity=1 ]   (78.92,85.03) .. controls (87.25,70.44) and (88.94,59.71) .. (105.93,69.28) ;
					\draw    (74.34,43.71) .. controls (77.01,44.66) and (94.94,59.38) .. (116.17,55.55) ; 
					\draw [color={rgb, 255:red, 248; green, 72; blue, 61 }  ,draw opacity=1 ]   (68.12,105.74) .. controls (73.41,96.76) and (76.51,90.06) .. (78.92,85.03) ;
					\draw   (29.48,126.92) .. controls (29.48,123.27) and (32.44,120.31) .. (36.09,120.31) .. controls (39.73,120.31) and (42.69,123.27) .. (42.69,126.92) .. controls (42.69,130.56) and (39.73,133.52) .. (36.09,133.52) .. controls (32.44,133.52) and (29.48,130.56) .. (29.48,126.92) -- cycle ;
					\draw   (71.48,139.32) .. controls (71.48,135.67) and (74.44,132.71) .. (78.09,132.71) .. controls (81.73,132.71) and (84.69,135.67) .. (84.69,139.32) .. controls (84.69,142.96) and (81.73,145.92) .. (78.09,145.92) .. controls (74.44,145.92) and (71.48,142.96) .. (71.48,139.32) -- cycle ;
					\draw  [fill={rgb, 255:red, 0; green, 0; blue, 0 }  ,fill opacity=1 ] (53.07,56.32) .. controls (53.07,54.65) and (54.43,53.29) .. (56.11,53.29) .. controls (57.78,53.29) and (59.14,54.65) .. (59.14,56.32) .. controls (59.14,58) and (57.78,59.36) .. (56.11,59.36) .. controls (54.43,59.36) and (53.07,58) .. (53.07,56.32) -- cycle ;
					\draw  [fill={rgb, 255:red, 0; green, 0; blue, 0 }  ,fill opacity=1 ] (71.31,43.71) .. controls (71.31,42.03) and (72.67,40.67) .. (74.34,40.67) .. controls (76.02,40.67) and (77.38,42.03) .. (77.38,43.71) .. controls (77.38,45.39) and (76.02,46.75) .. (74.34,46.75) .. controls (72.67,46.75) and (71.31,45.39) .. (71.31,43.71) -- cycle ;
					\draw   (118.6,79.36) -- (118.04,85.74) -- (112.34,82.81) ;
					\draw   (138.93,70.02) -- (138.37,76.41) -- (132.68,73.47) ;
					\draw [color={rgb, 255:red, 248; green, 72; blue, 61 }  ,draw opacity=1 ]   (518.22,71.61) .. controls (527.73,72.59) and (532.64,91.08) .. (525.3,104.74) ;
					\draw    (536.41,110.7) .. controls (540.43,128.12) and (537.99,169.89) .. (573.26,145.95) ;
					\draw [color={rgb, 255:red, 248; green, 72; blue, 61 }  ,draw opacity=1 ]   (525.3,104.74) .. controls (485.64,174.74) and (469.18,58.82) .. (507.03,68.32) ;
					\draw    (519.44,55.81) .. controls (491.42,104.11) and (494.03,142.36) .. (515.44,145.01) ;
					\draw    (515.44,145.01) .. controls (569.74,148.06) and (536.34,97.18) .. (576.38,99.07) ;
					\draw    (591.06,98.12) .. controls (613.74,96.7) and (599.98,55.1) .. (579.51,55.03) ;
					\draw    (569.26,68.77) .. controls (593.72,91.97) and (592.21,133.34) .. (573.26,145.95) ;
					\draw    (542.26,84.52) .. controls (550.58,69.92) and (565.34,55.98) .. (579.51,55.03) ;
					\draw    (537.68,43.19) .. controls (540.35,44.14) and (552.36,48.88) .. (557.37,55.66) ;
					\draw    (536.41,110.7) .. controls (535.3,96.74) and (539.84,89.55) .. (542.26,84.52) ;
					\draw   (492.81,126.4) .. controls (492.81,122.75) and (495.77,119.8) .. (499.42,119.8) .. controls (503.07,119.8) and (506.02,122.75) .. (506.02,126.4) .. controls (506.02,130.05) and (503.07,133.01) .. (499.42,133.01) .. controls (495.77,133.01) and (492.81,130.05) .. (492.81,126.4) -- cycle ;
					\draw   (534.81,138.8) .. controls (534.81,135.15) and (537.77,132.2) .. (541.42,132.2) .. controls (545.07,132.2) and (548.02,135.15) .. (548.02,138.8) .. controls (548.02,142.45) and (545.07,145.41) .. (541.42,145.41) .. controls (537.77,145.41) and (534.81,142.45) .. (534.81,138.8) -- cycle ;
					\draw  [fill={rgb, 255:red, 0; green, 0; blue, 0 }  ,fill opacity=1 ] (516.4,55.81) .. controls (516.4,54.13) and (517.76,52.77) .. (519.44,52.77) .. controls (521.12,52.77) and (522.48,54.13) .. (522.48,55.81) .. controls (522.48,57.49) and (521.12,58.85) .. (519.44,58.85) .. controls (517.76,58.85) and (516.4,57.49) .. (516.4,55.81) -- cycle ;
					\draw  [fill={rgb, 255:red, 0; green, 0; blue, 0 }  ,fill opacity=1 ] (534.64,43.19) .. controls (534.64,41.52) and (536,40.16) .. (537.68,40.16) .. controls (539.35,40.16) and (540.71,41.52) .. (540.71,43.19) .. controls (540.71,44.87) and (539.35,46.23) .. (537.68,46.23) .. controls (536,46.23) and (534.64,44.87) .. (534.64,43.19) -- cycle ;
					\draw    (369.11,55.99) .. controls (368.65,68.31) and (380.45,85.2) .. (380.94,89.47) ;
					\draw    (386.07,110.89) .. controls (390.1,128.31) and (387.65,170.07) .. (422.93,146.13) ;
					\draw [color={rgb, 255:red, 248; green, 72; blue, 61 }  ,draw opacity=1 ]   (381.12,105.41) .. controls (328.63,175.02) and (318.84,59) .. (356.7,68.5) ;
					\draw [color={rgb, 255:red, 248; green, 72; blue, 61 }  ,draw opacity=1 ]   (356.7,68.5) .. controls (374.42,71.77) and (343.69,142.54) .. (365.1,145.19) ;
					\draw [color={rgb, 255:red, 248; green, 72; blue, 61 }  ,draw opacity=1 ]   (365.1,145.19) .. controls (419.4,148.24) and (386.01,97.36) .. (426.04,99.25) ;
					\draw [color={rgb, 255:red, 248; green, 72; blue, 61 }  ,draw opacity=1 ]   (440.72,98.31) .. controls (463.41,96.89) and (449.64,55.29) .. (429.17,55.21) ;
					\draw    (418.93,68.95) .. controls (443.39,92.15) and (441.87,133.53) .. (422.93,146.13) ;
					\draw [color={rgb, 255:red, 248; green, 72; blue, 61 }  ,draw opacity=1 ]   (391.92,84.7) .. controls (400.25,70.11) and (415.01,56.16) .. (429.17,55.21) ;
					\draw    (387.34,43.38) .. controls (390.01,44.32) and (402.02,49.06) .. (407.04,55.84) ;
					\draw [color={rgb, 255:red, 248; green, 72; blue, 61 }  ,draw opacity=1 ]   (381.12,105.41) .. controls (386.41,96.43) and (389.51,89.73) .. (391.92,84.7) ;
					\draw  [color={rgb, 255:red, 248; green, 72; blue, 61 }  ,draw opacity=1 ] (350.48,125.58) .. controls (350.48,121.93) and (353.44,118.98) .. (357.09,118.98) .. controls (360.73,118.98) and (363.69,121.93) .. (363.69,125.58) .. controls (363.69,129.23) and (360.73,132.19) .. (357.09,132.19) .. controls (353.44,132.19) and (350.48,129.23) .. (350.48,125.58) -- cycle ;
					\draw   (384.48,138.98) .. controls (384.48,135.33) and (387.44,132.38) .. (391.09,132.38) .. controls (394.73,132.38) and (397.69,135.33) .. (397.69,138.98) .. controls (397.69,142.63) and (394.73,145.59) .. (391.09,145.59) .. controls (387.44,145.59) and (384.48,142.63) .. (384.48,138.98) -- cycle ;
					\draw  [fill={rgb, 255:red, 0; green, 0; blue, 0 }  ,fill opacity=1 ] (366.07,55.99) .. controls (366.07,54.31) and (367.43,52.95) .. (369.11,52.95) .. controls (370.78,52.95) and (372.14,54.31) .. (372.14,55.99) .. controls (372.14,57.67) and (370.78,59.03) .. (369.11,59.03) .. controls (367.43,59.03) and (366.07,57.67) .. (366.07,55.99) -- cycle ;
					\draw  [fill={rgb, 255:red, 0; green, 0; blue, 0 }  ,fill opacity=1 ] (384.31,43.38) .. controls (384.31,41.7) and (385.67,40.34) .. (387.34,40.34) .. controls (389.02,40.34) and (390.38,41.7) .. (390.38,43.38) .. controls (390.38,45.05) and (389.02,46.41) .. (387.34,46.41) .. controls (385.67,46.41) and (384.31,45.05) .. (384.31,43.38) -- cycle ;
					\draw   (431.6,79.02) -- (431.04,85.41) -- (425.34,82.47) ;
					\draw   (285.94,99.37) -- (280.99,103.44) -- (279.09,97.33) ;
					\draw   (524.74,84.79) -- (526.67,78.68) -- (531.6,82.78) ;
					\draw   (541.63,78.53) -- (547.69,76.44) -- (547.33,82.84) ;
					\draw   (354.71,71.67) -- (349.43,68.05) -- (354.78,64.53) ;
					\draw (88,67.33) node [anchor=north west][inner sep=0.75pt]  [font=\small]  {$D_{1}$};
					\draw (90.67,34.67) node [anchor=north west][inner sep=0.75pt]  [font=\small]  {$D_{2}$};
					\draw (244,76.67) node [anchor=north west][inner sep=0.75pt]  [font=\small]  {$D_{1}$};
					\draw (284,77.33) node [anchor=north west][inner sep=0.75pt]  [font=\small]  {$D_{2}$};
					\draw (536,88) node [anchor=north west][inner sep=0.75pt]  [font=\small]  {$D_{2}$};
					\draw (503.33,87.33) node [anchor=north west][inner sep=0.75pt]  [font=\small]  {$D_{1}$};
					\draw (340.67,71) node [anchor=north west][inner sep=0.75pt]  [font=\small]  {$D_{1}$};
					\draw (372.14,55.99) node [anchor=north west][inner sep=0.75pt]  [font=\small]  {$D_{2}$};
					\draw (110.8,100.4) node [anchor=north west][inner sep=0.75pt]  [font=\small]  {$-$};
					\draw (46.7,74.84) node [anchor=north west][inner sep=0.75pt]  [font=\small]  {$-$};
					\draw (244.2,44.8) node [anchor=north west][inner sep=0.75pt]  [font=\small]  {$+$};
					\draw (194.36,72.66) node [anchor=north west][inner sep=0.75pt]  [font=\small]  {$+$};
					\draw (509.8,72.6) node [anchor=north west][inner sep=0.75pt]  [font=\small]  {$-$};
					\draw (424,99.8) node [anchor=north west][inner sep=0.75pt]  [font=\small]  {$+$};
					\draw (385.94,95.47) node [anchor=north west][inner sep=0.75pt]  [font=\small]  {$+$};
					\draw (410,47.8) node [anchor=north west][inner sep=0.75pt]  [font=\small]  {$-$};
					\draw (64,157.85) node [anchor=north west][inner sep=0.75pt]   [align=left] {(a)};
					\draw (377,157.85) node [anchor=north west][inner sep=0.75pt]   [align=left] {(c)};
					\draw (524,157.85) node [anchor=north west][inner sep=0.75pt]   [align=left] {(d)};
					\draw (215,158.85) node [anchor=north west][inner sep=0.75pt]   [align=left] {(b)};
				\end{tikzpicture}
				\caption{(a)--(d): Diagrams obtained by orientation-preserving smoothing at classical crossings $c_1$, $c_2$, $c_3$ and $c_+$. }  \label{fig13}
			\end{center}
		\end{figure} 
		
		According to Fig.~\ref{fig1119}(b), the straightforward calculations give
		$$
		\begin{aligned}
			W_{D_+}(c_1) &= -(0-2) = 2, \quad W_{D_+}(c_2) =-(2-0)= -2, \\ W_{D_+}(c_3) &=1-2= -1, \quad W_{D_+}(c_+) =2-1= 1.
		\end{aligned}
		$$ 
		From Fig.~\ref{fig13}, we see that $I(c_1)=I(c_2)=2$, $I(c_3)=3$ and $I(c_+)=1$, thus $c_1, c_2 \in E(D_+)$ and $c_3, c_+ \in O(D_+)$. Moreover, $i(c_1)=-2$, $i(c_2)=2$, $i(c_3)=1$, and $i(c_+)=-1$. Therefore, 
		\begin{eqnarray*}
			\mathbf{P}_{D_+}(x,y)&=&\operatorname{sgn}(c_1) i(c_1) x^{W_{D_+}(c_1)}+\operatorname{sgn}(c_2) i(c_2) x^{W_{D_+}(c_2)}\\
			&&+\operatorname{sgn}(c_3) i(c_3) y^{W_{D_+}(c_3)}+\operatorname{sgn}(c_+) i(c_+) y^{W_{D_+}(c_+)}\\
			&=& -(-2) x^{2}-2 x^{-2} + y^{-1} +(-1) y= 2x^2-2x^{-2}+y^{-1}-y \neq 0.
	\end{eqnarray*} }
\end{example} 

%%%%%
\section{The properties of $\mathbf{P}_K(x,y)$} \label{section3}

\subsection{$\mathbf{P}_K(x,y)$ and the odd writhe} \label{subsection3.0}

For a virtual knotoid $K$, the odd writhe invariant introduced in~\cite{GK} is given by
$
J(K)=\sum_{c\in O(K)}\operatorname{sgn}(c),
$
where $O(K)$ denotes the set of odd crossings of $K$. We now exhibit two virtual knotoids $K_1$ and $K_2$ that cannot be distinguished by $J(K)$, while $\mathbf{P}_K(x,y)$ successfully distinguishes them.

\begin{example} \label{example3.0}
	{\rm
		Let $K_1$ and $K_2$ be oriented virtual knotoids with diagrams $D_1$ and $D_2$ presented in Fig.~\ref{fig602}. 
		\begin{figure}[!ht]
			\begin{center}
				\tikzset{every picture/.style={line width=1.pt}}  
				\begin{tikzpicture}[x=0.75pt,y=0.75pt,yscale=-0.9,xscale=0.9]		
					\draw    (131.96,87.54) .. controls (158.32,49.99) and (198.95,81.67) .. (193.46,113.35) ;
					\draw    (125.37,154.41) .. controls (80.55,131.4) and (101.21,59.38) .. (150.63,106.31) ; 
					\draw    (125.37,154.41) .. controls (141.32,163.47) and (195.69,181.71) .. (213.22,130.95) ; 
					\draw    (196.59,82.12) .. controls (219.32,82.22) and (220.02,114.58) .. (213.22,130.95) ;
					\draw    (181.58,91.12) .. controls (164.01,105.59) and (192.27,135.23) .. (175.06,158.7) ;
					\draw    (127.66,100.9) .. controls (119.6,118.89) and (108.39,214.8) .. (167.86,171.1) ;
					\draw  [fill={rgb, 255:red, 0; green, 0; blue, 0 }  ,fill opacity=1 ] (147.59,106.31) .. controls (147.59,104.63) and (148.95,103.27) .. (150.63,103.27) .. controls (152.31,103.27) and (153.67,104.63) .. (153.67,106.31) .. controls (153.67,107.99) and (152.31,109.35) .. (150.63,109.35) .. controls (148.95,109.35) and (147.59,107.99) .. (147.59,106.31) -- cycle ;
					\draw  [fill={rgb, 255:red, 0; green, 0; blue, 0 }  ,fill opacity=1 ] (190.42,113.35) .. controls (190.42,111.67) and (191.78,110.31) .. (193.46,110.31) .. controls (195.13,110.31) and (196.49,111.67) .. (196.49,113.35) .. controls (196.49,115.02) and (195.13,116.38) .. (193.46,116.38) .. controls (191.78,116.38) and (190.42,115.02) .. (190.42,113.35) -- cycle ;
					\draw   (115.1,151.96) .. controls (115.1,148.31) and (118.06,145.35) .. (121.71,145.35) .. controls (125.35,145.35) and (128.31,148.31) .. (128.31,151.96) .. controls (128.31,155.6) and (125.35,158.56) .. (121.71,158.56) .. controls (118.06,158.56) and (115.1,155.6) .. (115.1,151.96) -- cycle ;
					\draw   (175.08,78.51) -- (170.77,72.5) -- (177.73,72.58) ;
					\draw    (371.96,86.54) .. controls (398.32,48.99) and (438.95,80.67) .. (433.46,112.35) ; 
					\draw    (355.84,147.27) .. controls (326.84,138.27) and (341.21,58.38) .. (390.63,105.31) ;
					\draw    (371.84,153.77) .. controls (387.78,162.83) and (435.69,180.71) .. (453.22,129.95) ;
					\draw    (436.59,81.12) .. controls (459.32,81.22) and (460.02,113.58) .. (453.22,129.95) ;
					\draw    (421.58,90.12) .. controls (404.01,104.59) and (436.84,151.27) .. (407.86,170.1) ;
					\draw    (367.66,99.9) .. controls (359.6,117.89) and (354.84,206.27) .. (407.86,170.1) ;
					\draw  [fill={rgb, 255:red, 0; green, 0; blue, 0 }  ,fill opacity=1 ] (387.59,105.31) .. controls (387.59,103.63) and (388.95,102.27) .. (390.63,102.27) .. controls (392.31,102.27) and (393.67,103.63) .. (393.67,105.31) .. controls (393.67,106.99) and (392.31,108.35) .. (390.63,108.35) .. controls (388.95,108.35) and (387.59,106.99) .. (387.59,105.31) -- cycle ;
					\draw  [fill={rgb, 255:red, 0; green, 0; blue, 0 }  ,fill opacity=1 ] (430.42,112.35) .. controls (430.42,110.67) and (431.78,109.31) .. (433.46,109.31) .. controls (435.13,109.31) and (436.49,110.67) .. (436.49,112.35) .. controls (436.49,114.02) and (435.13,115.38) .. (433.46,115.38) .. controls (431.78,115.38) and (430.42,114.02) .. (430.42,112.35) -- cycle ;
					\draw   (407.6,163.96) .. controls (407.6,160.31) and (410.56,157.35) .. (414.21,157.35) .. controls (417.85,157.35) and (420.81,160.31) .. (420.81,163.96) .. controls (420.81,167.6) and (417.85,170.56) .. (414.21,170.56) .. controls (410.56,170.56) and (407.6,167.6) .. (407.6,163.96) -- cycle ;
					\draw   (415.08,77.51) -- (410.77,71.5) -- (417.73,71.58) ;
					\draw (182.24,68.44) node [anchor=north west][inner sep=0.75pt]    {$c_{1}$};
					\draw (115.24,78.44) node [anchor=north west][inner sep=0.75pt]    {$c_{2}$};
					\draw (169.74,166.94) node [anchor=north west][inner sep=0.75pt]    {$c_{3}$};
					\draw (421.74,68.84) node [anchor=north west][inner sep=0.75pt]    {$c_{1}$};
					\draw (355.24,78.34) node [anchor=north west][inner sep=0.75pt]    {$c_{2}$};
					\draw (349.74,150.34) node [anchor=north west][inner sep=0.75pt]    {$c_{3}$};
					\draw (70.24,59.94) node [anchor=north west][inner sep=0.75pt]    {$D_{1}$};
					\draw (314.9,60.44) node [anchor=north west][inner sep=0.75pt]    {$D_{2}$};
					\draw (144.9,194.94) node [anchor=north west][inner sep=0.75pt]    {(a)};
					\draw (390.9,194.94) node [anchor=north west][inner sep=0.75pt]    {(b)};
				\end{tikzpicture}
				\caption{Virtual knotoid diagrams $D_1$ and $D_2$.} \label{fig602}
			\end{center}
		\end{figure}
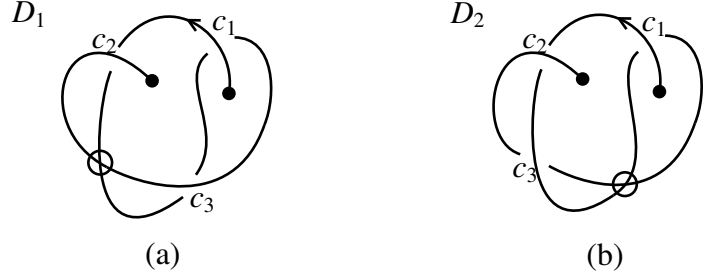 
		The diagram $D_1$ contains one virtual crossing and three classical crossings $c_1$, $c_2$, and $c_3$. One can see from Fig.~\ref{fig602}(a) that 
		$\operatorname{sgn}(c_1)=\operatorname{sgn}(c_2)=\operatorname{sgn}(c_3)=-1$. 
		\begin{figure}[!ht]
			\begin{center}
				\tikzset{every picture/.style={line width=1.pt}}  
				\begin{tikzpicture}[x=0.75pt,y=0.75pt,yscale=-0.9,xscale=0.9]		
					\draw    (74.46,477.6) .. controls (102.54,447.31) and (137.03,472.42) .. (124.08,481.18) ;
					\draw [color={rgb, 255:red, 248; green, 72; blue, 61 }  ,draw opacity=1 ]   (67.87,544.47) .. controls (23.05,521.46) and (43.71,449.44) .. (93.13,496.37) ;
					\draw [color={rgb, 255:red, 248; green, 72; blue, 61 }  ,draw opacity=1 ]   (67.87,544.47) .. controls (83.82,553.53) and (138.19,571.77) .. (155.72,521.01) ;
					\draw [color={rgb, 255:red, 248; green, 72; blue, 61 }  ,draw opacity=1 ]   (139.09,472.18) .. controls (161.82,472.28) and (162.52,504.64) .. (155.72,521.01) ;
					\draw    (124.08,481.18) .. controls (106.51,495.65) and (134.77,525.29) .. (117.56,548.76) ;
					\draw    (70.16,490.96) .. controls (62.1,508.95) and (50.89,604.86) .. (110.36,561.16) ;
					\draw  [color={rgb, 255:red, 248; green, 72; blue, 61 }  ,draw opacity=1 ][fill={rgb, 255:red, 248; green, 72; blue, 61 }  ,fill opacity=1 ] (90.09,496.37) .. controls (90.09,494.69) and (91.45,493.33) .. (93.13,493.33) .. controls (94.81,493.33) and (96.17,494.69) .. (96.17,496.37) .. controls (96.17,498.05) and (94.81,499.41) .. (93.13,499.41) .. controls (91.45,499.41) and (90.09,498.05) .. (90.09,496.37) -- cycle ;
					\draw  [color={rgb, 255:red, 248; green, 72; blue, 61 }  ,draw opacity=1 ][fill={rgb, 255:red, 248; green, 72; blue, 61 }  ,fill opacity=1 ] (132.92,503.41) .. controls (132.92,501.73) and (134.28,500.37) .. (135.96,500.37) .. controls (137.63,500.37) and (138.99,501.73) .. (138.99,503.41) .. controls (138.99,505.08) and (137.63,506.44) .. (135.96,506.44) .. controls (134.28,506.44) and (132.92,505.08) .. (132.92,503.41) -- cycle ;
					\draw   (57.6,542.02) .. controls (57.6,538.37) and (60.56,535.41) .. (64.21,535.41) .. controls (67.85,535.41) and (70.81,538.37) .. (70.81,542.02) .. controls (70.81,545.66) and (67.85,548.62) .. (64.21,548.62) .. controls (60.56,548.62) and (57.6,545.66) .. (57.6,542.02) -- cycle ;
					\draw   (100.15,466.02) -- (92.79,465.3) -- (97.43,460.11) ;
					\draw [color={rgb, 255:red, 248; green, 72; blue, 61 }  ,draw opacity=1 ]   (139.09,472.18) .. controls (129.03,472.02) and (131.03,496.82) .. (135.96,503.41) ;
					\draw   (158.79,482) -- (158.63,489.39) -- (153.11,485.16) ;
					\draw    (242.36,475.45) .. controls (270.36,444.95) and (305.61,472.14) .. (300.12,503.82) ;
					\draw [color={rgb, 255:red, 248; green, 72; blue, 61 }  ,draw opacity=1 ]   (232.04,544.88) .. controls (187.22,521.87) and (239.86,464.95) .. (234.32,491.37) ;
					\draw [color={rgb, 255:red, 248; green, 72; blue, 61 }  ,draw opacity=1 ]   (232.04,544.88) .. controls (247.98,553.94) and (302.36,572.18) .. (319.89,521.42) ;
					\draw [color={rgb, 255:red, 248; green, 72; blue, 61 }  ,draw opacity=1 ]   (303.26,472.6) .. controls (325.99,472.69) and (326.68,505.06) .. (319.89,521.42) ;
					\draw [color={rgb, 255:red, 248; green, 72; blue, 61 }  ,draw opacity=1 ]   (288.25,481.59) .. controls (270.68,496.06) and (298.93,525.7) .. (281.73,549.17) ;
					\draw [color={rgb, 255:red, 248; green, 72; blue, 61 }  ,draw opacity=1 ]   (234.32,491.37) .. controls (226.27,509.36) and (215.06,605.27) .. (274.53,561.57) ;
					\draw  [fill={rgb, 255:red, 0; green, 0; blue, 0 }  ,fill opacity=1 ] (254.26,496.78) .. controls (254.26,495.1) and (255.62,493.74) .. (257.3,493.74) .. controls (258.97,493.74) and (260.33,495.1) .. (260.33,496.78) .. controls (260.33,498.46) and (258.97,499.82) .. (257.3,499.82) .. controls (255.62,499.82) and (254.26,498.46) .. (254.26,496.78) -- cycle ;
					\draw  [fill={rgb, 255:red, 0; green, 0; blue, 0 }  ,fill opacity=1 ] (297.09,503.82) .. controls (297.09,502.14) and (298.45,500.78) .. (300.12,500.78) .. controls (301.8,500.78) and (303.16,502.14) .. (303.16,503.82) .. controls (303.16,505.5) and (301.8,506.86) .. (300.12,506.86) .. controls (298.45,506.86) and (297.09,505.5) .. (297.09,503.82) -- cycle ;
					\draw  [color={rgb, 255:red, 248; green, 72; blue, 61 }  ,draw opacity=1 ] (221.77,542.43) .. controls (221.77,538.78) and (224.72,535.82) .. (228.37,535.82) .. controls (232.02,535.82) and (234.98,538.78) .. (234.98,542.43) .. controls (234.98,546.08) and (232.02,549.04) .. (228.37,549.04) .. controls (224.72,549.04) and (221.77,546.08) .. (221.77,542.43) -- cycle ;
					\draw   (281.74,468.98) -- (277.44,462.97) -- (284.4,463.05) ;
					\draw    (402.46,478.68) .. controls (428.82,441.13) and (469.45,472.81) .. (463.96,504.49) ;
					\draw    (395.87,545.55) .. controls (351.05,522.54) and (371.71,450.52) .. (421.13,497.45) ;
					\draw [color={rgb, 255:red, 248; green, 72; blue, 61 }  ,draw opacity=1 ]   (445.56,549.84) .. controls (440.59,557.19) and (466.19,572.85) .. (483.72,522.09) ;
					\draw [color={rgb, 255:red, 248; green, 72; blue, 61 }  ,draw opacity=1 ]   (467.09,473.26) .. controls (489.82,473.36) and (490.52,505.72) .. (483.72,522.09) ;
					\draw [color={rgb, 255:red, 248; green, 72; blue, 61 }  ,draw opacity=1 ]   (452.08,482.26) .. controls (434.51,496.73) and (462.77,526.37) .. (445.56,549.84) ;
					\draw    (398.16,492.04) .. controls (390.1,510.03) and (378.89,605.94) .. (438.36,562.24) ;
					\draw  [fill={rgb, 255:red, 0; green, 0; blue, 0 }  ,fill opacity=1 ] (418.09,497.45) .. controls (418.09,495.77) and (419.45,494.41) .. (421.13,494.41) .. controls (422.81,494.41) and (424.17,495.77) .. (424.17,497.45) .. controls (424.17,499.13) and (422.81,500.49) .. (421.13,500.49) .. controls (419.45,500.49) and (418.09,499.13) .. (418.09,497.45) -- cycle ;
					\draw  [fill={rgb, 255:red, 0; green, 0; blue, 0 }  ,fill opacity=1 ] (460.92,504.49) .. controls (460.92,502.81) and (462.28,501.45) .. (463.96,501.45) .. controls (465.63,501.45) and (466.99,502.81) .. (466.99,504.49) .. controls (466.99,506.16) and (465.63,507.52) .. (463.96,507.52) .. controls (462.28,507.52) and (460.92,506.16) .. (460.92,504.49) -- cycle ;
					\draw   (385.6,543.1) .. controls (385.6,539.45) and (388.56,536.49) .. (392.21,536.49) .. controls (395.85,536.49) and (398.81,539.45) .. (398.81,543.1) .. controls (398.81,546.74) and (395.85,549.7) .. (392.21,549.7) .. controls (388.56,549.7) and (385.6,546.74) .. (385.6,543.1) -- cycle ;
					\draw   (445.58,469.65) -- (441.27,463.64) -- (448.23,463.72) ;
					\draw    (242.36,475.45) .. controls (238.06,480.82) and (244.36,494.45) .. (257.3,496.78) ;
					\draw   (233.5,507.55) -- (229.32,513.65) -- (227.02,507.08) ;
					\draw    (395.87,545.55) .. controls (405.35,550.29) and (446.59,556.19) .. (438.36,562.24) ;
					\draw   (488.32,486.49) -- (486.93,493.75) -- (482.18,488.66) ;
					\draw (135.7,455.75) node [anchor=north west][inner sep=0.75pt]  [font=\small]  {$D_{1_{1}}$};
					\draw (106.6,449.15) node [anchor=north west][inner sep=0.75pt]  [font=\small]  {$D_{1_{2}}$};
					\draw (205.95,470.88) node [anchor=north west][inner sep=0.75pt]  [font=\small]  {$D_{1_{1}}$};
					\draw (234.85,447.88) node [anchor=north west][inner sep=0.75pt]  [font=\small]  {$D_{1_{2}}$};
					\draw (423.85,445.38) node [anchor=north west][inner sep=0.75pt]  [font=\small]  {$D_{1_{2}}$};
					\draw (477.06,460.38) node [anchor=north west][inner sep=0.75pt]  [font=\small]  {$D_{1_{1}}$};
					\draw (78.9,580.94) node [anchor=north west][inner sep=0.75pt]    {(a)};
					\draw (252.9,580.9) node [anchor=north west][inner sep=0.75pt]    {(b)};
					\draw (420.9,580.94) node [anchor=north west][inner sep=0.75pt]    {(c)};
					\draw (60.5,470.18) node [anchor=north west][inner sep=0.75pt]  [font=\small]  {$-$};
					\draw (290.82,460.68) node [anchor=north west][inner sep=0.75pt]  [font=\small]  {$+$};
					\draw (453.82,462.18) node [anchor=north west][inner sep=0.75pt]  [font=\small]  {$+$};
					\draw (114.5,554.68) node [anchor=north west][inner sep=0.75pt]  [font=\small]  {$-$};
				\end{tikzpicture}
				\caption{(a)--(c): Diagrams obtained by orientation-preserving smoothing at classical crossings $c_1$, $c_2$, and $c_3$ in $D_1$.} \label{fig604}
			\end{center}
		\end{figure}
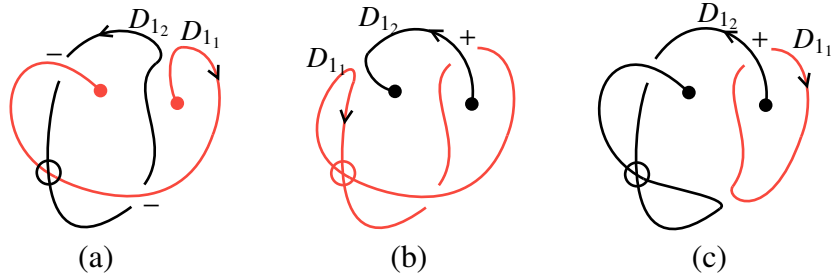 
		From Fig.~\ref{fig604}, we obtain $I(c_1)=2$, $I(c_2)=I(c_3)=1$, thus $c_2,c_3\in O(D_1)$ and $c_1\in E(D_1)$. The diagram $D_2$ contains one virtual crossing and three classical crossings $c_1$, $c_2$, and $c_3$. One can see from Fig.~\ref{fig602}(b) that 
		$\operatorname{sgn}(c_1)=\operatorname{sgn}(c_2)=\operatorname{sgn}(c_3)=-1$. 
		\begin{figure}[!ht]
			\begin{center}
				\tikzset{every picture/.style={line width=1.pt}}  
				\begin{tikzpicture}[x=0.75pt,y=0.75pt,yscale=-0.9,xscale=0.9]		
					\draw    (85.96,656.6) .. controls (123.04,625.01) and (148.54,651.51) .. (135.58,660.18) ;
					\draw [color={rgb, 255:red, 248; green, 72; blue, 61 }  ,draw opacity=1 ]   (69.84,717.33) .. controls (40.84,708.33) and (55.21,628.44) .. (104.63,675.37) ;
					\draw [color={rgb, 255:red, 248; green, 72; blue, 61 }  ,draw opacity=1 ]   (85.84,723.83) .. controls (101.78,732.89) and (149.69,750.77) .. (167.22,700.01) ;
					\draw [color={rgb, 255:red, 248; green, 72; blue, 61 }  ,draw opacity=1 ]   (150.59,651.18) .. controls (173.32,651.28) and (174.02,683.64) .. (167.22,700.01) ;
					\draw    (135.58,660.18) .. controls (118.01,674.65) and (150.84,721.33) .. (121.86,740.16) ;
					\draw    (81.66,669.96) .. controls (73.6,687.95) and (68.84,776.33) .. (121.86,740.16) ; 
					\draw  [color={rgb, 255:red, 248; green, 72; blue, 61 }  ,draw opacity=1 ][fill={rgb, 255:red, 248; green, 72; blue, 61 }  ,fill opacity=1 ] (101.59,675.37) .. controls (101.59,673.69) and (102.95,672.33) .. (104.63,672.33) .. controls (106.31,672.33) and (107.67,673.69) .. (107.67,675.37) .. controls (107.67,677.05) and (106.31,678.41) .. (104.63,678.41) .. controls (102.95,678.41) and (101.59,677.05) .. (101.59,675.37) -- cycle ;
					\draw  [color={rgb, 255:red, 248; green, 72; blue, 61 }  ,draw opacity=1 ][fill={rgb, 255:red, 248; green, 72; blue, 61 }  ,fill opacity=1 ] (144.42,682.41) .. controls (144.42,680.73) and (145.78,679.37) .. (147.46,679.37) .. controls (149.13,679.37) and (150.49,680.73) .. (150.49,682.41) .. controls (150.49,684.08) and (149.13,685.44) .. (147.46,685.44) .. controls (145.78,685.44) and (144.42,684.08) .. (144.42,682.41) -- cycle ;
					\draw   (121.6,734.02) .. controls (121.6,730.37) and (124.56,727.41) .. (128.21,727.41) .. controls (131.85,727.41) and (134.81,730.37) .. (134.81,734.02) .. controls (134.81,737.66) and (131.85,740.62) .. (128.21,740.62) .. controls (124.56,740.62) and (121.6,737.66) .. (121.6,734.02) -- cycle ;
					\draw   (110.55,645.66) -- (103.16,645.37) -- (107.49,639.92) ;
					\draw    (259.69,654.04) .. controls (293.91,621.93) and (319.91,655.27) .. (317.46,682.41) ;
					\draw [color={rgb, 255:red, 248; green, 72; blue, 61 }  ,draw opacity=1 ]   (239.84,717.33) .. controls (210.84,708.33) and (254.57,642.6) .. (251.66,669.96) ;
					\draw [color={rgb, 255:red, 248; green, 72; blue, 61 }  ,draw opacity=1 ]   (255.84,723.83) .. controls (271.78,732.89) and (319.69,750.77) .. (337.22,700.01) ;
					\draw [color={rgb, 255:red, 248; green, 72; blue, 61 }  ,draw opacity=1 ]   (320.59,651.18) .. controls (343.32,651.28) and (344.02,683.64) .. (337.22,700.01) ;
					\draw [color={rgb, 255:red, 248; green, 72; blue, 61 }  ,draw opacity=1 ]   (305.58,660.18) .. controls (288.01,674.65) and (320.84,721.33) .. (291.86,740.16) ;
					\draw [color={rgb, 255:red, 248; green, 72; blue, 61 }  ,draw opacity=1 ]   (251.66,669.96) .. controls (243.6,687.95) and (238.84,776.33) .. (291.86,740.16) ;
					\draw  [fill={rgb, 255:red, 0; green, 0; blue, 0 }  ,fill opacity=1 ] (271.59,675.37) .. controls (271.59,673.69) and (272.95,672.33) .. (274.63,672.33) .. controls (276.31,672.33) and (277.67,673.69) .. (277.67,675.37) .. controls (277.67,677.05) and (276.31,678.41) .. (274.63,678.41) .. controls (272.95,678.41) and (271.59,677.05) .. (271.59,675.37) -- cycle ;
					\draw  [fill={rgb, 255:red, 0; green, 0; blue, 0 }  ,fill opacity=1 ] (314.42,682.41) .. controls (314.42,680.73) and (315.78,679.37) .. (317.46,679.37) .. controls (319.13,679.37) and (320.49,680.73) .. (320.49,682.41) .. controls (320.49,684.08) and (319.13,685.44) .. (317.46,685.44) .. controls (315.78,685.44) and (314.42,684.08) .. (314.42,682.41) -- cycle ;
					\draw  [color={rgb, 255:red, 248; green, 72; blue, 61 }  ,draw opacity=1 ] (291.6,734.02) .. controls (291.6,730.37) and (294.56,727.41) .. (298.21,727.41) .. controls (301.85,727.41) and (304.81,730.37) .. (304.81,734.02) .. controls (304.81,737.66) and (301.85,740.62) .. (298.21,740.62) .. controls (294.56,740.62) and (291.6,737.66) .. (291.6,734.02) -- cycle ;
					\draw   (299.08,647.57) -- (294.77,641.56) -- (301.73,641.64) ;
					\draw [color={rgb, 255:red, 248; green, 72; blue, 61 }  ,draw opacity=1 ]   (427.96,656.6) .. controls (454.32,619.05) and (494.95,650.73) .. (489.46,682.41) ;
					\draw [color={rgb, 255:red, 248; green, 72; blue, 61 }  ,draw opacity=1 ]   (411.84,717.33) .. controls (382.84,708.33) and (397.21,628.44) .. (446.63,675.37) ;
					\draw    (427.84,723.83) .. controls (443.78,732.89) and (491.69,750.77) .. (509.22,700.01) ;
					\draw    (492.59,651.18) .. controls (515.32,651.28) and (516.02,683.64) .. (509.22,700.01) ;
					\draw    (477.58,660.18) .. controls (460.01,674.65) and (492.84,721.33) .. (463.86,740.16) ;
					\draw    (427.84,723.83) .. controls (415.27,715.27) and (410.84,776.33) .. (463.86,740.16) ;
					\draw  [color={rgb, 255:red, 248; green, 72; blue, 61 }  ,draw opacity=1 ][fill={rgb, 255:red, 248; green, 72; blue, 61 }  ,fill opacity=1 ] (443.59,675.37) .. controls (443.59,673.69) and (444.95,672.33) .. (446.63,672.33) .. controls (448.31,672.33) and (449.67,673.69) .. (449.67,675.37) .. controls (449.67,677.05) and (448.31,678.41) .. (446.63,678.41) .. controls (444.95,678.41) and (443.59,677.05) .. (443.59,675.37) -- cycle ;
					\draw  [color={rgb, 255:red, 248; green, 72; blue, 61 }  ,draw opacity=1 ][fill={rgb, 255:red, 248; green, 72; blue, 61 }  ,fill opacity=1 ] (486.42,682.41) .. controls (486.42,680.73) and (487.78,679.37) .. (489.46,679.37) .. controls (491.13,679.37) and (492.49,680.73) .. (492.49,682.41) .. controls (492.49,684.08) and (491.13,685.44) .. (489.46,685.44) .. controls (487.78,685.44) and (486.42,684.08) .. (486.42,682.41) -- cycle ;
					\draw   (463.6,734.02) .. controls (463.6,730.37) and (466.56,727.41) .. (470.21,727.41) .. controls (473.85,727.41) and (476.81,730.37) .. (476.81,734.02) .. controls (476.81,737.66) and (473.85,740.62) .. (470.21,740.62) .. controls (466.56,740.62) and (463.6,737.66) .. (463.6,734.02) -- cycle ;
					\draw   (471.08,647.57) -- (466.77,641.56) -- (473.73,641.64) ;
					\draw [color={rgb, 255:red, 248; green, 72; blue, 61 }  ,draw opacity=1 ]   (150.59,651.18) .. controls (144.47,651.16) and (139.54,675.01) .. (147.46,682.41) ;
					\draw   (172.43,666.45) -- (170.61,673.61) -- (166.18,668.25) ;
					\draw    (259.69,654.04) .. controls (255.39,659.4) and (261.69,673.04) .. (274.63,675.37) ;
					\draw   (250.5,692.89) -- (246.32,698.98) -- (244.02,692.41) ;
					\draw [color={rgb, 255:red, 248; green, 72; blue, 61 }  ,draw opacity=1 ]   (411.84,717.33) .. controls (420.6,722.6) and (419.93,702.6) .. (424.15,670.51) ;
					\draw   (423.68,728.54) -- (419.41,734.57) -- (417.2,727.97) ;
					\draw (94.9,755.94) node [anchor=north west][inner sep=0.75pt]    {(a)};
					\draw (271.9,755.9) node [anchor=north west][inner sep=0.75pt]    {(b)};
					\draw (442.9,755.94) node [anchor=north west][inner sep=0.75pt]    {(c)};
					\draw (155.03,639.84) node [anchor=north west][inner sep=0.75pt]  [font=\small]  {$D_{2_{1}}$};
					\draw (117.27,627.84) node [anchor=north west][inner sep=0.75pt]  [font=\small]  {$D_{2_{2}}$};
					\draw (221.37,654.51) node [anchor=north west][inner sep=0.75pt]  [font=\small]  {$D_{2_{1}}$};
					\draw (255.18,627.84) node [anchor=north west][inner sep=0.75pt]  [font=\small]  {$D_{2_{2}}$};
					\draw (383.37,703.17) node [anchor=north west][inner sep=0.75pt]  [font=\small]  {$D_{2_{1}}$};
					\draw (401.85,738.17) node [anchor=north west][inner sep=0.75pt]  [font=\small]  {$D_{2_{2}}$};
					\draw (80.32,710.18) node [anchor=north west][inner sep=0.75pt]  [font=\small]  {$+$};
					\draw (71.84,646.58) node [anchor=north west][inner sep=0.75pt]  [font=\small]  {$-$};
					\draw (306.34,639.08) node [anchor=north west][inner sep=0.75pt]  [font=\small]  {$+$};
					\draw (480.34,638.58) node [anchor=north west][inner sep=0.75pt]  [font=\small]  {$-$};
				\end{tikzpicture}
				\caption{(a)--(c): Diagrams obtained by orientation-preserving smoothing at classical crossings $c_1$, $c_2$, and $c_3$ in $D_2$.} \label{fig605}
			\end{center}
		\end{figure}
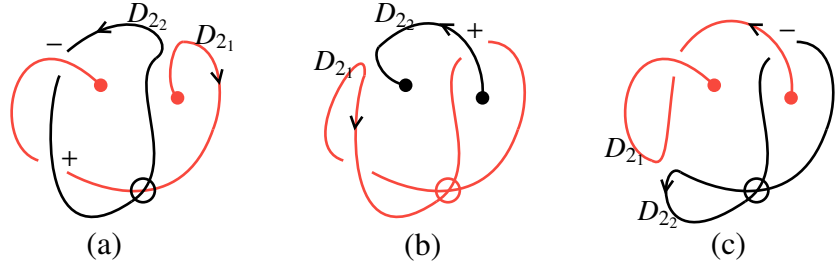
		From Fig.~\ref{fig605}, we obtain $I(c_1)=2$, $I(c_2)=I(c_3)=1$, thus $c_2,c_3\in O(D_1)$ and $c_1\in E(D_1)$. Then we have $J(K_1)=J(K_2)=-2$.

		\begin{figure}[!ht]
			\begin{center}
				\tikzset{every picture/.style={line width=1.pt}}  
				\begin{tikzpicture}[x=0.75pt,y=0.75pt,yscale=-0.9,xscale=0.9]		
					\draw    (130.96,275.54) .. controls (157.32,237.99) and (197.95,269.67) .. (192.46,301.35) ;
					\draw    (124.37,342.41) .. controls (79.55,319.4) and (100.21,247.38) .. (149.63,294.31) ;
					\draw    (124.37,342.41) .. controls (140.32,351.47) and (194.69,369.71) .. (212.22,318.95) ;
					\draw    (195.59,270.12) .. controls (218.32,270.22) and (219.02,302.58) .. (212.22,318.95) ;
					\draw    (180.58,279.12) .. controls (163.01,293.59) and (191.27,323.23) .. (174.06,346.7) ;
					\draw    (126.66,288.9) .. controls (118.6,306.89) and (107.39,402.8) .. (166.86,359.1) ; 
					\draw  [fill={rgb, 255:red, 0; green, 0; blue, 0 }  ,fill opacity=1 ] (146.59,294.31) .. controls (146.59,292.63) and (147.95,291.27) .. (149.63,291.27) .. controls (151.31,291.27) and (152.67,292.63) .. (152.67,294.31) .. controls (152.67,295.99) and (151.31,297.35) .. (149.63,297.35) .. controls (147.95,297.35) and (146.59,295.99) .. (146.59,294.31) -- cycle ;
					\draw  [fill={rgb, 255:red, 0; green, 0; blue, 0 }  ,fill opacity=1 ] (189.42,301.35) .. controls (189.42,299.67) and (190.78,298.31) .. (192.46,298.31) .. controls (194.13,298.31) and (195.49,299.67) .. (195.49,301.35) .. controls (195.49,303.02) and (194.13,304.38) .. (192.46,304.38) .. controls (190.78,304.38) and (189.42,303.02) .. (189.42,301.35) -- cycle ;
					\draw   (114.1,339.96) .. controls (114.1,336.31) and (117.06,333.35) .. (120.71,333.35) .. controls (124.35,333.35) and (127.31,336.31) .. (127.31,339.96) .. controls (127.31,343.6) and (124.35,346.56) .. (120.71,346.56) .. controls (117.06,346.56) and (114.1,343.6) .. (114.1,339.96) -- cycle ;
					\draw   (174.08,266.51) -- (169.77,260.5) -- (176.73,260.58) ;
					\draw    (370.96,274.54) .. controls (397.32,236.99) and (437.95,268.67) .. (432.46,300.35) ;
					\draw    (354.84,335.27) .. controls (325.84,326.27) and (340.21,246.38) .. (389.63,293.31) ;
					\draw    (370.84,341.77) .. controls (386.78,350.83) and (434.69,368.71) .. (452.22,317.95) ;
					\draw    (435.59,269.12) .. controls (458.32,269.22) and (459.02,301.58) .. (452.22,317.95) ;
					\draw    (420.58,278.12) .. controls (403.01,292.59) and (435.84,339.27) .. (406.86,358.1) ; 
					\draw    (366.66,287.9) .. controls (358.6,305.89) and (353.84,394.27) .. (406.86,358.1) ;
					\draw  [fill={rgb, 255:red, 0; green, 0; blue, 0 }  ,fill opacity=1 ] (386.59,293.31) .. controls (386.59,291.63) and (387.95,290.27) .. (389.63,290.27) .. controls (391.31,290.27) and (392.67,291.63) .. (392.67,293.31) .. controls (392.67,294.99) and (391.31,296.35) .. (389.63,296.35) .. controls (387.95,296.35) and (386.59,294.99) .. (386.59,293.31) -- cycle ;
					\draw  [fill={rgb, 255:red, 0; green, 0; blue, 0 }  ,fill opacity=1 ] (429.42,300.35) .. controls (429.42,298.67) and (430.78,297.31) .. (432.46,297.31) .. controls (434.13,297.31) and (435.49,298.67) .. (435.49,300.35) .. controls (435.49,302.02) and (434.13,303.38) .. (432.46,303.38) .. controls (430.78,303.38) and (429.42,302.02) .. (429.42,300.35) -- cycle ;
					\draw   (406.6,351.96) .. controls (406.6,348.31) and (409.56,345.35) .. (413.21,345.35) .. controls (416.85,345.35) and (419.81,348.31) .. (419.81,351.96) .. controls (419.81,355.6) and (416.85,358.56) .. (413.21,358.56) .. controls (409.56,358.56) and (406.6,355.6) .. (406.6,351.96) -- cycle ;
					\draw   (414.08,265.51) -- (409.77,259.5) -- (416.73,259.58) ;
					\draw (69.24,247.94) node [anchor=north west][inner sep=0.75pt]    {$D_{1}$};
					\draw (313.9,248.44) node [anchor=north west][inner sep=0.75pt]    {$D_{2}$};
					\draw (143.9,382.94) node [anchor=north west][inner sep=0.75pt]    {(a)};
					\draw (389.9,382.94) node [anchor=north west][inner sep=0.75pt]    {(b)};
					\draw (192.66,277.54) node [anchor=north west][inner sep=0.75pt]    {$0$};
					\draw (157.16,312.04) node [anchor=north west][inner sep=0.75pt]    {$-1$};
					\draw (151.66,245.04) node [anchor=north west][inner sep=0.75pt]    {$1$};
					\draw (432.66,277.94) node [anchor=north west][inner sep=0.75pt]    {$0$};
					\draw (77.66,299.44) node [anchor=north west][inner sep=0.75pt]    {$-1$};
					\draw (122.16,307.44) node [anchor=north west][inner sep=0.75pt]    {$0$};
					\draw (208.16,322.94) node [anchor=north west][inner sep=0.75pt]    {$-2$};
					\draw (391.16,243.94) node [anchor=north west][inner sep=0.75pt]    {$1$};
					\draw (139.66,273.94) node [anchor=north west][inner sep=0.75pt]    {$0$};
					\draw (316.66,293.94) node [anchor=north west][inner sep=0.75pt]    {$-1$};
					\draw (404.66,314.44) node [anchor=north west][inner sep=0.75pt]    {$1$};
					\draw (363.16,301.94) node [anchor=north west][inner sep=0.75pt]    {$0$};
					\draw (455.66,297.94) node [anchor=north west][inner sep=0.75pt]    {$0$};
					\draw (378.16,273.44) node [anchor=north west][inner sep=0.75pt]    {$0$};
				\end{tikzpicture}
				\caption{Integer labeling for $D_1$ and $D_2$, respectively.} \label{fig603}
			\end{center}
		\end{figure}
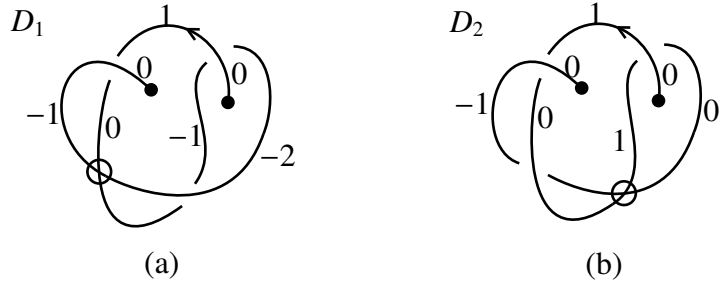     
		Now we compute $\mathbf{P}_{K_1}(x,y)$ and $\mathbf{P}_{K_2}(x,y)$.  Based on Fig.~\ref{fig603}(a), we have 
		$$
		W_{D_1}(c_1)=-(-1-1)=2, \quad W_{D_1}(c_2)=-(1-0)=-1, \quad
		W_{D_1}(c_3)=-(0-(-1))=-1.
		$$  
		From Fig.~\ref{fig604} we get $i(c_1)=-2$, $i(c_2)=1$, $i(c_3)=1$. Therefore, 
		$$
		\mathbf{P}_{K_1}(x,y)=-(-2)x^2-y^{-1}-y^{-1}=2x^2-2y^{-1}.
		$$ 
		
		Based on Fig.~\ref{fig603}(b), we have 
		$$
		W_{D_2}(c_1)=-(1-1)=0, \quad W_{D_2}(c_2)=-(1-0)=-1, \quad
		W_{D_2}(c_3)=-(0-1)=1.
		$$ 
		From Fig.~\ref{fig605} we get $i(c_1)=0$, $i(c_2)=1$, $i(c_3)=-1$. Therefore, 
		$$
		\mathbf{P}_{K_2}(x,y)=-0-y^{-1}-(-1)y=y-y^{-1}.
		$$ 
		From the foregoing considerations, we conclude that $\mathbf{P}_{K_1}(x,y)\neq \mathbf{P}_{K_2}(x,y)$. Hence, $K_1$ and $K_2$ are distinct virtual knotoids.    
	}
\end{example}

\subsection{$\mathbf{P}_K(x,y)$ and the affine index polynomial} \label{subsection3.1}

For an oriented virtual knotoid $K$, let $AI_K(t)$ be the affine index polynomial of $K$ defined in~\cite{GK}. Recall that for a virtual or classical knotoid diagram $K$, its affine index polynomial is defined by $AI_K(t)=\sum_{c}\operatorname{sgn}(c)\bigl(t^{w_K(c)}-1\bigr)$. We show that $\mathbf{P}_K(x,y)$ is equivalent to $AI_K(t)$.

We begin with the following lemma, which will be needed in the proof of Theorem~\ref{thm2}.	

\begin{lemma}\label{lm2}
	Let $D$ be an oriented virtual knotoid diagram and let $c$ be a classical crossing of $D$. Then the intersection index $i(c)$ has the same parity as the crossing $c$.
\end{lemma}

\begin{proof}
	Suppose that the orientation-preserving smoothing at $c$ produces the ordered two-component virtual knotoid diagram $D_1\cup D_2$.
	
	By definition,
	$ i(c)=\sum_{e\in D_1\cap D_2}\operatorname{ind}(e)$,	
	where $\operatorname{ind}(e)=\pm1$. Hence,
	$$
	i(c)\equiv
	\sum_{e\in D_1\cap D_2}1=I(c)\pmod2,
	$$
	where $I(c)$ denotes the number of classical crossings between the two components $D_1$ and $D_2$. Since the parity of the crossing $c$ is defined by the parity of $I(c)$, see subsection~\ref{subsection2.2}, it follows that $i(c)$ and $c$ have the same parity.
\end{proof}

Let $D$ be an oriented virtual knotoid diagram and let $c$ be a classical crossing of $D$. Consider the diagram obtained by applying the orientation-preserving smoothing, also known as  $1$-smoothing, at the crossing $c$. This operation produces a two-component diagram, denoted by $D_1 \cup D_2$. The following relation between the $i(c)$ and the $W_D(c)$ will be useful, where $i(c)$ denotes the intersection index of $c$ and $W_D(c)$ denotes the weight of $c$.

We note that $i(c) = -W(D_1\cup D_2)$, where $W(D_1\cup D_2)$ denotes the wiggle number defined in~\cite{FK}. From the proof of Theorem 7 in~\cite{FK}, we have $W(D_1\cup D_2)=W_D(c)$, which yields the following conclusion.

\begin{lemma}
	$i(c)=-W_D(c)$. \label{lm1}
\end{lemma}

Combined with Lemma~\ref{lm2}, we have

\begin{corollary} \label{cor:weight-parity}
	Let $D$ be an oriented virtual knotoid diagram and let $c$ be a classical crossing of $D$.	Then $W_D(c)$ and $c$ have the same parity.
\end{corollary}

We now turn to a discussion of the relation between $\mathbf{P}_K(x,y)$ and $AI_K(t)$.

\begin{definition}
	Let $\mathbf{Q}^1$ and $\mathbf{Q}^2$ be two invariants of virtual knotoid. We say that $\mathbf{Q}^1$ and $\mathbf{Q}^2$ are equivalent if $$\mathbf{Q}^1_{K_1}=\mathbf{Q}^1_{K_2} \iff \mathbf{Q}^2_{K_1}=\mathbf{Q}^2_{K_2},$$ for any virtual knotoids $K_1$ and $K_2$.
\end{definition}

\begin{theorem}\label{thm2} 
	The two-variable parity polynomial $\mathbf{P}_D(x,y)$ and the affine index polynomial $AI_K(t)$ are equivalent invariants, that is, for any two oriented virtual knotoids $K_1$ and $K_2$,
	$$
	AI_{K_1}(t)=AI_{K_2}(t)
	\iff
	\mathbf P_{K_1}(x,y)=
	\mathbf P_{K_2}(x,y).
	$$
\end{theorem}

\begin{proof}
	For every oriented virtual knotoid diagram $D$, write
	\[
	AI_D(t)=
	\sum_{m\neq0}
	a_m(D)(t^m-1),
	\]
	where
	$
	a_m(D)=
	\sum_{\substack{c\in C(D)\\W_D(c)=m}}
	\operatorname{sgn}(c).
	$
	Let $V$ be the free abelian group with basis
	$\{\,t^m-1\mid m\neq0\,\}$, there exists a unique
	$\mathbb Z$-linear map
	$
	\Phi:V\longrightarrow
	\mathbb Z[x^{\pm1},y^{\pm1}]
	$
	defined by
	\[
	\Phi(t^m-1)=
	\begin{cases}
		-mx^m,&m\text{ even},\\[1ex]
		-my^m,&m\text{ odd}.
	\end{cases}
	\]
	
	By Lemma~\ref{lm1} and Corollary~\ref{cor:weight-parity}, we have
	$i(c)=-W_D(c)$, and the weight $W_D(c)$ has the same parity as the crossing $c$.
	Hence,
	\begin{equation} \label{equ6}
		\mathbf P_D(x,y)
		=
		\Phi(AI_D(t)).	
	\end{equation}
	Suppose first that
	$
	AI_{K_1}(t)=AI_{K_2}(t).
	$
	Applying $\Phi$ to both sides gives
	\[
	\mathbf P_{K_1}(x,y)
	=
	\Phi(AI_{K_1}(t))
	=
	\Phi(AI_{K_2}(t))
	=
	\mathbf P_{K_2}(x,y).
	\]
	Conversely, suppose that
	$
	\mathbf P_{K_1}(x,y)
	=
	\mathbf P_{K_2}(x,y).
	$
	Then
	$
	\Phi\!\left(
	AI_{K_1}(t)-AI_{K_2}(t)
	\right)=0.
	$
	Since
	\[
	\Phi\!\left(
	\sum_{m\neq0}a_m(t^m-1)
	\right)
	=
	-\sum_{\substack{m\neq0\\m\text{ even}}}
	ma_mx^m
	-
	\sum_{\substack{m\text{ odd}}}
	ma_my^m,
	\]
	and the monomials
	$
	\{x^m\mid m\neq0,\;m\text{ even}\}
	\cup
	\{y^m\mid m\neq0,\;m\text{ odd}\}
	$
	are linearly independent over $\mathbb Z$, it follows that
	$ma_m=0$ for every $m\neq0$, and hence $a_m=0$ for all $m\neq0$.
	Therefore,
	$
	AI_{K_1}(t)-AI_{K_2}(t)=0,
	$
	that is,
	$
	AI_{K_1}(t)=AI_{K_2}(t).
	$
	
	This completes the proof.	
\end{proof}

Although the two-variable parity polynomial $\mathbf{P}_K(x,y)$ is equivalent to the affine index polynomial $AI_K(t)$, under the specialization $x=y=t$, the two polynomials are, in general, not equal. The following example illustrates this.

\begin{example} \label{exp3}
	{\rm
		Let $K_1$ and $K_2$  be oriented virtual knotoids with diagrams $D_1$ and $D_2$ presented in Fig.~\ref{fig514}. 
		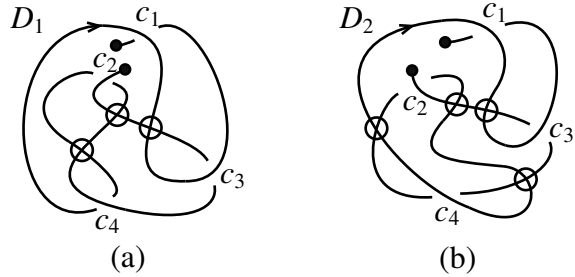
\begin{figure}[!ht]
			\begin{center}
				\tikzset{every picture/.style={line width=1.pt}}  
				\begin{tikzpicture}[x=0.75pt,y=0.75pt,yscale=-1,xscale=1]	
					\draw    (115.14,104.96) .. controls (68,121.01) and (63.39,13.38) .. (118.06,12.81) ;
					\draw    (118.06,12.81) .. controls (187.45,15.54) and (105.67,92.64) .. (159.61,90.92) ;
					\draw    (129.41,34.81) .. controls (80.38,55.91) and (159.89,61.87) .. (170.84,80.67) ;
					\draw    (159.61,90.92) .. controls (201.79,91.67) and (171.87,3.02) .. (146.16,13.67) ;
					\draw    (123.15,41.64) .. controls (157.08,53.83) and (55.04,87.97) .. (129.54,105.2) ; 
					\draw    (113.36,36.87) .. controls (103.64,32.72) and (70.83,51.6) .. (101.2,71.06) ; 
					\draw    (101.2,71.06) .. controls (108.49,75.21) and (132.14,93.76) .. (123.23,98.55) ; 
					\draw    (129.54,105.2) .. controls (162.75,112.54) and (176.56,103.59) .. (174.18,93.04) ;
					\draw   (130.9,57.58) .. controls (130.96,60.57) and (128.57,63.04) .. (125.58,63.09) .. controls (122.59,63.15) and (120.12,60.76) .. (120.07,57.77) .. controls (120.02,54.78) and (122.4,52.31) .. (125.39,52.26) .. controls (128.38,52.21) and (130.85,54.59) .. (130.9,57.58) -- cycle ;
					\draw   (113.18,75.26) .. controls (113.23,78.25) and (110.85,80.72) .. (107.85,80.77) .. controls (104.86,80.82) and (102.39,78.44) .. (102.34,75.45) .. controls (102.29,72.46) and (104.67,69.99) .. (107.67,69.94) .. controls (110.66,69.88) and (113.13,72.27) .. (113.18,75.26) -- cycle ; 
					\draw   (147.65,63.93) .. controls (147.71,66.92) and (145.32,69.39) .. (142.33,69.44) .. controls (139.34,69.49) and (136.87,67.11) .. (136.82,64.12) .. controls (136.77,61.12) and (139.15,58.66) .. (142.14,58.6) .. controls (145.13,58.55) and (147.6,60.93) .. (147.65,63.93) -- cycle ;
					\draw    (124.86,22.8) .. controls (126.77,23.15) and (132.11,21.19) .. (134.11,20.18) ;
					\draw  [fill={rgb, 255:red, 17; green, 16; blue, 16 }  ,fill opacity=1 ] (127.34,22.77) .. controls (127.32,21.41) and (126.2,20.31) .. (124.84,20.32) .. controls (123.47,20.34) and (122.38,21.46) .. (122.39,22.83) .. controls (122.41,24.2) and (123.53,25.3) .. (124.89,25.28) .. controls (126.26,25.26) and (127.35,24.14) .. (127.34,22.77) -- cycle ; 
					\draw  [fill={rgb, 255:red, 17; green, 16; blue, 16 }  ,fill opacity=1 ] (131.88,34.78) .. controls (131.86,33.41) and (130.75,32.31) .. (129.38,32.33) .. controls (128.02,32.34) and (126.92,33.46) .. (126.94,34.83) .. controls (126.95,36.2) and (128.07,37.3) .. (129.44,37.28) .. controls (130.8,37.27) and (131.89,36.15) .. (131.88,34.78) -- cycle ;
					\draw   (110.26,11.32) -- (116.43,12.98) -- (110.75,15.9) ;
					\draw    (301.98,104.13) .. controls (264.01,90.12) and (206.42,19.46) .. (287.98,10.42) ;
					\draw    (287.98,10.42) .. controls (316.61,6.52) and (326.3,27.93) .. (315.71,44.18) .. controls (305.13,60.44) and (308.7,69.46) .. (320.19,72.88) ;
					\draw    (273.38,35.33) .. controls (273.93,56.8) and (305.9,49.01) .. (332.43,62.37) ;
					\draw    (320.19,72.88) .. controls (347.88,81.3) and (360.67,0.64) .. (319.89,12.41) ; 
					\draw    (282.87,39.06) .. controls (298.99,34.96) and (305.86,42.04) .. (288.84,60.74) ;
					\draw    (266.42,47.05) .. controls (255.87,51.85) and (237.49,101.74) .. (283.47,98.92) ; 
					\draw    (301.98,104.13) .. controls (339.6,121.87) and (340.49,85.08) .. (319,83.44) ; 
					\draw    (288.84,60.74) .. controls (272.42,79.7) and (293.62,80.98) .. (319,83.44) ;
					\draw    (289.96,21.18) .. controls (292.8,21.57) and (300.75,19.38) .. (303.74,18.25) ;
					\draw  [fill={rgb, 255:red, 17; green, 16; blue, 16 }  ,fill opacity=1 ] (292.43,21.15) .. controls (292.41,19.78) and (291.3,18.69) .. (289.93,18.7) .. controls (288.57,18.72) and (287.47,19.84) .. (287.49,21.21) .. controls (287.5,22.58) and (288.62,23.68) .. (289.99,23.66) .. controls (291.35,23.64) and (292.44,22.52) .. (292.43,21.15) -- cycle ;
					\draw  [fill={rgb, 255:red, 17; green, 16; blue, 16 }  ,fill opacity=1 ] (275.85,35.3) .. controls (275.83,33.93) and (274.71,32.84) .. (273.35,32.85) .. controls (271.98,32.87) and (270.89,33.99) .. (270.9,35.36) .. controls (270.92,36.73) and (272.04,37.82) .. (273.4,37.81) .. controls (274.77,37.79) and (275.86,36.67) .. (275.85,35.3) -- cycle ;
					\draw   (266.85,12.65) -- (273.23,13.04) -- (268.25,17.04) ;
					\draw    (297.5,97.12) .. controls (329.74,96.61) and (347.34,81.21) .. (342.37,71.22) ;
					\draw   (260.91,64.25) .. controls (260.97,67.24) and (258.58,69.71) .. (255.59,69.76) .. controls (252.6,69.81) and (250.13,67.43) .. (250.08,64.44) .. controls (250.03,61.45) and (252.41,58.98) .. (255.4,58.93) .. controls (258.39,58.87) and (260.86,61.26) .. (260.91,64.25) -- cycle ;
					\draw   (316.09,55.32) .. controls (316.14,58.31) and (313.76,60.78) .. (310.77,60.83) .. controls (307.78,60.88) and (305.31,58.5) .. (305.26,55.51) .. controls (305.2,52.52) and (307.59,50.05) .. (310.58,50) .. controls (313.57,49.94) and (316.04,52.33) .. (316.09,55.32) -- cycle ; 
					\draw   (335.84,89.89) .. controls (335.89,92.88) and (333.51,95.35) .. (330.52,95.4) .. controls (327.53,95.46) and (325.06,93.07) .. (325.01,90.08) .. controls (324.95,87.09) and (327.34,84.62) .. (330.33,84.57) .. controls (333.32,84.52) and (335.79,86.9) .. (335.84,89.89) -- cycle ;
					\draw   (301.2,52.28) .. controls (301.25,55.27) and (298.87,57.74) .. (295.88,57.79) .. controls (292.88,57.84) and (290.42,55.46) .. (290.36,52.47) .. controls (290.31,49.48) and (292.7,47.01) .. (295.69,46.96) .. controls (298.68,46.9) and (301.15,49.29) .. (301.2,52.28) -- cycle ;
					\draw (120.33,121.67) node [anchor=north west][inner sep=0.75pt]    {(a)};
					\draw (285.33,121.67) node [anchor=north west][inner sep=0.75pt]    {(b)};
					\draw (70,2) node [anchor=north west][inner sep=0.75pt]    {$D_{1}$};
					\draw (235,2) node [anchor=north west][inner sep=0.75pt]    {$D_{2}$};
					\draw (133,2) node [anchor=north west][inner sep=0.75pt]    {$c_{1}$};
					\draw (111.5,25) node [anchor=north west][inner sep=0.75pt]    {$c_{2}$};
					\draw (175.5,83) node [anchor=north west][inner sep=0.75pt]    {$c_{3}$};
					\draw (112.5,106) node [anchor=north west][inner sep=0.75pt]    {$c_{4}$};
					\draw (307,0) node [anchor=north west][inner sep=0.75pt]    {$c_{1}$};
					\draw (267,46.5) node [anchor=north west][inner sep=0.75pt]    {$c_{2}$};
					\draw (340.5,60) node [anchor=north west][inner sep=0.75pt]    {$c_{3}$};
					\draw (283,102) node [anchor=north west][inner sep=0.75pt]    {$c_{4}$};	    
				\end{tikzpicture}
				\caption{Virtual knotoid diagrams $D_1$ and $D_2$.} \label{fig514}
			\end{center}
		\end{figure} 
		The diagram $D_1$ contains three virtual crossings and four classical crossings $c_1$, $c_2$, $c_3$ and $c_4$. One can see from Fig.~\ref{fig514}(a) that 
		$$
		\operatorname{sgn}(c_1)=\operatorname{sgn}(c_2)=\operatorname{sgn}(c_3)=-1, \quad \operatorname{sgn}(c_4)=1.
		$$ 
		\begin{figure}[!ht]
			\begin{center}
				\tikzset{every picture/.style={line width=1.0pt}}  
				\begin{tikzpicture}[x=0.75pt,y=0.75pt,yscale=-1.1,xscale=1.1]
					\draw    (118.14,286.96) .. controls (71,303.01) and (66.39,195.38) .. (121.06,194.81) ;
					\draw    (121.06,194.81) .. controls (190.45,197.54) and (108.67,274.64) .. (162.61,272.92) ;
					\draw    (132.41,216.81) .. controls (83.38,237.91) and (162.89,243.87) .. (173.84,262.67) ;
					\draw    (162.61,272.92) .. controls (204.79,273.67) and (174.87,185.02) .. (149.16,195.67) ; 
					\draw    (126.15,223.64) .. controls (160.08,235.83) and (58.04,269.97) .. (132.54,287.2) ; 
					\draw    (116.36,218.87) .. controls (106.64,214.72) and (73.83,233.6) .. (104.2,253.06) ; 
					\draw    (104.2,253.06) .. controls (111.49,257.21) and (135.14,275.76) .. (126.23,280.55) ; 
					\draw    (132.54,287.2) .. controls (165.75,294.54) and (179.56,285.59) .. (177.18,275.04) ; 
					\draw   (133.9,239.58) .. controls (133.96,242.57) and (131.57,245.04) .. (128.58,245.09) .. controls (125.59,245.15) and (123.12,242.76) .. (123.07,239.77) .. controls (123.02,236.78) and (125.4,234.31) .. (128.39,234.26) .. controls (131.38,234.21) and (133.85,236.59) .. (133.9,239.58) -- cycle ; 
					\draw   (116.18,257.26) .. controls (116.23,260.25) and (113.85,262.72) .. (110.85,262.77) .. controls (107.86,262.82) and (105.39,260.44) .. (105.34,257.45) .. controls (105.29,254.46) and (107.67,251.99) .. (110.67,251.94) .. controls (113.66,251.88) and (116.13,254.27) .. (116.18,257.26) -- cycle ; 
					\draw   (150.65,245.93) .. controls (150.71,248.92) and (148.32,251.39) .. (145.33,251.44) .. controls (142.34,251.49) and (139.87,249.11) .. (139.82,246.12) .. controls (139.77,243.12) and (142.15,240.66) .. (145.14,240.6) .. controls (148.13,240.55) and (150.6,242.93) .. (150.65,245.93) -- cycle ;
					\draw    (127.86,204.8) .. controls (129.77,205.15) and (135.11,203.19) .. (137.11,202.18) ; 
					\draw  [fill={rgb, 255:red, 17; green, 16; blue, 16 }  ,fill opacity=1 ] (130.34,204.77) .. controls (130.32,203.41) and (129.2,202.31) .. (127.84,202.32) .. controls (126.47,202.34) and (125.38,203.46) .. (125.39,204.83) .. controls (125.41,206.2) and (126.53,207.3) .. (127.89,207.28) .. controls (129.26,207.26) and (130.35,206.14) .. (130.34,204.77) -- cycle ; 
					\draw  [fill={rgb, 255:red, 17; green, 16; blue, 16 }  ,fill opacity=1 ] (134.88,216.78) .. controls (134.86,215.41) and (133.75,214.31) .. (132.38,214.33) .. controls (131.02,214.34) and (129.92,215.46) .. (129.94,216.83) .. controls (129.95,218.2) and (131.07,219.3) .. (132.44,219.28) .. controls (133.8,219.27) and (134.89,218.15) .. (134.88,216.78) -- cycle ;
					\draw   (113.26,193.32) -- (119.43,194.98) -- (113.75,197.9) ;
					\draw    (304.98,286.13) .. controls (267.01,272.12) and (209.42,201.46) .. (290.98,192.42) ;
					\draw    (290.98,192.42) .. controls (319.61,188.52) and (329.3,209.93) .. (318.71,226.18) .. controls (308.13,242.44) and (311.7,251.46) .. (323.19,254.88) ; 
					\draw    (276.38,217.33) .. controls (276.93,238.8) and (308.9,231.01) .. (335.43,244.37) ;
					\draw    (323.19,254.88) .. controls (350.88,263.3) and (363.67,182.64) .. (322.89,194.41) ; 
					\draw    (285.87,221.06) .. controls (301.99,216.96) and (308.86,224.04) .. (291.84,242.74) ;
					\draw    (269.42,229.05) .. controls (258.87,233.85) and (240.49,283.74) .. (286.47,280.92) ; 
					\draw    (304.98,286.13) .. controls (342.6,303.87) and (343.49,267.08) .. (322,265.44) ; 
					\draw    (291.84,242.74) .. controls (275.42,261.7) and (296.62,262.98) .. (322,265.44) ;
					\draw    (292.96,203.18) .. controls (295.8,203.57) and (303.75,201.38) .. (306.74,200.25) ; 
					\draw  [fill={rgb, 255:red, 17; green, 16; blue, 16 }  ,fill opacity=1 ] (295.43,203.15) .. controls (295.41,201.78) and (294.3,200.69) .. (292.93,200.7) .. controls (291.57,200.72) and (290.47,201.84) .. (290.49,203.21) .. controls (290.5,204.58) and (291.62,205.68) .. (292.99,205.66) .. controls (294.35,205.64) and (295.44,204.52) .. (295.43,203.15) -- cycle ; 
					\draw  [fill={rgb, 255:red, 17; green, 16; blue, 16 }  ,fill opacity=1 ] (278.85,217.3) .. controls (278.83,215.93) and (277.71,214.84) .. (276.35,214.85) .. controls (274.98,214.87) and (273.89,215.99) .. (273.9,217.36) .. controls (273.92,218.73) and (275.04,219.82) .. (276.4,219.81) .. controls (277.77,219.79) and (278.86,218.67) .. (278.85,217.3) -- cycle ;
					\draw   (269.85,194.65) -- (276.23,195.04) -- (271.25,199.04) ;
					\draw    (300.5,279.12) .. controls (332.74,278.61) and (350.34,263.21) .. (345.37,253.22) ; 
					\draw   (263.91,246.25) .. controls (263.97,249.24) and (261.58,251.71) .. (258.59,251.76) .. controls (255.6,251.81) and (253.13,249.43) .. (253.08,246.44) .. controls (253.03,243.45) and (255.41,240.98) .. (258.4,240.93) .. controls (261.39,240.87) and (263.86,243.26) .. (263.91,246.25) -- cycle ; 
					\draw   (319.09,237.32) .. controls (319.14,240.31) and (316.76,242.78) .. (313.77,242.83) .. controls (310.78,242.88) and (308.31,240.5) .. (308.26,237.51) .. controls (308.2,234.52) and (310.59,232.05) .. (313.58,232) .. controls (316.57,231.94) and (319.04,234.33) .. (319.09,237.32) -- cycle ;
					\draw   (338.84,271.89) .. controls (338.89,274.88) and (336.51,277.35) .. (333.52,277.4) .. controls (330.53,277.46) and (328.06,275.07) .. (328.01,272.08) .. controls (327.95,269.09) and (330.34,266.62) .. (333.33,266.57) .. controls (336.32,266.52) and (338.79,268.9) .. (338.84,271.89) -- cycle ;
					\draw   (304.2,234.28) .. controls (304.25,237.27) and (301.87,239.74) .. (298.88,239.79) .. controls (295.88,239.84) and (293.42,237.46) .. (293.36,234.47) .. controls (293.31,231.48) and (295.7,229.01) .. (298.69,228.96) .. controls (301.68,228.9) and (304.15,231.29) .. (304.2,234.28) -- cycle ;
					\draw (118.73,208) node [anchor=north west][inner sep=0.75pt]    {$0$};
					\draw (174.23,199.39) node [anchor=north west][inner sep=0.75pt]    {$1$};
					\draw (116.23,183.14) node [anchor=north west][inner sep=0.75pt]    {$-1$};
					\draw (94.48,206.89) node [anchor=north west][inner sep=0.75pt]    {$-2$};
					\draw (323.48,209) node [anchor=north west][inner sep=0.75pt]    {$0$};
					\draw (292.23,250) node [anchor=north west][inner sep=0.75pt]    {$0$};
					\draw (296.73,203) node [anchor=north west][inner sep=0.75pt]    {$0$};
					\draw (347.98,201.89) node [anchor=north west][inner sep=0.75pt]    {$1$};
					\draw (265.98,216.25) node [anchor=north west][inner sep=0.75pt]    {$0$};
					\draw (232.73,204.89) node [anchor=north west][inner sep=0.75pt]    {$-1$};
					\draw (151.73,208.5) node [anchor=north west][inner sep=0.75pt]    {$0$};
					\draw (269.73,231.14) node [anchor=north west][inner sep=0.75pt]    {$-1$};
					\draw (149.48,278.25) node [anchor=north west][inner sep=0.75pt]    {$0$};
					\draw (248.23,275.64) node [anchor=north west][inner sep=0.75pt]    {$-1$};
					\draw (344.48,256.64) node [anchor=north west][inner sep=0.75pt]    {$-2$};
					\draw (130.48,203.5) node [anchor=north west][inner sep=0.75pt]    {$0$};
					\draw (157.73,238.89) node [anchor=north west][inner sep=0.75pt]    {$1$};
					\draw (86.98,264.39) node [anchor=north west][inner sep=0.75pt]    {$-1$};
					\draw (119,299.13) node [anchor=north west][inner sep=0.75pt]    {(a)};
					\draw (291,298.28) node [anchor=north west][inner sep=0.75pt]    {(b)};
					\draw (75,180) node [anchor=north west][inner sep=0.75pt]    {$D_1$};
					\draw (245,180) node [anchor=north west][inner sep=0.75pt]    {$D_2$};
				\end{tikzpicture}
				\caption{Integer labeling for $D_1$ and $D_2$, respectively.} \label{fig515}
			\end{center}
		\end{figure}
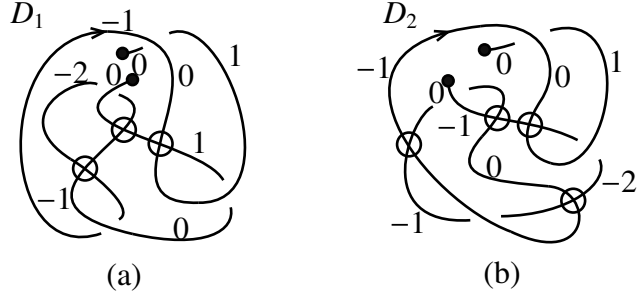 
		Based on Fig.~\ref{fig515}(a), we have 
		$$
		\begin{gathered}
			W_{D_1}(c_1)=-(1-0)=-1, \quad W_{D_1}(c_2)=-(-1-1)=2, \cr 
			W_{D_1}(c_3)=-(1-1)=0, \quad W_{D_1}(c_4)=0-(-1)=1.
		\end{gathered}
		$$
		The diagram $D_2$ contains four virtual crossings and four classical crossings $c_1$, $c_2$, $c_3$ and $c_4$. One can see from Fig.~\ref{fig514}(b) that 
		$$
		\operatorname{sgn}(c_1)=\operatorname{sgn}(c_3)=-1, \quad \operatorname{sgn}(c_2)=\operatorname{sgn}(c_4)=1.
		$$ 
		Based on Fig.~\ref{fig515}(b), we have 
		$$
		\begin{gathered} 
			W_{D_2}(c_1)=-(1-0)=-1, \quad W_{D_2}(c_2)=0-0=0, \cr 
			W_{D_2}(c_3)=-(-1-1)=2, \quad W_{D_2}(c_4)=0-(-1)=1.
		\end{gathered}
		$$
		
		Then we have $AI_{K_1}(t)=AI_{K_2}(t)=-t^2+t-t^{-1}+1$.
		
		Now we compute $\mathbf{P}_{K_1}(x,y)$ and $\mathbf{P}_{K_2}(x,y)$. 
		From Fig.~\ref{fig516}, we obtain $I(c_1)=I(c_4)=1$, $I(c_2)=I(c_3)=2$, thus $c_1,c_4\in O(D_1)$ and $c_2,c_3\in E(D_1)$. Moreover, $i(c_1)=1$, $i(c_2)=-2$, $i(c_3)=0$, $i(c_4)=-1$. Therefore, 
		$$
		\mathbf{P}_{K_1}(x,y)=2x^2-y-y^{-1}.
		$$

		From Fig.~\ref{fig517}, we obtain $I(c_1)=I(c_4)=1$, $I(c_2)=I(c_3)=2$, thus $c_1,c_4\in O(D_2)$ and $c_2,c_3\in E(D_2)$. Moreover, $i(c_1)=1$, $i(c_2)=0$, $i(c_3)=-2$, $i(c_4)=-1$. Therefore, 
		$$
		\mathbf{P}_{K_2}(x,y)=2x^2-y-y^{-1}.
		$$
		\begin{figure}[!ht]
			\centering
				\tikzset{every picture/.style={line width=1.0pt}}  
				\begin{tikzpicture}[x=0.75pt,y=0.75pt,yscale=-1.0,xscale=1.0]
					\draw [color={rgb, 255:red, 248; green, 72; blue, 61 }  ,draw opacity=1 ]   (70.47,131.63) .. controls (23.34,147.68) and (18.73,40.05) .. (73.39,39.48) ;
					\draw    (99.49,51.86) .. controls (128.49,69.38) and (61,119.31) .. (114.94,117.58) ;
					\draw [color={rgb, 255:red, 248; green, 72; blue, 61 }  ,draw opacity=1 ]   (84.74,61.47) .. controls (35.72,82.58) and (115.23,88.54) .. (126.18,107.33) ;
					\draw    (114.94,117.58) .. controls (157.13,118.34) and (127.21,29.68) .. (101.5,40.34) ;
					\draw [color={rgb, 255:red, 248; green, 72; blue, 61 }  ,draw opacity=1 ]   (78.48,68.31) .. controls (112.41,80.5) and (10.37,114.64) .. (84.88,131.87) ;
					\draw [color={rgb, 255:red, 248; green, 72; blue, 61 }  ,draw opacity=1 ]   (68.69,63.53) .. controls (58.97,59.39) and (26.16,78.26) .. (56.53,97.73) ;
					\draw [color={rgb, 255:red, 248; green, 72; blue, 61 }  ,draw opacity=1 ]   (56.53,97.73) .. controls (63.82,101.87) and (87.47,120.43) .. (78.57,125.21) ;
					\draw [color={rgb, 255:red, 248; green, 72; blue, 61 }  ,draw opacity=1 ]   (84.88,131.87) .. controls (118.08,139.2) and (131.89,130.26) .. (129.51,119.71) ;
					\draw  [color={rgb, 255:red, 248; green, 72; blue, 61 }  ,draw opacity=1 ] (86.24,84.25) .. controls (86.29,87.24) and (83.91,89.71) .. (80.91,89.76) .. controls (77.92,89.81) and (75.45,87.43) .. (75.4,84.44) .. controls (75.35,81.45) and (77.73,78.98) .. (80.72,78.93) .. controls (83.72,78.87) and (86.18,81.26) .. (86.24,84.25) -- cycle ;
					\draw  [color={rgb, 255:red, 248; green, 72; blue, 61 }  ,draw opacity=1 ] (68.51,101.93) .. controls (68.56,104.92) and (66.18,107.39) .. (63.19,107.44) .. controls (60.2,107.49) and (57.73,105.11) .. (57.68,102.12) .. controls (57.62,99.12) and (60.01,96.66) .. (63,96.6) .. controls (65.99,96.55) and (68.46,98.93) .. (68.51,101.93) -- cycle ;
					\draw   (102.99,90.59) .. controls (103.04,93.59) and (100.66,96.05) .. (97.66,96.11) .. controls (94.67,96.16) and (92.2,93.77) .. (92.15,90.78) .. controls (92.1,87.79) and (94.48,85.32) .. (97.47,85.27) .. controls (100.47,85.22) and (102.93,87.6) .. (102.99,90.59) -- cycle ; 
					\draw [color={rgb, 255:red, 248; green, 72; blue, 61 }  ,draw opacity=1 ]   (80.2,49.47) .. controls (82.1,49.81) and (87.44,47.86) .. (89.45,46.85) ;
					\draw  [color={rgb, 255:red, 248; green, 72; blue, 61 }  ,draw opacity=1 ][fill={rgb, 255:red, 248; green, 72; blue, 61 }  ,fill opacity=1 ] (82.67,49.44) .. controls (82.65,48.07) and (81.53,46.98) .. (80.17,46.99) .. controls (78.81,47.01) and (77.71,48.13) .. (77.73,49.5) .. controls (77.74,50.87) and (78.86,51.96) .. (80.23,51.95) .. controls (81.59,51.93) and (82.68,50.81) .. (82.67,49.44) -- cycle ;
					\draw  [color={rgb, 255:red, 248; green, 72; blue, 61 }  ,draw opacity=1 ][fill={rgb, 255:red, 248; green, 72; blue, 61 }  ,fill opacity=1 ] (87.21,61.44) .. controls (87.2,60.08) and (86.08,58.98) .. (84.71,58.99) .. controls (83.35,59.01) and (82.26,60.13) .. (82.27,61.5) .. controls (82.29,62.87) and (83.4,63.97) .. (84.77,63.95) .. controls (86.13,63.93) and (87.23,62.81) .. (87.21,61.44) -- cycle ;
					\draw   (65.59,37.98) -- (71.77,39.64) -- (66.08,42.57) ;
					\draw [color={rgb, 255:red, 248; green, 72; blue, 61 }  ,draw opacity=1 ]   (194.47,130.63) .. controls (147.34,146.68) and (142.73,39.05) .. (197.39,38.48) ;
					\draw [color={rgb, 255:red, 248; green, 72; blue, 61 }  ,draw opacity=1 ]   (197.39,38.48) .. controls (266.78,41.21) and (185,118.31) .. (238.94,116.58) ;
					\draw    (197.74,69.19) .. controls (177.88,83.76) and (239.23,87.54) .. (250.18,106.33) ; 
					\draw [color={rgb, 255:red, 248; green, 72; blue, 61 }  ,draw opacity=1 ]   (238.94,116.58) .. controls (281.13,117.34) and (251.21,28.68) .. (225.5,39.34) ;
					\draw    (202.48,67.31) .. controls (236.41,79.5) and (134.37,113.64) .. (208.88,130.87) ;
					\draw [color={rgb, 255:red, 248; green, 72; blue, 61 }  ,draw opacity=1 ]   (192.69,62.53) .. controls (177.02,61) and (150.16,77.26) .. (180.53,96.73) ; 
					\draw [color={rgb, 255:red, 248; green, 72; blue, 61 }  ,draw opacity=1 ]   (180.53,96.73) .. controls (187.82,100.87) and (211.47,119.43) .. (202.57,124.21) ;
					\draw    (208.88,130.87) .. controls (242.08,138.2) and (255.89,129.26) .. (253.51,118.71) ;
					\draw   (210.24,83.25) .. controls (210.29,86.24) and (207.91,88.71) .. (204.91,88.76) .. controls (201.92,88.81) and (199.45,86.43) .. (199.4,83.44) .. controls (199.35,80.45) and (201.73,77.98) .. (204.72,77.93) .. controls (207.72,77.87) and (210.18,80.26) .. (210.24,83.25) -- cycle ;
					\draw   (192.51,100.93) .. controls (192.56,103.92) and (190.18,106.39) .. (187.19,106.44) .. controls (184.2,106.49) and (181.73,104.11) .. (181.68,101.12) .. controls (181.62,98.12) and (184.01,95.66) .. (187,95.6) .. controls (189.99,95.55) and (192.46,97.93) .. (192.51,100.93) -- cycle ;
					\draw   (226.99,89.59) .. controls (227.04,92.59) and (224.66,95.05) .. (221.66,95.11) .. controls (218.67,95.16) and (216.2,92.77) .. (216.15,89.78) .. controls (216.1,86.79) and (218.48,84.32) .. (221.47,84.27) .. controls (224.47,84.22) and (226.93,86.6) .. (226.99,89.59) -- cycle ;
					\draw [color={rgb, 255:red, 248; green, 72; blue, 61 }  ,draw opacity=1 ][fill={rgb, 255:red, 248; green, 72; blue, 61 }  ,fill opacity=1 ]   (204.2,48.47) .. controls (206.1,48.81) and (211.44,46.86) .. (213.45,45.85) ;
					\draw  [color={rgb, 255:red, 248; green, 72; blue, 61 }  ,draw opacity=1 ][fill={rgb, 255:red, 248; green, 72; blue, 61 }  ,fill opacity=1 ] (206.67,48.44) .. controls (206.65,47.07) and (205.53,45.98) .. (204.17,45.99) .. controls (202.81,46.01) and (201.71,47.13) .. (201.73,48.5) .. controls (201.74,49.87) and (202.86,50.96) .. (204.23,50.95) .. controls (205.59,50.93) and (206.68,49.81) .. (206.67,48.44) -- cycle ; 
					\draw  [color={rgb, 255:red, 248; green, 72; blue, 61 }  ,draw opacity=1 ][fill={rgb, 255:red, 248; green, 72; blue, 61 }  ,fill opacity=1 ] (211.21,60.44) .. controls (211.2,59.08) and (210.08,57.98) .. (208.71,57.99) .. controls (207.35,58.01) and (206.26,59.13) .. (206.27,60.5) .. controls (206.29,61.87) and (207.4,62.97) .. (208.77,62.95) .. controls (210.13,62.93) and (211.23,61.81) .. (211.21,60.44) -- cycle ;
					\draw   (185.24,64.9) -- (178.84,64.91) -- (183.57,60.6) ; 
					\draw [color={rgb, 255:red, 248; green, 72; blue, 61 }  ,draw opacity=1 ]   (318.63,131.3) .. controls (271.5,147.34) and (266.89,39.72) .. (321.55,39.14) ; 
					\draw [color={rgb, 255:red, 248; green, 72; blue, 61 }  ,draw opacity=1 ]   (321.55,39.14) .. controls (390.94,41.88) and (309.16,118.97) .. (363.1,117.25) ; 
					\draw    (332.9,61.14) .. controls (288.13,78.08) and (367.4,104.44) .. (374.34,107) ;
					\draw    (374.34,107) .. controls (395.77,114.26) and (375.37,29.35) .. (349.66,40.01) ; 
					\draw [color={rgb, 255:red, 248; green, 72; blue, 61 }  ,draw opacity=1 ]   (326.64,67.98) .. controls (360.57,80.16) and (258.53,114.3) .. (333.04,131.53) ;
					\draw [color={rgb, 255:red, 248; green, 72; blue, 61 }  ,draw opacity=1 ]   (316.85,63.2) .. controls (307.13,59.05) and (274.32,77.93) .. (304.69,97.39) ;
					\draw [color={rgb, 255:red, 248; green, 72; blue, 61 }  ,draw opacity=1 ]   (304.69,97.39) .. controls (311.98,101.54) and (335.63,120.09) .. (326.73,124.88) ; 
					\draw [color={rgb, 255:red, 248; green, 72; blue, 61 }  ,draw opacity=1 ]   (333.04,131.53) .. controls (366.24,138.87) and (380.05,129.92) .. (377.67,119.38) ;
					\draw   (332.76,84.64) .. controls (332.81,87.63) and (330.43,90.1) .. (327.44,90.15) .. controls (324.45,90.21) and (321.98,87.82) .. (321.93,84.83) .. controls (321.87,81.84) and (324.26,79.37) .. (327.25,79.32) .. controls (330.24,79.27) and (332.71,81.65) .. (332.76,84.64) -- cycle ;
					\draw  [color={rgb, 255:red, 248; green, 72; blue, 61 }  ,draw opacity=1 ] (316.67,101.53) .. controls (316.72,104.52) and (314.34,106.99) .. (311.35,107.04) .. controls (308.36,107.09) and (305.89,104.71) .. (305.84,101.72) .. controls (305.78,98.72) and (308.17,96.26) .. (311.16,96.2) .. controls (314.15,96.15) and (316.62,98.54) .. (316.67,101.53) -- cycle ;
					\draw   (349.88,94.44) .. controls (349.93,97.43) and (347.54,99.9) .. (344.55,99.95) .. controls (341.56,100.01) and (339.09,97.62) .. (339.04,94.63) .. controls (338.99,91.64) and (341.37,89.17) .. (344.36,89.12) .. controls (347.35,89.07) and (349.82,91.45) .. (349.88,94.44) -- cycle ;
					\draw    (328.36,49.14) .. controls (330.26,49.48) and (335.6,47.53) .. (337.61,46.52) ;
					\draw  [fill={rgb, 255:red, 17; green, 16; blue, 16 }  ,fill opacity=1 ] (330.83,49.11) .. controls (330.81,47.74) and (329.7,46.64) .. (328.33,46.66) .. controls (326.97,46.67) and (325.87,47.79) .. (325.89,49.16) .. controls (325.9,50.53) and (327.02,51.63) .. (328.39,51.61) .. controls (329.75,51.6) and (330.84,50.48) .. (330.83,49.11) -- cycle ;
					\draw  [fill={rgb, 255:red, 17; green, 16; blue, 16 }  ,fill opacity=1 ] (335.37,61.11) .. controls (335.36,59.74) and (334.24,58.65) .. (332.87,58.66) .. controls (331.51,58.68) and (330.42,59.8) .. (330.43,61.17) .. controls (330.45,62.54) and (331.57,63.63) .. (332.93,63.62) .. controls (334.29,63.6) and (335.39,62.48) .. (335.37,61.11) -- cycle ;
					\draw   (313.75,37.45) -- (319.93,39.11) -- (314.24,42.04) ;
					\draw [color={rgb, 255:red, 248; green, 72; blue, 61 }  ,draw opacity=1 ]   (442.63,131.63) .. controls (395.5,147.68) and (390.89,40.05) .. (445.55,39.48) ; 
					\draw [color={rgb, 255:red, 248; green, 72; blue, 61 }  ,draw opacity=1 ]   (445.55,39.48) .. controls (514.94,42.21) and (433.16,119.31) .. (487.1,117.58) ;
					\draw [color={rgb, 255:red, 248; green, 72; blue, 61 }  ,draw opacity=1 ]   (456.9,61.47) .. controls (407.88,82.58) and (487.39,88.54) .. (498.34,107.33) ; 
					\draw [color={rgb, 255:red, 248; green, 72; blue, 61 }  ,draw opacity=1 ]   (487.1,117.58) .. controls (529.29,118.34) and (499.37,29.68) .. (473.66,40.34) ;
					\draw    (450.64,68.31) .. controls (484.57,80.5) and (387.36,119.36) .. (444.16,125.76) ;
					\draw    (440.85,63.53) .. controls (431.13,59.39) and (398.32,78.26) .. (428.69,97.73) ;
					\draw    (428.69,97.73) .. controls (435.98,101.87) and (459.63,120.43) .. (450.73,125.21) ;
					\draw [color={rgb, 255:red, 248; green, 72; blue, 61 }  ,draw opacity=1 ]   (457.04,131.87) .. controls (490.24,139.2) and (504.05,130.26) .. (501.67,119.71) ;
					\draw   (458.4,84.25) .. controls (458.45,87.24) and (456.07,89.71) .. (453.07,89.76) .. controls (450.08,89.81) and (447.61,87.43) .. (447.56,84.44) .. controls (447.51,81.45) and (449.89,78.98) .. (452.89,78.93) .. controls (455.88,78.87) and (458.35,81.26) .. (458.4,84.25) -- cycle ;
					\draw   (440.67,101.93) .. controls (440.72,104.92) and (438.34,107.39) .. (435.35,107.44) .. controls (432.36,107.49) and (429.89,105.11) .. (429.84,102.12) .. controls (429.78,99.12) and (432.17,96.66) .. (435.16,96.6) .. controls (438.15,96.55) and (440.62,98.93) .. (440.67,101.93) -- cycle ;
					\draw  [color={rgb, 255:red, 248; green, 72; blue, 61 }  ,draw opacity=1 ] (475.15,90.59) .. controls (475.2,93.59) and (472.82,96.05) .. (469.82,96.11) .. controls (466.83,96.16) and (464.36,93.77) .. (464.31,90.78) .. controls (464.26,87.79) and (466.64,85.32) .. (469.64,85.27) .. controls (472.63,85.22) and (475.1,87.6) .. (475.15,90.59) -- cycle ;
					\draw [color={rgb, 255:red, 248; green, 72; blue, 61 }  ,draw opacity=1 ]   (452.36,49.47) .. controls (454.26,49.81) and (459.6,47.86) .. (461.61,46.85) ;
					\draw  [color={rgb, 255:red, 248; green, 72; blue, 61 }  ,draw opacity=1 ][fill={rgb, 255:red, 248; green, 72; blue, 61 }  ,fill opacity=1 ] (454.83,49.44) .. controls (454.81,48.07) and (453.7,46.98) .. (452.33,46.99) .. controls (450.97,47.01) and (449.87,48.13) .. (449.89,49.5) .. controls (449.9,50.87) and (451.02,51.96) .. (452.39,51.95) .. controls (453.75,51.93) and (454.84,50.81) .. (454.83,49.44) -- cycle ; 
					\draw  [color={rgb, 255:red, 248; green, 72; blue, 61 }  ,draw opacity=1 ][fill={rgb, 255:red, 248; green, 72; blue, 61 }  ,fill opacity=1 ] (459.37,61.44) .. controls (459.36,60.08) and (458.24,58.98) .. (456.87,58.99) .. controls (455.51,59.01) and (454.42,60.13) .. (454.43,61.5) .. controls (454.45,62.87) and (455.57,63.97) .. (456.93,63.95) .. controls (458.29,63.93) and (459.39,62.81) .. (459.37,61.44) -- cycle ;
					\draw   (437.75,37.98) -- (443.93,39.64) -- (438.24,42.57) ;
					\draw [color={rgb, 255:red, 248; green, 72; blue, 61 }  ,draw opacity=1 ]   (73.39,39.48) .. controls (95.76,37.75) and (92.53,45.36) .. (89.45,46.85) ;
					\draw    (101.5,40.34) .. controls (94.2,43.64) and (94.43,49.29) .. (99.49,51.86) ;
					\draw   (108.44,58.64) -- (107.15,64.9) -- (103.89,59.4) ;
					\draw    (197.74,69.19) .. controls (198.44,68.6) and (200.62,66.69) .. (202.48,67.31) ;
					\draw [color={rgb, 255:red, 248; green, 72; blue, 61 }  ,draw opacity=1 ]   (192.69,62.53) .. controls (198.72,63) and (199.93,63.1) .. (202.84,62.4) .. controls (205.74,61.7) and (208.11,60.76) .. (208.74,60.47) ;
					\draw   (201.3,69.52) -- (195.2,71.46) -- (198.41,65.93) ;
					\draw [color={rgb, 255:red, 248; green, 72; blue, 61 }  ,draw opacity=1 ]   (363.1,117.25) .. controls (369.04,116.02) and (377.22,116.2) .. (377.67,119.38) ;
					\draw   (360.74,99.32) -- (365.3,103.81) -- (358.91,103.56) ; 
					\draw    (444.16,125.76) .. controls (445.82,125.91) and (449.35,125.91) .. (450.73,125.21) ; 
					\draw [color={rgb, 255:red, 248; green, 72; blue, 61 }  ,draw opacity=1 ]   (442.63,131.63) .. controls (447.01,130.29) and (453.66,130.75) .. (457.04,131.87) ;
					\draw   (441.62,127.88) -- (436.18,124.51) -- (442.47,123.35) ;
					\draw (67.8,24.82) node [anchor=north west][inner sep=0.75pt]  [font=\small]  {$D_{1_1}$};
					\draw (98.8,24.8) node [anchor=north west][inner sep=0.75pt]  [font=\small]  {$D_{1_2}$};
					\draw (130,109) node [anchor=north west][inner sep=0.75pt]  [font=\small]  {$+$};
					\draw (183.11,46.44) node [anchor=north west][inner sep=0.75pt]  [font=\small]  {$D_{1_1}$};
					\draw (178.85,72.66) node [anchor=north west][inner sep=0.75pt]  [font=\small]  {$D_{1_2}$};
					\draw (254.7,106.55) node [anchor=north west][inner sep=0.75pt]  [font=\small]  {$-$};
					\draw (192.33,129.02) node [anchor=north west][inner sep=0.75pt]  [font=\small]  {$-$};
					\draw (308.76,64.93) node [anchor=north west][inner sep=0.75pt]  [font=\small]  {$+$};
					\draw (337.3,29.39) node [anchor=north west][inner sep=0.75pt]  [font=\small]  {$-$};
					\draw (377.67,119.38) node [anchor=north west][inner sep=0.75pt]  [font=\small]  {$D_{1_1}$};
					\draw (383.03,93.67) node [anchor=north west][inner sep=0.75pt]  [font=\small]  {$D_{1_2}$};
					\draw (440.75,130.88) node [anchor=north west][inner sep=0.75pt]  [font=\small]  {$D_{1_1}$};
					\draw (453.49,109.32) node [anchor=north west][inner sep=0.75pt]  [font=\small]  {$D_{1_2}$};
					\draw (430.7,64.12) node [anchor=north west][inner sep=0.75pt]  [font=\small]  {$-$};
					\draw (69,145.85) node [anchor=north west][inner sep=0.75pt]   [align=left] {(a)};
					\draw (453,145.85) node [anchor=north west][inner sep=0.75pt]   [align=left] {(d)};
					\draw (326,145.85) node [anchor=north west][inner sep=0.75pt]   [align=left] {(c)};
					\draw (197,145.85) node [anchor=north west][inner sep=0.75pt]   [align=left] {(b)};
				\end{tikzpicture}
			\caption{(a)--(d): Diagrams obtained by orientation-preserving smoothing at classical crossings $c_1$, $c_2$, $c_3$ and $c_4$ in $D_1$.} \label{fig516}
		\end{figure}
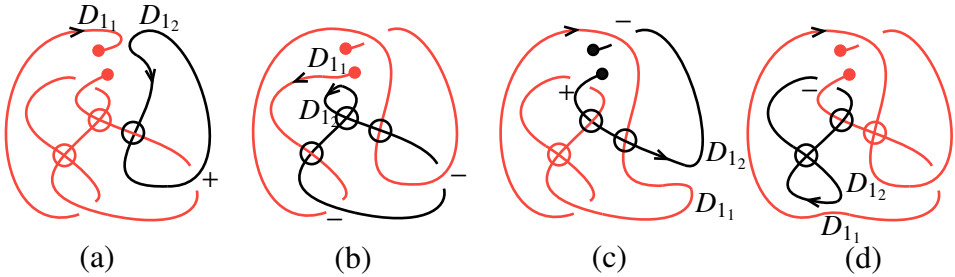
		
		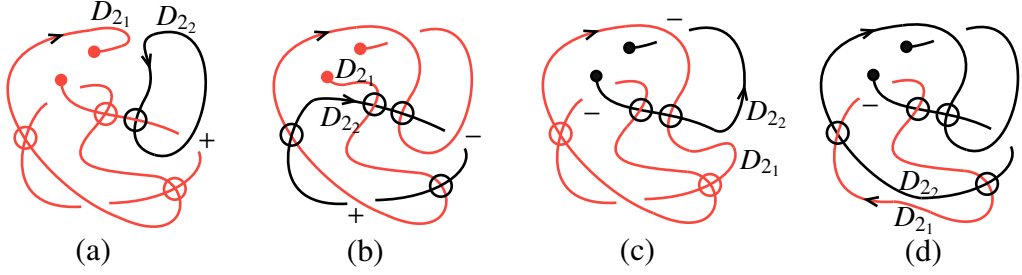
\begin{figure}[!ht]
			\centering
				\tikzset{every picture/.style={line width=1.0pt}} 
				\begin{tikzpicture}[x=0.75pt,y=0.75pt,yscale=-1.0,xscale=1.0]
					\draw [color={rgb, 255:red, 248; green, 72; blue, 61 }  ,draw opacity=1 ]   (90.98,266.63) .. controls (53.01,252.62) and (-4.58,181.96) .. (76.98,172.92) ;
					\draw    (108.89,174.91) .. controls (96.65,179.65) and (113.32,190.98) .. (104.71,206.68) .. controls (96.1,222.38) and (97.7,231.96) .. (109.19,235.38) ;
					\draw [color={rgb, 255:red, 248; green, 72; blue, 61 }  ,draw opacity=1 ]   (62.38,197.83) .. controls (62.93,219.3) and (94.9,211.51) .. (121.43,224.87) ;
					\draw    (109.19,235.38) .. controls (136.88,243.8) and (149.67,163.14) .. (108.89,174.91) ;
					\draw [color={rgb, 255:red, 248; green, 72; blue, 61 }  ,draw opacity=1 ]   (71.87,201.56) .. controls (87.99,197.46) and (94.86,204.54) .. (77.84,223.24) ;
					\draw [color={rgb, 255:red, 248; green, 72; blue, 61 }  ,draw opacity=1 ]   (55.42,209.55) .. controls (44.87,214.35) and (26.49,264.24) .. (72.47,261.42) ;
					\draw [color={rgb, 255:red, 248; green, 72; blue, 61 }  ,draw opacity=1 ]   (90.98,266.63) .. controls (128.6,284.37) and (129.49,247.58) .. (108,245.94) ;
					\draw [color={rgb, 255:red, 248; green, 72; blue, 61 }  ,draw opacity=1 ]   (77.84,223.24) .. controls (61.42,242.2) and (82.62,243.48) .. (108,245.94) ;
					\draw [color={rgb, 255:red, 248; green, 72; blue, 61 }  ,draw opacity=1 ]   (78.96,183.68) .. controls (81.8,184.07) and (89.75,181.88) .. (92.74,180.75) ;
					\draw  [color={rgb, 255:red, 248; green, 72; blue, 61 }  ,draw opacity=1 ][fill={rgb, 255:red, 248; green, 72; blue, 61 }  ,fill opacity=1 ] (81.43,183.65) .. controls (81.41,182.28) and (80.3,181.19) .. (78.93,181.2) .. controls (77.57,181.22) and (76.47,182.34) .. (76.49,183.71) .. controls (76.5,185.08) and (77.62,186.18) .. (78.99,186.16) .. controls (80.35,186.14) and (81.44,185.02) .. (81.43,183.65) -- cycle ;
					\draw  [color={rgb, 255:red, 248; green, 72; blue, 61 }  ,draw opacity=1 ][fill={rgb, 255:red, 248; green, 72; blue, 61 }  ,fill opacity=1 ] (64.85,197.8) .. controls (64.83,196.43) and (63.71,195.34) .. (62.35,195.35) .. controls (60.98,195.37) and (59.89,196.49) .. (59.9,197.86) .. controls (59.92,199.23) and (61.04,200.32) .. (62.4,200.31) .. controls (63.77,200.29) and (64.86,199.17) .. (64.85,197.8) -- cycle ;
					\draw   (55.85,175.15) -- (62.23,175.54) -- (57.25,179.54) ;
					\draw [color={rgb, 255:red, 248; green, 72; blue, 61 }  ,draw opacity=1 ]   (86.5,259.62) .. controls (118.74,259.11) and (136.34,243.71) .. (131.37,233.72) ;
					\draw  [color={rgb, 255:red, 248; green, 72; blue, 61 }  ,draw opacity=1 ] (49.91,226.75) .. controls (49.97,229.74) and (47.58,232.21) .. (44.59,232.26) .. controls (41.6,232.31) and (39.13,229.93) .. (39.08,226.94) .. controls (39.03,223.95) and (41.41,221.48) .. (44.4,221.43) .. controls (47.39,221.37) and (49.86,223.76) .. (49.91,226.75) -- cycle ;
					\draw   (105.09,217.82) .. controls (105.14,220.81) and (102.76,223.28) .. (99.77,223.33) .. controls (96.78,223.38) and (94.31,221) .. (94.26,218.01) .. controls (94.2,215.02) and (96.59,212.55) .. (99.58,212.5) .. controls (102.57,212.44) and (105.04,214.83) .. (105.09,217.82) -- cycle ;
					\draw  [color={rgb, 255:red, 248; green, 72; blue, 61 }  ,draw opacity=1 ] (124.84,252.39) .. controls (124.89,255.38) and (122.51,257.85) .. (119.52,257.9) .. controls (116.53,257.96) and (114.06,255.57) .. (114.01,252.58) .. controls (113.95,249.59) and (116.34,247.12) .. (119.33,247.07) .. controls (122.32,247.02) and (124.79,249.4) .. (124.84,252.39) -- cycle ;
					\draw  [color={rgb, 255:red, 248; green, 72; blue, 61 }  ,draw opacity=1 ] (90.2,214.78) .. controls (90.25,217.77) and (87.87,220.24) .. (84.88,220.29) .. controls (81.88,220.34) and (79.42,217.96) .. (79.36,214.97) .. controls (79.31,211.98) and (81.7,209.51) .. (84.69,209.46) .. controls (87.68,209.4) and (90.15,211.79) .. (90.2,214.78) -- cycle ;
					\draw [color={rgb, 255:red, 248; green, 72; blue, 61 }  ,draw opacity=1 ]   (224.48,265.13) .. controls (186.51,251.12) and (128.92,180.46) .. (210.48,171.42) ;
					\draw [color={rgb, 255:red, 248; green, 72; blue, 61 }  ,draw opacity=1 ]   (210.48,171.42) .. controls (239.11,167.52) and (248.8,188.93) .. (238.21,205.18) .. controls (227.63,221.44) and (231.2,230.46) .. (242.69,233.88) ;
					\draw    (188.92,208.05) .. controls (195.32,204.79) and (228.4,210.01) .. (254.93,223.37) ;
					\draw [color={rgb, 255:red, 248; green, 72; blue, 61 }  ,draw opacity=1 ]   (242.69,233.88) .. controls (270.38,242.3) and (283.17,161.64) .. (242.39,173.41) ;
					\draw [color={rgb, 255:red, 248; green, 72; blue, 61 }  ,draw opacity=1 ]   (205.37,200.06) .. controls (221.49,195.96) and (228.36,203.04) .. (211.34,221.74) ;
					\draw    (188.92,208.05) .. controls (178.37,212.85) and (159.99,262.74) .. (205.97,259.92) ;
					\draw [color={rgb, 255:red, 248; green, 72; blue, 61 }  ,draw opacity=1 ]   (224.48,265.13) .. controls (262.1,282.87) and (262.99,246.08) .. (241.5,244.44) ;
					\draw [color={rgb, 255:red, 248; green, 72; blue, 61 }  ,draw opacity=1 ]   (211.34,221.74) .. controls (194.92,240.7) and (216.12,241.98) .. (241.5,244.44) ;
					\draw [color={rgb, 255:red, 248; green, 72; blue, 61 }  ,draw opacity=1 ][fill={rgb, 255:red, 74; green, 144; blue, 226 }  ,fill opacity=1 ]   (212.46,182.18) .. controls (215.3,182.57) and (223.25,180.38) .. (226.24,179.25) ;
					\draw  [color={rgb, 255:red, 248; green, 72; blue, 61 }  ,draw opacity=1 ][fill={rgb, 255:red, 248; green, 72; blue, 61 }  ,fill opacity=1 ] (214.93,182.15) .. controls (214.91,180.78) and (213.8,179.69) .. (212.43,179.7) .. controls (211.07,179.72) and (209.97,180.84) .. (209.99,182.21) .. controls (210,183.58) and (211.12,184.68) .. (212.49,184.66) .. controls (213.85,184.64) and (214.94,183.52) .. (214.93,182.15) -- cycle ; 
					\draw  [color={rgb, 255:red, 248; green, 72; blue, 61 }  ,draw opacity=1 ][fill={rgb, 255:red, 248; green, 72; blue, 61 }  ,fill opacity=1 ] (198.35,196.3) .. controls (198.33,194.93) and (197.21,193.84) .. (195.85,193.85) .. controls (194.48,193.87) and (193.39,194.99) .. (193.4,196.36) .. controls (193.42,197.73) and (194.54,198.82) .. (195.9,198.81) .. controls (197.27,198.79) and (198.36,197.67) .. (198.35,196.3) -- cycle ;
					\draw   (189.35,173.65) -- (195.73,174.04) -- (190.75,178.04) ;
					\draw    (220,258.12) .. controls (252.24,257.61) and (269.84,242.21) .. (264.87,232.22) ;
					\draw   (183.41,225.25) .. controls (183.47,228.24) and (181.08,230.71) .. (178.09,230.76) .. controls (175.1,230.81) and (172.63,228.43) .. (172.58,225.44) .. controls (172.53,222.45) and (174.91,219.98) .. (177.9,219.93) .. controls (180.89,219.87) and (183.36,222.26) .. (183.41,225.25) -- cycle ;
					\draw   (239.04,214.32) .. controls (239.09,217.31) and (236.71,219.78) .. (233.71,219.83) .. controls (230.72,219.88) and (228.25,217.5) .. (228.2,214.51) .. controls (228.15,211.52) and (230.53,209.05) .. (233.52,209) .. controls (236.52,208.94) and (238.98,211.33) .. (239.04,214.32) -- cycle ;
					\draw   (258.34,250.89) .. controls (258.39,253.88) and (256.01,256.35) .. (253.02,256.4) .. controls (250.03,256.46) and (247.56,254.07) .. (247.51,251.08) .. controls (247.45,248.09) and (249.84,245.62) .. (252.83,245.57) .. controls (255.82,245.52) and (258.29,247.9) .. (258.34,250.89) -- cycle ;
					\draw   (225.03,209.95) .. controls (225.08,212.94) and (222.7,215.41) .. (219.71,215.46) .. controls (216.72,215.51) and (214.25,213.13) .. (214.2,210.14) .. controls (214.15,207.14) and (216.53,204.68) .. (219.52,204.62) .. controls (222.51,204.57) and (224.98,206.95) .. (225.03,209.95) -- cycle ;
					\draw [color={rgb, 255:red, 248; green, 72; blue, 61 }  ,draw opacity=1 ]   (359.48,264.63) .. controls (321.51,250.62) and (263.92,179.96) .. (345.48,170.92) ;
					\draw [color={rgb, 255:red, 248; green, 72; blue, 61 }  ,draw opacity=1 ]   (345.48,170.92) .. controls (374.11,167.02) and (383.8,188.43) .. (373.21,204.68) .. controls (362.63,220.94) and (366.2,229.96) .. (377.69,233.38) ;
					\draw    (330.88,195.83) .. controls (331.43,217.3) and (363.4,209.51) .. (389.93,222.87) ;
					\draw    (389.93,222.87) .. controls (406.76,231.93) and (418.17,161.14) .. (377.39,172.91) ;
					\draw [color={rgb, 255:red, 248; green, 72; blue, 61 }  ,draw opacity=1 ]   (340.37,199.56) .. controls (356.49,195.46) and (363.36,202.54) .. (346.34,221.24) ;
					\draw [color={rgb, 255:red, 248; green, 72; blue, 61 }  ,draw opacity=1 ]   (323.92,207.55) .. controls (313.37,212.35) and (294.99,262.24) .. (340.97,259.42) ;
					\draw [color={rgb, 255:red, 248; green, 72; blue, 61 }  ,draw opacity=1 ]   (359.48,264.63) .. controls (397.1,282.37) and (397.99,245.58) .. (376.5,243.94) ;
					\draw [color={rgb, 255:red, 248; green, 72; blue, 61 }  ,draw opacity=1 ]   (346.34,221.24) .. controls (329.92,240.2) and (351.12,241.48) .. (376.5,243.94) ;
					\draw    (347.46,181.68) .. controls (350.3,182.07) and (358.25,179.88) .. (361.24,178.75) ;
					\draw  [fill={rgb, 255:red, 17; green, 16; blue, 16 }  ,fill opacity=1 ] (349.93,181.65) .. controls (349.91,180.28) and (348.8,179.19) .. (347.43,179.2) .. controls (346.07,179.22) and (344.97,180.34) .. (344.99,181.71) .. controls (345,183.08) and (346.12,184.18) .. (347.49,184.16) .. controls (348.85,184.14) and (349.94,183.02) .. (349.93,181.65) -- cycle ;
					\draw  [fill={rgb, 255:red, 17; green, 16; blue, 16 }  ,fill opacity=1 ] (333.35,195.8) .. controls (333.33,194.43) and (332.21,193.34) .. (330.85,193.35) .. controls (329.48,193.37) and (328.39,194.49) .. (328.4,195.86) .. controls (328.42,197.23) and (329.54,198.32) .. (330.9,198.31) .. controls (332.27,198.29) and (333.36,197.17) .. (333.35,195.8) -- cycle ;
					\draw   (324.35,173.15) -- (330.73,173.54) -- (325.75,177.54) ;
					\draw [color={rgb, 255:red, 248; green, 72; blue, 61 }  ,draw opacity=1 ]   (355,257.62) .. controls (387.24,257.11) and (404.84,241.71) .. (399.87,231.72) ;
					\draw  [color={rgb, 255:red, 248; green, 72; blue, 61 }  ,draw opacity=1 ] (318.41,224.75) .. controls (318.47,227.74) and (316.08,230.21) .. (313.09,230.26) .. controls (310.1,230.31) and (307.63,227.93) .. (307.58,224.94) .. controls (307.53,221.95) and (309.91,219.48) .. (312.9,219.43) .. controls (315.89,219.37) and (318.36,221.76) .. (318.41,224.75) -- cycle ;
					\draw   (373.59,215.82) .. controls (373.64,218.81) and (371.26,221.28) .. (368.27,221.33) .. controls (365.28,221.38) and (362.81,219) .. (362.76,216.01) .. controls (362.7,213.02) and (365.09,210.55) .. (368.08,210.5) .. controls (371.07,210.44) and (373.54,212.83) .. (373.59,215.82) -- cycle ;
					\draw  [color={rgb, 255:red, 248; green, 72; blue, 61 }  ,draw opacity=1 ] (393.34,250.39) .. controls (393.39,253.38) and (391.01,255.85) .. (388.02,255.9) .. controls (385.03,255.96) and (382.56,253.57) .. (382.51,250.58) .. controls (382.45,247.59) and (384.84,245.12) .. (387.83,245.07) .. controls (390.82,245.02) and (393.29,247.4) .. (393.34,250.39) -- cycle ;
					\draw   (358.7,212.78) .. controls (358.75,215.77) and (356.37,218.24) .. (353.38,218.29) .. controls (350.38,218.34) and (347.92,215.96) .. (347.86,212.97) .. controls (347.81,209.98) and (350.2,207.51) .. (353.19,207.46) .. controls (356.18,207.4) and (358.65,209.79) .. (358.7,212.78) -- cycle ;
					\draw    (494,257.12) .. controls (463.57,254.54) and (402.92,179.46) .. (484.48,170.42) ;
					\draw    (484.48,170.42) .. controls (513.11,166.52) and (522.8,187.93) .. (512.21,204.18) .. controls (501.63,220.44) and (505.2,229.46) .. (516.69,232.88) ;
					\draw    (469.88,195.33) .. controls (470.43,216.8) and (502.4,209.01) .. (528.93,222.37) ;
					\draw    (516.69,232.88) .. controls (544.38,241.3) and (557.17,160.64) .. (516.39,172.41) ;
					\draw [color={rgb, 255:red, 248; green, 72; blue, 61 }  ,draw opacity=1 ]   (479.37,199.06) .. controls (495.49,194.96) and (502.36,202.04) .. (485.34,220.74) ;
					\draw [color={rgb, 255:red, 248; green, 72; blue, 61 }  ,draw opacity=1 ]   (462.92,207.05) .. controls (452.37,211.85) and (433.99,261.74) .. (479.97,258.92) ;
					\draw [color={rgb, 255:red, 248; green, 72; blue, 61 }  ,draw opacity=1 ]   (498.48,264.13) .. controls (536.1,281.87) and (536.99,245.08) .. (515.5,243.44) ;
					\draw [color={rgb, 255:red, 248; green, 72; blue, 61 }  ,draw opacity=1 ]   (485.34,220.74) .. controls (468.92,239.7) and (490.12,240.98) .. (515.5,243.44) ;
					\draw    (486.46,181.18) .. controls (489.3,181.57) and (497.25,179.38) .. (500.24,178.25) ;
					\draw  [fill={rgb, 255:red, 17; green, 16; blue, 16 }  ,fill opacity=1 ] (488.93,181.15) .. controls (488.91,179.78) and (487.8,178.69) .. (486.43,178.7) .. controls (485.07,178.72) and (483.97,179.84) .. (483.99,181.21) .. controls (484,182.58) and (485.12,183.68) .. (486.49,183.66) .. controls (487.85,183.64) and (488.94,182.52) .. (488.93,181.15) -- cycle ;
					\draw  [fill={rgb, 255:red, 17; green, 16; blue, 16 }  ,fill opacity=1 ] (472.35,195.3) .. controls (472.33,193.93) and (471.21,192.84) .. (469.85,192.85) .. controls (468.48,192.87) and (467.39,193.99) .. (467.4,195.36) .. controls (467.42,196.73) and (468.54,197.82) .. (469.9,197.81) .. controls (471.27,197.79) and (472.36,196.67) .. (472.35,195.3) -- cycle ;
					\draw   (463.35,172.65) -- (469.73,173.04) -- (464.75,177.04) ;
					\draw    (494,257.12) .. controls (526.24,256.61) and (543.84,241.21) .. (538.87,231.22) ;
					\draw   (457.41,224.25) .. controls (457.47,227.24) and (455.08,229.71) .. (452.09,229.76) .. controls (449.1,229.81) and (446.63,227.43) .. (446.58,224.44) .. controls (446.53,221.45) and (448.91,218.98) .. (451.9,218.93) .. controls (454.89,218.87) and (457.36,221.26) .. (457.41,224.25) -- cycle ;
					\draw   (512.59,215.32) .. controls (512.64,218.31) and (510.26,220.78) .. (507.27,220.83) .. controls (504.28,220.88) and (501.81,218.5) .. (501.76,215.51) .. controls (501.7,212.52) and (504.09,210.05) .. (507.08,210) .. controls (510.07,209.94) and (512.54,212.33) .. (512.59,215.32) -- cycle ;
					\draw   (532.34,249.89) .. controls (532.39,252.88) and (530.01,255.35) .. (527.02,255.4) .. controls (524.03,255.46) and (521.56,253.07) .. (521.51,250.08) .. controls (521.45,247.09) and (523.84,244.62) .. (526.83,244.57) .. controls (529.82,244.52) and (532.29,246.9) .. (532.34,249.89) -- cycle ;
					\draw   (497.7,212.28) .. controls (497.75,215.27) and (495.37,217.74) .. (492.38,217.79) .. controls (489.38,217.84) and (486.92,215.46) .. (486.86,212.47) .. controls (486.81,209.48) and (489.2,207.01) .. (492.19,206.96) .. controls (495.18,206.9) and (497.65,209.29) .. (497.7,212.28) -- cycle ;
					\draw [color={rgb, 255:red, 248; green, 72; blue, 61 }  ,draw opacity=1 ]   (76.98,172.92) .. controls (95.32,168.31) and (100.99,178.65) .. (92.74,180.75) ;
					\draw   (107.3,184.6) -- (106.65,190.96) -- (102.85,185.82) ;
					\draw [color={rgb, 255:red, 248; green, 72; blue, 61 }  ,draw opacity=1 ]   (195.88,196.33) .. controls (196.87,199.68) and (202.65,200.79) .. (205.37,200.06) ;
					\draw   (203.88,205.32) -- (209.44,208.48) -- (203.2,209.88) ;
					\draw [color={rgb, 255:red, 248; green, 72; blue, 61 }  ,draw opacity=1 ]   (377.69,233.38) .. controls (385.06,235.58) and (396.91,225.42) .. (399.87,231.72) ;
					\draw   (401.9,207.47) -- (405.14,201.96) -- (406.45,208.22) ;
					\draw [color={rgb, 255:red, 248; green, 72; blue, 61 }  ,draw opacity=1 ]   (479.97,258.92) .. controls (484.9,258.6) and (497.9,263.94) .. (498.48,264.13) ;
					\draw   (471.9,261.18) -- (466.41,257.89) -- (472.68,256.64) ;
					\draw (75.47,156.44) node [anchor=north west][inner sep=0.75pt]  [font=\small]  {$D_{2_1}$};
					\draw (108.13,159.09) node [anchor=north west][inner sep=0.75pt]  [font=\small]  {$D_{2_2}$};
					\draw (129.22,222.83) node [anchor=north west][inner sep=0.75pt]  [font=\small]  {$+$};
					\draw (198.02,185.68) node [anchor=north west][inner sep=0.75pt]  [font=\small]  {$D_{2_1}$};
					\draw (190.47,209.56) node [anchor=north west][inner sep=0.75pt]  [font=\small]  {$D_{2_2}$};
					\draw (204,260.27) node [anchor=north west][inner sep=0.75pt]  [font=\small]  {$+$};
					\draw (262.89,220.56) node [anchor=north west][inner sep=0.75pt]  [font=\small]  {$-$};
					\draw (401.07,231.16) node [anchor=north west][inner sep=0.75pt]  [font=\small]  {$D_{2_1}$};
					\draw (404.2,206.89) node [anchor=north west][inner sep=0.75pt]  [font=\small]  {$D_{2_2}$};
					\draw (322.39,209.19) node [anchor=north west][inner sep=0.75pt]  [font=\small]  {$-$};
					\draw (363.89,162.19) node [anchor=north west][inner sep=0.75pt]  [font=\small]  {$-$};
					\draw (478.24,260) node [anchor=north west][inner sep=0.75pt]  [font=\small]  {$D_{2_1}$};
					\draw (480.87,241.89) node [anchor=north west][inner sep=0.75pt]  [font=\small]  {$D_{2_2}$};
					\draw (462.17,205.34) node [anchor=north west][inner sep=0.75pt]  [font=\small]  {$-$};
					\draw (68,275.88) node [anchor=north west][inner sep=0.75pt]   [align=left] {(a)};
					\draw (484,275.88) node [anchor=north west][inner sep=0.75pt]   [align=left] {(d)};
					\draw (341,275.88) node [anchor=north west][inner sep=0.75pt]   [align=left] {(c)};
					\draw (204,275.88) node [anchor=north west][inner sep=0.75pt]   [align=left] {(b)};
				\end{tikzpicture}
			\caption{(a)--(d): Diagrams obtained by orientation-preserving smoothing at classical crossings $c_1$, $c_2$, $c_3$ and $c_4$ in $D_2$.} \label{fig517}
		\end{figure}
		Combining the above results, we obtain that $\mathbf{P}_{K_1}(x,y)= \mathbf{P}_{K_2}(x,y)$. 
	}
\end{example}

\subsection{$\mathbf{P}_K(x,y)$ is a  Vassiliev invariant of order one} \label{subsection3.3}

We introduce singular virtual knotoids following the standard approach in \cite{Hen}. A \textit{singular crossing} is a transverse double point without over/under-crossing information, indicated by a solid dot, see Fig.~\ref{fig1}(d).
A \textit{singular virtual knotoid diagram} is a virtual knotoid diagram with finitely many singular crossings. \textit{Singular virtual knotoids} are defined as equivalence classes of singular virtual knotoid diagrams under the natural extension of the  Reidemeister moves.

We now use singular virtual knotoids to define Vassiliev invariants. Any virtual knotoid invariant $V$ with values in an abelian group $A$ can be extended to an invariant of singular virtual knotoids.  Let $D$ be a singular virtual knotoid diagram with $n \geq 1$ singular crossings. By resolving a chosen singular crossing of $D$ into a positive crossing, we obtain $D_+$, and by resolving it into a negative crossing, we obtain $D_-$, as shown in Fig.~\ref{fig4}. The diagrams $D_+$ and $D_-$ each have $n-1$ singular crossings.		
Define the \textit{$n$-th derivative} $V^{(n)}$, $n \geq 1$, of $V$ by the following recursive relation: 
\begin{equation}
	V^{(n)}(D)=
	\begin{cases}
		V^{(n-1)}(D_{+}) - V^{(n-1)}(D_{-}), & n \ge 1,\\
		V(D), & n = 0.
	\end{cases}
	\label{eq4.1}
\end{equation}

\begin{figure}[!ht]
	\begin{center}
		\tikzset{every picture/.style={line width=1.0pt}}  
		\begin{tikzpicture}[x=0.75pt,y=0.75pt,yscale=-1,xscale=1]	
			\draw    (490.26,442.05) -- (431.17,502.42) ;
			\draw   (431.41,447.9) -- (430.66,441.97) -- (436.43,443.17) ;
			\draw   (484.79,443.32) -- (490.51,441.91) -- (489.98,447.87) ;
			\draw    (465.09,476.88) -- (492.09,503.71) ;
			\draw    (189.81,441.61) -- (165.76,466.15) ;
			\draw    (130.6,441.67) -- (189.93,501.92) ;
			\draw   (131.25,447.46) -- (130.51,441.53) -- (136.27,442.73) ;
			\draw   (184.34,442.88) -- (190.06,441.47) -- (189.53,447.43) ;
			\draw    (155.09,477.15) -- (130.72,501.98) ;
			\draw    (430.17,441.66) -- (455.76,467.54) ;
			\draw    (342.82,441.65) -- (283.74,502.03) ;
			\draw    (283.62,441.72) -- (342.94,501.96) ;
			\draw   (284.27,447.51) -- (283.52,441.58) -- (289.29,442.77) ;
			\draw   (337.35,442.93) -- (343.07,441.52) -- (342.54,447.47) ;
			\draw  [fill={rgb, 255:red, 14; green, 13; blue, 13 }  ,fill opacity=1 ] (307.86,471.84) .. controls (307.86,468.85) and (310.29,466.42) .. (313.28,466.42) .. controls (316.27,466.42) and (318.7,468.85) .. (318.7,471.84) .. controls (318.7,474.83) and (316.27,477.26) .. (313.28,477.26) .. controls (310.29,477.26) and (307.86,474.83) .. (307.86,471.84) -- cycle ;
			\draw    (371.11,469.66) -- (403.03,469.17) ;
			\draw   (400.74,466.29) -- (403.76,469.11) -- (400.81,472) ;
			\draw    (252.76,470.66) -- (220.84,470.17) ;
			\draw   (223.14,467.29) -- (220.11,470.11) -- (223.06,473) ;
			\draw (147,504.66) node [anchor=north west][inner sep=0.75pt]  [font=\normalsize]   {$D_{-}$};
			\draw (447,505.66) node [anchor=north west][inner sep=0.75pt]  [font=\normalsize]   {$D_{+}$};
			\draw (304,507.62) node [anchor=north west][inner sep=0.75pt]  [font=\normalsize]   {$D$};
		\end{tikzpicture}
		\caption{Two resolutions of a singular crossing in $D$.} \label{fig4}
	\end{center}
\end{figure}
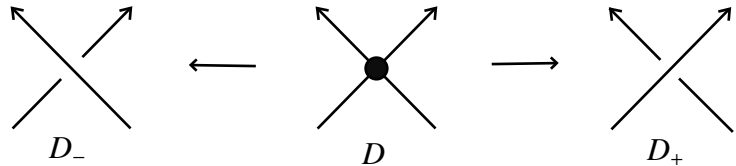

An invariant $V$ is called a \textit{Vassiliev invariant of order $n$} if it vanishes on all singular virtual knotoids with more than $n$ singular crossings, and does not vanish on at least one singular virtual knotoid with $n$ singular crossings.

\begin{theorem} \label{thm3}
	The polynomial $\mathbf P_K(x,y)$ is an order one Vassiliev invariant of virtual knotoids.
\end{theorem}

\begin{proof}
	We first prove that $\mathbf P_K(x,y)$ is a Vassiliev invariant of order at most one.
	
	Let $D_{pq}$ be a singular virtual knotoid diagram with two singular crossings $p$ and $q$.
	By replacing $p$ and $q$ with positive or negative classical crossings, denoted as $p_+$, $p_-$, $q_+$, $q_-$, we obtain four diagrams, denoted by $D_{p+q+}$, $D_{p+q-}$, $D_{p-q+}$, and $D_{p-q-}$. Then denoting by $\mathbf{P}^{(i)}$ the $i$-th derivative of $\mathbf{P}$.
	By Equation~\eqref{eq4.1}, 
	$$
	\begin{aligned}
		\mathbf P^{(2)}_{D_{pq}}(x,y)
		&=
		\mathbf P^{(1)}_{D_{p+}}(x,y)
		-
		\mathbf P^{(1)}_{D_{p-}}(x,y) \\
		&=
		\bigl(
		\mathbf P_{D_{p+q+}}(x,y)
		-
		\mathbf P_{D_{p+q-}}(x,y)
		\bigr)
		-
		\bigl(
		\mathbf P_{D_{p-q+}}(x,y)
		-
		\mathbf P_{D_{p-q-}}(x,y)
		\bigr).
	\end{aligned}
	$$		
	By Equation~\eqref{equ6} in the proof of Theorem~\ref{thm2}, we have
	$
	\mathbf P_D(x,y)=\Phi(AI_D(t))
	$. Therefore,
	$$
	\mathbf P^{(2)}_{D_{pq}}(x,y)
	=
	\Phi\left(AI_{D_{p+q+}}(t)\right)
	-
	\Phi\left(AI_{D_{p+q-}}(t)\right) 
	-
	\Phi\left(AI_{D_{p-q+}}(t)\right)
	+
	\Phi\left(AI_{D_{p-q-}}(t)\right).
	$$	
	Since $\Phi$ is $\mathbb Z$-linear,
	\[
	\begin{aligned}
		\mathbf P^{(2)}_{D_{pq}}(x,y)
		&=
		\Phi\left(
		AI_{D_{p+q+}}(t)
		-
		AI_{D_{p+q-}}(t)  
		-
		AI_{D_{p-q+}}(t)
		+
		AI_{D_{p-q-}}(t)
		\right) 
		=
		\Phi\left(
		AI^{(2)}_{D_{pq}}(t)
		\right).
	\end{aligned}
	\]		
	
	By \cite[Proposition~7.3]{MLK}, the affine index polynomial $AI_K(t)$ is an order one Vassiliev invariant, we have $AI^{(2)}_{D_{pq}}(t)=0$. Hence,	
	$\mathbf P^{(2)}_{D_{pq}}(x,y)=\Phi(0)=0.$
	
	Therefore, $\mathbf P_K(x,y)$ vanishes on every singular virtual knotoid with two singular crossings.
	
	It remains to show that there is a singular virtual knotoid with one singular crossing which has nontrivial $\mathbf P^{(1)}$. Suppose, to the contrary, that $\mathbf P^{(1)}_D(x,y)=0$ for every singular virtual knotoid diagram $D$ with one singular crossing. Then
	$\mathbf P_D^{(1)}(x,y)=\mathbf P_{D_+}(x,y)-\mathbf P_{D_-}(x,y) =0,$	
	and hence $\mathbf P_{D_+}(x,y)=\mathbf P_{D_-}(x,y)$. By Theorem~\ref{thm2}, this implies that $AI_{D_+}(t)=AI_{D_-}(t)$. Therefore, $AI_D^{(1)}(t)=AI_{D_+}(t)-AI_{D_-}(t) =0$, contradicting the fact that the affine index polynomial is a Vassiliev invariant of order one.
	
	Hence, there exists a singular virtual knotoid diagram $D$ with one singular crossing such that $ \mathbf P^{(1)}_D(x,y)\neq0$. 
	
	Consequently, $\mathbf P_K(x,y)$ is an order one Vassiliev invariant.
\end{proof}

\subsection{Reversing the orientation and taking the mirror image}
In this subsection we discuss the behavior of $\mathbf{P}_K(x,y)$ under reversing the orientation and taking the mirror image.

Let $D$ be a virtual knotoid diagram. Denote the orientation-reversed virtual knotoid diagram of $D$ by $r(D)$, and call $r(D)$ \textit{the inverted virtual knotoid diagram} of $D$. A virtual knotoid $K$ is \textit{invertible} if $K$ is equivalent to $r(K)$. Denote the diagram obtained by switching the under- and over-strands of every classical crossing in $D$ by $m(D)$, and call $m(D)$ the \textit{mirror image} of $D$. A virtual knotoid $K$ is \textit{amphichiral} if $K$ is equivalent to $m(K)$.

\begin{proposition} \label{prop3.1}
	Let $K$ be an oriented virtual knotoid, and $D$ a diagram of $K$. Then the following properties hold:	
	\begin{itemize}
		\item[{(1)}]
		$\mathbf{P}_{r(D)}(x,y)= -\mathbf{P}_D(x^{-1},y^{-1})$. If $D$ is invertible, then $\mathbf{P}_{D}(x,y)=-\mathbf{P}_{D}(x^{-1},y^{-1})$.
		\item[{(2)}]
		$\mathbf{P}_{m(D)}(x,y)= \mathbf{P}_{D}(x^{-1},y^{-1})$. If $D$ is amphichiral, then $\mathbf{P}_D(x,y)= \mathbf{P}_D(x^{-1},y^{-1})$.
		\item[{(3)}]
		If $D$ is invertible and amphichiral, then $\mathbf{P}_D (x,y)=0$.
	\end{itemize}
\end{proposition}

\begin{proof}
	Let $\mathcal{D}$ be the set of all classical crossings in $D$. Let $r(\mathcal{D})$ (resp. $m(\mathcal{D})$) be the set of all classical crossings in $r(D)$ (resp. $m(D)$). Let $r(\mathcal{A})$ (resp. $m(\mathcal{A})$) be the set of arcs of $r(D)$ (resp. $m(D)$) segmented by the classical crossings. For any $c\in \mathcal{D}$, we denote the corresponding elements in $r(\mathcal{D})$ and $m(\mathcal{D})$ by $r(c)$ and $m(c)$. By the definition of parity, the parity of the corresponding classical crossings in $D$ and $r(D)$ ($m(D)$) remains unchanged. For $\alpha \in \mathcal{A}$, we denote the corresponding elements in $r(\mathcal{A})$ and $m(\mathcal{A})$ by $r(\alpha)$, $m(\alpha)$. We assign integer labelings to the arcs as shown in Fig.~\ref{fig413}. 
	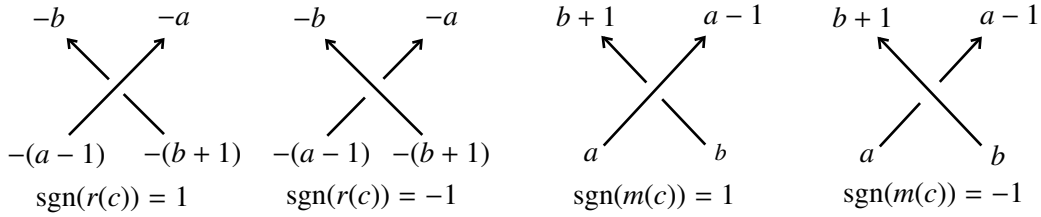
\begin{figure}[htbp]
		\begin{center}
			\tikzset{every picture/.style={line width=1.0pt}}  
			\begin{tikzpicture}[x=0.75pt,y=0.75pt,yscale=-1,xscale=1]
				\draw    (225.58,59.05) -- (273.34,106.5) ;
				\draw    (143.09,58.74) -- (96.08,106.53) ;
				\draw    (95.99,58.64) -- (116.1,78.61) ;
				\draw   (96.27,63.37) -- (95.68,58.68) -- (100.27,59.63) ;
				\draw   (138.74,59.75) -- (143.29,58.63) -- (142.87,63.34) ;
				\draw    (123.51,86) -- (143.75,106.1) ;
				\draw   (225.86,63.78) -- (225.26,59.08) -- (229.85,60.03) ;
				\draw   (268.33,60.15) -- (272.88,59.03) -- (272.46,63.75) ;
				\draw    (272.68,59.14) -- (253.81,78.35) ;
				\draw    (244.54,87.73) -- (225.67,106.93) ;
				\draw    (502.78,58.01) -- (553.83,111.81) ;
				\draw    (414.61,57.66) -- (364.36,111.85) ;
				\draw    (364.26,57.55) -- (385.75,80.19) ;
				\draw   (364.56,62.91) -- (363.92,57.59) -- (368.83,58.66) ;
				\draw   (409.95,58.8) -- (414.82,57.53) -- (414.37,62.88) ;
				\draw    (393.68,88.57) -- (415.31,111.35) ;
				\draw   (503.08,63.37) -- (502.44,58.04) -- (507.35,59.12) ;
				\draw   (548.47,59.26) -- (553.34,57.99) -- (552.89,63.34) ;
				\draw    (553.13,58.11) -- (532.96,79.89) ;
				\draw    (523.05,90.52) -- (502.88,112.3) ;
				\draw (63.36,108.95) node [anchor=north west][inner sep=0.75pt]  [font=\small]  {$-( a-1)$};
				\draw (131.61,108.5) node [anchor=north west][inner sep=0.75pt]  [font=\small]  {$-( b+1)$};
				\draw (75.91,40.67) node [anchor=north west][inner sep=0.75pt]  [font=\small]  {$-b$};
				\draw (138.62,42.06) node [anchor=north west][inner sep=0.75pt]  [font=\small]  {$-a$};
				\draw (206.69,40.58) node [anchor=north west][inner sep=0.75pt]  [font=\small]  {$-b$};
				\draw (273.01,42.06) node [anchor=north west][inner sep=0.75pt]  [font=\small]  {$-a$};
				\draw (352.74,111.32) node [anchor=north west][inner sep=0.75pt]  [font=\small]  {$a$};
				\draw (418.42,109.45) node [anchor=north west][inner sep=0.75pt]  [font=\scriptsize]  {$b$};
				\draw (338.89,40.12) node [anchor=north west][inner sep=0.75pt]  [font=\small]  {$b+1$};
				\draw (412.95,39.50) node [anchor=north west][inner sep=0.75pt]  [font=\small]  {$a-1$};
				\draw (491.26,111.77) node [anchor=north west][inner sep=0.75pt]  [font=\small]  {$a$};
				\draw (556.94,109.91) node [anchor=north west][inner sep=0.75pt]  [font=\small]  {$b$};
				\draw (477.41,40.58) node [anchor=north west][inner sep=0.75pt]  [font=\small]  {$b+1$};
				\draw (551.47,39.50) node [anchor=north west][inner sep=0.75pt]  [font=\small]  {$a-1$};
				\draw (196.32,108.47) node [anchor=north west][inner sep=0.75pt]  [font=\small]  {$-( a-1)$};
				\draw (256.3,108.5) node [anchor=north west][inner sep=0.75pt]  [font=\small]  {$-( b+1)$};
				\draw (77.8,130.9) node [anchor=north west][inner sep=0.75pt]    [font=\small]  {$\operatorname{sgn}( r( c)) =1$};
				\draw (203.34,129.8) node [anchor=north west][inner sep=0.75pt]    [font=\small]  {$\operatorname{sgn}( r( c)) =-1$};
				\draw (347.34,129.8) node [anchor=north west][inner sep=0.75pt]    [font=\small]  {$\operatorname{sgn}( m( c)) =1$};
				\draw (482.88,128.7) node [anchor=north west][inner sep=0.75pt]    [font=\small]  {$\operatorname{sgn}( m( c)) =-1$};					
			\end{tikzpicture}
			\caption{Assigning labels to the arcs of $r(D)$ and $m(D)$.\label{fig413}}	
		\end{center}
	\end{figure}
	
	\begin{itemize} 
		\item[(1)] It is clear that $\operatorname{sgn}(r(c))=\operatorname{sgn}(c)$. According to the definition of $W_D(c)$, we get:
		\begin{eqnarray*}
			W_{r(D)}\big(r(c)\big) & = &  \operatorname{sgn} \big(r(c) \big)\left[-(a-1)-(-b)\right]  
			=  \operatorname{sgn} (c)\left[b-a+1\right]  \cr 
			& = & -\left( \operatorname{sgn}(c) [a-(b+1)] \right)  
			=  - W_{D} (c).   
		\end{eqnarray*}
		Therefore
		\begin{eqnarray*} 
			%\hspace{3em}
			\mathbf{P}_{r(D)}(x,y) & = & \sum_{r(c) \in E\big(r(D)\big)} \operatorname{sgn} \big(r(c)\big) i \big(r(c) \big) x^{W_{r(D)}\big(r(c)\big)} \\ &&+ \sum_{r(c) \in O\big(r(D)\big)} \operatorname{sgn}\big(r(c)\big) i(r(c)) y^{W_{r(D)} (r(c))}   \cr 
			& = & \sum_{c \in E(D)} \operatorname{sgn}(c) \big(-i(c)\big) x^{-W_D(c)} + \sum_{c \in O(D)} \operatorname{sgn}(c) \big(-i(c)\big) y^{-W_D(c)} \\ &=& -\mathbf{P}_D(x^{-1},y^{-1}). 
		\end{eqnarray*}
		If $D$ is invertible, then $\mathbf{P}_D(x,y)= -\mathbf{P}_D(x^{-1},y^{-1})$.
		
		\item[(2)] It is obvious that $\operatorname{sgn}(m(c))=-\operatorname{sgn}(c)$ and
		$$ 
		W_{m(D)}\big(m(c)\big) =  \operatorname{sgn}\big(m(c)\big)[a-(b+1)]  
		=  -\operatorname{sgn} ( c )[a-(b+1)]  
		=   -W_{D} (c).
		$$
		Therefore 
		\begin{eqnarray*}
			\mathbf{P}_{m(D)}(x,y) & = &  \sum_{m(c) \in E\big(m(D)\big)} \operatorname{sgn}\big(m(c)\big) i\big(m(c)\big) x^{W_{m(D)}\big(m(c)\big)}\\ &&+ \sum_{m(c) \in O\big(m(D)\big)} \operatorname{sgn}\big(m(c)\big) i\big(m(c)\big) y^{W_{m(D)}\big(m(c)\big)} \cr 
			& = & - \sum_{c \in E(D)} \operatorname{sgn}(c) \big(-i(c)\big) x^{-W_D(c)} - \sum_{c \in O(D)} \operatorname{sgn}(c) \big(-i(c)\big) y^{-W_D(c)} \\ &=& \mathbf{P}_{D}(x^{-1},y^{-1}).
		\end{eqnarray*}    
		If $D$ is amphichiral, then $\mathbf{P}_D (x,y)=\mathbf{P}_D (x^{-1},y^{-1})$.
		\item[(3)] The proof follows from items (1) and (2). 
	\end{itemize}	
	This completes the proof of the proposition.
\end{proof} 

%%%%%
\section{$\mathbf{P}_K(x,y)$ and the universal Vassiliev invariant $\mathcal G(K)$ } \label{section4}

%%%%%
\subsection{Singular based matrices for singular virtual open strings} \label{subsection4.1}

The notion of a based matrix for virtual strings, which correspond to flat virtual knots, was originally introduced by Turaev in~\cite{Tur2}.  Based matrices were subsequently generalized by Henrich~\cite{Hen} to the setting of flat singular virtual knots. It was further extended by Petit~\cite{Pet} to the categories of framed virtual knots and of long virtual knots, both framed and unframed. In the present work, we extend this construction to the setting of flat virtual knotoids. Analogously, we introduce virtual open strings corresponding to flat virtual knotoids. We now give the precise definitions.

For an integer $m \geq 0$, a \textit{virtual open string} $\alpha$ of rank $m$ consists of an oriented arc $S$, called the \textit{core arc} of $\alpha$, together with a distinguished set of $2m$ distinct points on $S$ partitioned into $m$ ordered pairs. The starting point of $S$ is called the \textit{tail} of $\alpha$, and the end point is called the \textit{head}. The $m$ ordered pairs of points are called \textit{arrows} of $\alpha$, and their collection is denoted by $\operatorname{arr}(\alpha)$. For an arrow $(a,b) \in \operatorname{arr}(\alpha)$, the points $a,b \in S$ are referred to as its \textit{tail} and \textit{head}, respectively.

Two virtual open strings are said to be \textit{homeomorphic} if there exists an orientation-preserving homeomorphism of their core arcs that maps the set of arrows of one to that of the other. The corresponding homeomorphism classes are also referred to as virtual open strings.

Given a virtual open string $\alpha$, the associated flat virtual knotoid diagram can be constructed as follows. Each arrow of $\alpha$ corresponds to a crossing in the diagram. Traversing the oriented core arc $S$ through the head of an arrow corresponds to passing through the strand labeled $b$ at the crossing, while passing through the tail of an arrow corresponds to traversing the strand labeled $a$, see Fig.~\ref{fig127}.
\begin{figure}[htbp]
	\begin{center}
		\tikzset{every picture/.style={line width=1.0pt}}  
		\begin{tikzpicture}[x=0.75pt,y=0.75pt,yscale=-0.8, xscale=0.8] 
			\draw    (147.86,21.77) -- (79.76,92.61) ;
			\draw    (79.62,21.85) -- (148,92.53) ;
			\draw   (80.37,28.64) -- (79.51,21.68) -- (86.16,23.08) ;
			\draw    (314.66,56.78) -- (234.76,56.32) ;
			\draw   (309.96,53.6) -- (315.34,56.76) -- (309.72,59.44) ;
			\draw    (245.86,73.74) .. controls (231.21,65.76) and (230.64,46.99) .. (245.67,39.12) ;
			\draw   (141.56,23.27) -- (148.15,21.61) -- (147.54,28.6) ;
			\draw    (304.95,39.66) .. controls (319.51,47.83) and (319.86,66.59) .. (304.74,74.29) ;
			\draw (147.93,8.43) node [anchor=north west][inner sep=0.75pt]    {$a$};
			\draw (69.84,7) node [anchor=north west][inner sep=0.75pt]    {$b$};
			\draw (218.27,48.5) node [anchor=north west][inner sep=0.75pt]    {$a$};
			\draw (321.49,46.97) node [anchor=north west][inner sep=0.75pt]    {$b$};
		\end{tikzpicture}
		\caption{A flat crossing and its corresponding arrow $(a,b)$.\label{fig127}}
	\end{center}
\end{figure}
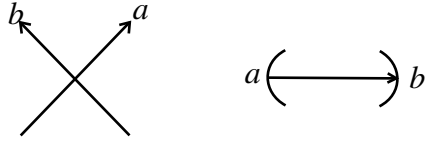
Flat singular virtual knotoids with one singular crossing can be regarded as virtual open strings in which one arrow is designated as the preferred arrow, which can be pictured by a thickened arrow. Such modified virtual open strings will be referred to as \textit{singular virtual open strings}. The preferred arrow in a singular virtual open string naturally corresponds to the singular crossing in the associated flat singular virtual knotoid diagram.

We next introduce the moves defined on singular virtual open strings.
Recall that there are three moves (i)--(iii) for virtual open strings:
\begin{itemize}
	\item[(i)] Adding an arrow $(a, b)$ to $\operatorname{arr}(\alpha)$, where $a, b \in S$ are points for which the arc $ab$ is disjoint from $\operatorname{arr}(\alpha)$.
	\item[(ii)] Adding a pair of new arrows $(a, b)$ and $(b',a')$, where $a, b, a', b' \in S$ are four distinct points such that the two corresponding arcs on $S$ with endpoints $a,a'$ and $b,b'$ are both disjoint from $\operatorname{arr}(\alpha)$.
	\item[(iii)] Replacing the arrows $(a',b)$, $(b',c)$, $(c',a)$ with $(a, b')$, $(b, c')$, $(c, a')$, where $a, b, c, a', b',$ $c' \in S$ satisfy that $(a',b)$, $(b',c)$, $(c',a)$ are arrows of $\alpha$ and the arcs $aa'$, $bb'$, $cc'$ are all disjoint from the other arrows, see Fig.~\ref{fig530}. 
\end{itemize}

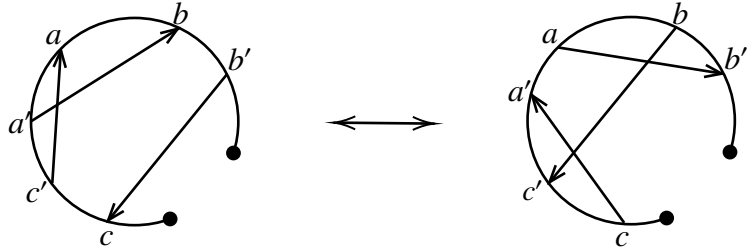
\begin{figure}[htbp]
	\begin{center}
		\tikzset{every picture/.style={line width=1.0pt}}  
		\begin{tikzpicture}[x=0.75pt,y=0.75pt,yscale=-1.0, xscale=1.0] 
			\draw  [draw opacity=0] (152.73,262.89) .. controls (147.18,264.91) and (141.19,266.01) .. (134.94,266.01) .. controls (106.22,266.01) and (82.94,242.73) .. (82.94,214.01) .. controls (82.94,185.29) and (106.22,162.01) .. (134.94,162.01) .. controls (163.66,162.01) and (186.94,185.29) .. (186.94,214.01) .. controls (186.94,219.51) and (186.09,224.82) .. (184.5,229.79) -- (134.94,214.01) -- cycle ; \draw   (152.73,262.89) .. controls (147.18,264.91) and (141.19,266.01) .. (134.94,266.01) .. controls (106.22,266.01) and (82.94,242.73) .. (82.94,214.01) .. controls (82.94,185.29) and (106.22,162.01) .. (134.94,162.01) .. controls (163.66,162.01) and (186.94,185.29) .. (186.94,214.01) .. controls (186.94,219.51) and (186.09,224.82) .. (184.5,229.79) ;  
			\draw  [fill={rgb, 255:red, 0; green, 0; blue, 0 }  ,fill opacity=1 ] (181.47,229.79) .. controls (181.47,228.11) and (182.83,226.76) .. (184.5,226.76) .. controls (186.18,226.76) and (187.54,228.11) .. (187.54,229.79) .. controls (187.54,231.47) and (186.18,232.83) .. (184.5,232.83) .. controls (182.83,232.83) and (181.47,231.47) .. (181.47,229.79) -- cycle ;
			\draw  [fill={rgb, 255:red, 0; green, 0; blue, 0 }  ,fill opacity=1 ] (149.69,262.89) .. controls (149.69,261.21) and (151.05,259.85) .. (152.73,259.85) .. controls (154.41,259.85) and (155.77,261.21) .. (155.77,262.89) .. controls (155.77,264.57) and (154.41,265.93) .. (152.73,265.93) .. controls (151.05,265.93) and (149.69,264.57) .. (149.69,262.89) -- cycle ;
			\draw    (121.44,263.94) -- (181.08,190.48) ;
			\draw    (97.9,178.05) -- (93.62,244.88) ;
			\draw    (83,213.86) -- (157.44,167.03) ;
			\draw  [draw opacity=0] (401.84,262.43) .. controls (396.29,264.45) and (390.3,265.55) .. (384.05,265.55) .. controls (355.33,265.55) and (332.05,242.27) .. (332.05,213.55) .. controls (332.05,184.83) and (355.33,161.55) .. (384.05,161.55) .. controls (412.77,161.55) and (436.05,184.83) .. (436.05,213.55) .. controls (436.05,219.05) and (435.2,224.35) .. (433.61,229.33) -- (384.05,213.55) -- cycle ; \draw   (401.84,262.43) .. controls (396.29,264.45) and (390.3,265.55) .. (384.05,265.55) .. controls (355.33,265.55) and (332.05,242.27) .. (332.05,213.55) .. controls (332.05,184.83) and (355.33,161.55) .. (384.05,161.55) .. controls (412.77,161.55) and (436.05,184.83) .. (436.05,213.55) .. controls (436.05,219.05) and (435.2,224.35) .. (433.61,229.33) ;  
			\draw  [fill={rgb, 255:red, 0; green, 0; blue, 0 }  ,fill opacity=1 ] (430.58,229.33) .. controls (430.58,227.65) and (431.94,226.29) .. (433.61,226.29) .. controls (435.29,226.29) and (436.65,227.65) .. (436.65,229.33) .. controls (436.65,231.01) and (435.29,232.37) .. (433.61,232.37) .. controls (431.94,232.37) and (430.58,231.01) .. (430.58,229.33) -- cycle ;
			\draw  [fill={rgb, 255:red, 0; green, 0; blue, 0 }  ,fill opacity=1 ] (398.8,262.43) .. controls (398.8,260.75) and (400.16,259.39) .. (401.84,259.39) .. controls (403.52,259.39) and (404.88,260.75) .. (404.88,262.43) .. controls (404.88,264.1) and (403.52,265.46) .. (401.84,265.46) .. controls (400.16,265.46) and (398.8,264.1) .. (398.8,262.43) -- cycle ;
			\draw    (347.2,176.79) -- (430.19,190.02) ;
			\draw    (406.55,166.79) -- (342.73,244.97) ;
			\draw    (333.81,200.09) -- (380.53,264.8) ;
			\draw   (350.88,239.83) -- (343.11,244.39) -- (346.19,235.92) ;
			\draw   (421.54,185.62) -- (429.43,189.97) -- (420.57,191.65) ;
			\draw   (336.88,209.57) -- (334.36,200.91) -- (341.82,205.98) ;
			\draw   (147.75,169.67) -- (156.49,167.48) -- (151.15,174.74) ;
			\draw   (94.29,186.6) -- (97.94,178.36) -- (100.38,187.04) ;
			\draw   (129.47,258.93) -- (121.7,263.49) -- (124.78,255.02) ;
			\draw    (236.67,214.56) -- (283.42,215.02) ;
			\draw [shift={(285.42,215.04)}, rotate = 180.56] [color={rgb, 255:red, 0; green, 0; blue, 0 }  ][line width=0.75]    (10.93,-3.29) .. controls (6.95,-1.4) and (3.31,-0.3) .. (0,0) .. controls (3.31,0.3) and (6.95,1.4) .. (10.93,3.29)   ;
			\draw [shift={(234.67,214.54)}, rotate = 0.56] [color={rgb, 255:red, 0; green, 0; blue, 0 }  ][line width=0.75]    (10.93,-3.29) .. controls (6.95,-1.4) and (3.31,-0.3) .. (0,0) .. controls (3.31,0.3) and (6.95,1.4) .. (10.93,3.29)   ;
			\draw (180.14,175.58) node [anchor=north west][inner sep=0.75pt]    {$b'$};
			\draw (115.68,267.19) node [anchor=north west][inner sep=0.75pt]    {$c$};
			\draw (88.67,166.49) node [anchor=north west][inner sep=0.75pt]    {$a$};
			\draw (152.35,152.69) node [anchor=north west][inner sep=0.75pt]    {$b$};
			\draw (69.93,207.25) node [anchor=north west][inner sep=0.75pt]    {$a'$};
			\draw (78.69,242) node [anchor=north west][inner sep=0.75pt]    {$c'$};
			\draw (429.25,175.11) node [anchor=north west][inner sep=0.75pt]    {$b'$};
			\draw (374.39,267.53) node [anchor=north west][inner sep=0.75pt]    {$c$};
			\draw (337.78,166.02) node [anchor=north west][inner sep=0.75pt]    {$a$};
			\draw (403.05,152.62) node [anchor=north west][inner sep=0.75pt]    {$b$};
			\draw (320.64,191.38) node [anchor=north west][inner sep=0.75pt]    {$a'$};
			\draw (327.8,241.53) node [anchor=north west][inner sep=0.75pt]  [font=\small]  {$c'$};
		\end{tikzpicture}
		\caption{Move (iii).} \label{fig530}
	\end{center}
\end{figure}

In the singular case, we further introduce an additional move that allows us to change the preferred arrow:
\begin{itemize} 
	\item[(s-ii)] Assume that $(a,b)$ is the preferred arrow and $(a',b')$ is another arrow in the diagram such that the interior of the arcs on the core circle $S$ with endpoints $a$ and $b'$ is disjoint from $\operatorname{arr}(\alpha)$, and similarly the interior of the arcs on $S$ with endpoints $a'$ and $b$ is disjoint from $\operatorname{arr}(\alpha)$, see Fig.~\ref{fig531}.   Then the role of the preferred arrow can be switched from $(a,b)$ to $(a',b')$.
\end{itemize}

\begin{figure}[htbp]
	\begin{center}
		\tikzset{every picture/.style={line width=1.0pt}}  
		\begin{tikzpicture}[x=0.75pt,y=0.75pt,yscale=-1.0, xscale=1.0] 
			\draw    (90.27,374.2) -- (161,329.46) ;
			\draw  [draw opacity=0] (405.32,420.86) .. controls (399.77,422.88) and (393.78,423.98) .. (387.53,423.98) .. controls (358.81,423.98) and (335.53,400.7) .. (335.53,371.98) .. controls (335.53,343.26) and (358.81,319.98) .. (387.53,319.98) .. controls (416.25,319.98) and (439.53,343.26) .. (439.53,371.98) .. controls (439.53,377.48) and (438.67,382.78) .. (437.09,387.76) -- (387.53,371.98) -- cycle ; \draw   (405.32,420.86) .. controls (399.77,422.88) and (393.78,423.98) .. (387.53,423.98) .. controls (358.81,423.98) and (335.53,400.7) .. (335.53,371.98) .. controls (335.53,343.26) and (358.81,319.98) .. (387.53,319.98) .. controls (416.25,319.98) and (439.53,343.26) .. (439.53,371.98) .. controls (439.53,377.48) and (438.67,382.78) .. (437.09,387.76) ;  
			\draw  [draw opacity=0] (159.73,421.89) .. controls (154.18,423.91) and (148.19,425.01) .. (141.94,425.01) .. controls (113.22,425.01) and (89.94,401.73) .. (89.94,373.01) .. controls (89.94,344.29) and (113.22,321.01) .. (141.94,321.01) .. controls (170.66,321.01) and (193.94,344.29) .. (193.94,373.01) .. controls (193.94,378.51) and (193.09,383.82) .. (191.5,388.79) -- (141.94,373.01) -- cycle ; \draw   (159.73,421.89) .. controls (154.18,423.91) and (148.19,425.01) .. (141.94,425.01) .. controls (113.22,425.01) and (89.94,401.73) .. (89.94,373.01) .. controls (89.94,344.29) and (113.22,321.01) .. (141.94,321.01) .. controls (170.66,321.01) and (193.94,344.29) .. (193.94,373.01) .. controls (193.94,378.51) and (193.09,383.82) .. (191.5,388.79) ;  
			\draw  [fill={rgb, 255:red, 0; green, 0; blue, 0 }  ,fill opacity=1 ] (188.47,388.79) .. controls (188.47,387.11) and (189.83,385.76) .. (191.5,385.76) .. controls (193.18,385.76) and (194.54,387.11) .. (194.54,388.79) .. controls (194.54,390.47) and (193.18,391.83) .. (191.5,391.83) .. controls (189.83,391.83) and (188.47,390.47) .. (188.47,388.79) -- cycle ;
			\draw  [fill={rgb, 255:red, 0; green, 0; blue, 0 }  ,fill opacity=1 ] (156.69,421.89) .. controls (156.69,420.21) and (158.05,418.85) .. (159.73,418.85) .. controls (161.41,418.85) and (162.77,420.21) .. (162.77,421.89) .. controls (162.77,423.57) and (161.41,424.93) .. (159.73,424.93) .. controls (158.05,424.93) and (156.69,423.57) .. (156.69,421.89) -- cycle ;
			\draw    (104.14,408.45) -- (187.64,348.99) ;
			\draw   (113.49,405.7) -- (104.81,408.11) -- (109.96,400.72) ;
			\draw    (90.27,371.46) -- (161.7,326.89) ;
			\draw  [fill={rgb, 255:red, 0; green, 0; blue, 0 }  ,fill opacity=1 ] (434.05,387.76) .. controls (434.05,386.08) and (435.41,384.72) .. (437.09,384.72) .. controls (438.77,384.72) and (440.13,386.08) .. (440.13,387.76) .. controls (440.13,389.43) and (438.77,390.79) .. (437.09,390.79) .. controls (435.41,390.79) and (434.05,389.43) .. (434.05,387.76) -- cycle ; 
			\draw  [fill={rgb, 255:red, 0; green, 0; blue, 0 }  ,fill opacity=1 ] (402.28,420.86) .. controls (402.28,419.18) and (403.64,417.82) .. (405.32,417.82) .. controls (406.99,417.82) and (408.35,419.18) .. (408.35,420.86) .. controls (408.35,422.53) and (406.99,423.89) .. (405.32,423.89) .. controls (403.64,423.89) and (402.28,422.53) .. (402.28,420.86) -- cycle ;
			\draw    (351.73,405.27) -- (433.28,347.26) ;
			\draw   (400.33,327.63) -- (409.07,325.45) -- (403.73,332.7) ;
			\draw   (359.29,404.96) -- (350.5,407.22) -- (355.55,399.68) ;
			\draw    (335.72,371.63) -- (408,326.2) ;
			\draw    (354,406.21) -- (434.23,349.39) ;
			\draw   (154.75,328.67) -- (163.5,326.5) -- (158.17,333.77) ;
			\draw    (242.67,371.89) -- (271.11,372.17) -- (289.42,372.36) ;
			\draw [shift={(291.42,372.38)}, rotate = 180.56] [color={rgb, 255:red, 0; green, 0; blue, 0 }  ][line width=0.75]    (10.93,-3.29) .. controls (6.95,-1.4) and (3.31,-0.3) .. (0,0) .. controls (3.31,0.3) and (6.95,1.4) .. (10.93,3.29)   ;
			\draw [shift={(240.67,371.88)}, rotate = 0.56] [color={rgb, 255:red, 0; green, 0; blue, 0 }  ][line width=0.75]    (10.93,-3.29) .. controls (6.95,-1.4) and (3.31,-0.3) .. (0,0) .. controls (3.31,0.3) and (6.95,1.4) .. (10.93,3.29)   ;
			\draw (188.14,335.86) node [anchor=north west][inner sep=0.75pt]    {$b'$};
			\draw (159.35,311.69) node [anchor=north west][inner sep=0.75pt]    {$b$};
			\draw (79.5,373.25) node [anchor=north west][inner sep=0.75pt]    {$a$};
			\draw (90.44,410.58) node [anchor=north west][inner sep=0.75pt]    {$a'$};
			\draw (433.73,335.83) node [anchor=north west][inner sep=0.75pt]    {$b'$};
			\draw (404.93,310.65) node [anchor=north west][inner sep=0.75pt]    {$b$};
			\draw (325.08,372.21) node [anchor=north west][inner sep=0.75pt]    {$a$};
			\draw (335.44,407.33) node [anchor=north west][inner sep=0.75pt]    {$a'$};
		\end{tikzpicture}
		\caption{Move (s-ii).}\label{fig531}
	\end{center}
\end{figure}
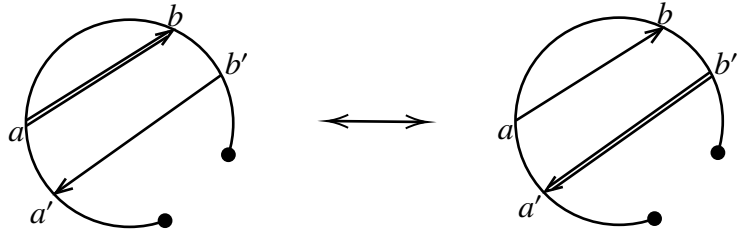

Moves (i)--(iii) still apply when only ordinary arrows are involved, and move (iii) is also allowed even if one of the arrows involved is the preferred one.
Two singular virtual open strings are said to be \textit{homotopic} if one can be obtained from the other by a finite sequence of moves consisting of (i), (ii), (iii), and (s-ii).
The homotopy equivalence relation obtained by this
collection of moves corresponds precisely to the equivalence relation on flat singular virtual knotoids with precisely one singular crossing.

A singular virtual open string can be associated with a singular based matrix. Let $\alpha$ be a singular virtual open string with a preferred arrow corresponding to the singular crossing. We define a finite set 
$G = \operatorname{arr}(\alpha) \sqcup \{s\},$
where the preferred arrow is denoted by $d \in G$.

The \textit{singular based matrix (SBM)} is a quadruple $(G, s, d, {\bf b})$, where ${\bf b}: G \times G \to \mathbb{Z}$ is a skew-symmetric map defined as follows.

\noindent
\textbf{Rule 1.}  
For $e \in \operatorname{arr}(\alpha)$, ${\bf b}(e,s)$ is computed by performing the $1$-smoothing, see Fig.~\ref{fig000}, at the flat crossing corresponding to $e$, which produces a flat virtual multi-knotoid diagram with two components ordered as $(\ell_1, \ell_2)$. Fig.~\ref{fig12} shows the 1-smoothing applied at classical crossings. Analogously, we present the case for flat crossings with no restrictions on branch orders. Denote by $\ell_1 \cap \ell_2$ the set of all flat crossings between these components in the flat virtual multi-knotoid diagram. For a flat crossing $c$ we define $\operatorname{ind}(c)$, see Fig.~\ref{2.10}, that agrees with $\operatorname{ind}(c^*)$ for classical crossing $c^*$ in Fig.~\ref{fig511}, and the intersection index 
$i(e) = \sum_{c \in \ell_1 \cap \ell_2} \operatorname{ind}(c)$. Then we  set ${\bf b}(e,s) = i(e)$. 
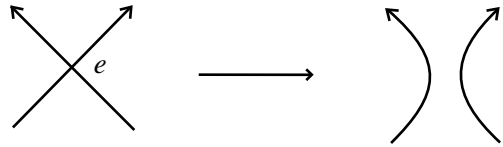
\begin{figure}[htbp]
	\begin{center}
		\tikzset{every picture/.style={line width=1.0pt}} 
		\begin{tikzpicture}[x=0.75pt,y=0.75pt,yscale=-1,xscale=1]
			\draw    (172.17,570.58) -- (234.09,632.64) ;
			\draw    (232.26,570.97) -- (173.17,631.35) ;
			\draw   (173.41,576.83) -- (172.66,570.9) -- (178.43,572.09) ;
			\draw   (226.79,572.25) -- (232.51,570.84) -- (231.98,576.79) ;
			\draw   (362.13,577.9) -- (360.73,572.18) -- (366.69,572.72) ;
			\draw    (362.89,639.17) .. controls (387.17,615.07) and (391.06,598.38) .. (361.37,572.5) ;
			\draw   (410.79,572.73) -- (416.59,571.69) -- (415.68,577.6) ;
			\draw    (416.27,572.22) .. controls (391.88,595.37) and (391.34,614.7) .. (417.57,638.9) ;
			\draw    (266.39,604.87) -- (322.53,604.93) ;
			\draw   (320.32,602.18) -- (323.12,604.88) -- (320.35,607.61) ;
			\draw (212.17,595.97) node [anchor=north west][inner sep=0.75pt]  [font=\small]  {$e$};		
		\end{tikzpicture}
		\caption{Applying 1-smoothing at the flat crossing corresponding to $e$.} \label{fig000}
	\end{center}
\end{figure}

\begin{figure}[htbp]
	\begin{center}
		\tikzset{every picture/.style={line width=1.0pt}} 
		\begin{tikzpicture}[x=0.75pt,y=0.75pt,yscale=-1,xscale=1]
			\draw    (210.67,2337.8) -- (149.85,2400.22) ; 
			\draw    (149.73,2337.87) -- (210.79,2400.15) ;
			\draw   (150.4,2343.85) -- (149.63,2337.72) -- (155.56,2338.96) ;
			\draw   (205.04,2339.12) -- (210.93,2337.66) -- (210.38,2343.82) ;
			\draw    (346.2,2338.42) -- (285.38,2400.84) ;
			\draw    (285.25,2338.49) -- (346.32,2400.77) ;
			\draw   (285.92,2344.47) -- (285.15,2338.34) -- (291.09,2339.58) ;
			\draw   (340.57,2339.74) -- (346.46,2338.28) -- (345.91,2344.44) ;
			\draw (134.27,2320.64) node [anchor=north west][inner sep=0.75pt]  [font=\normalsize]  {$\ell_2$};
			\draw (210,2321.26) node [anchor=north west][inner sep=0.75pt]  [font=\normalsize]  {$\ell_1$};
			\draw (347.36,2320.64) node [anchor=north west][inner sep=0.75pt]  [font=\normalsize]  {$\ell_2$};
			\draw (270.41,2321.26) node [anchor=north west][inner sep=0.75pt]  [font=\normalsize]  {$\ell_1$};
			\draw (188.76,2362) node [anchor=north west][inner sep=0.75pt]    {$+$};
			\draw (327.46,2362) node [anchor=north west][inner sep=0.75pt]    {$-$};	
			\draw (162.76,2364) node [anchor=north west][inner sep=0.75pt]    {$c$};
			\draw (296.46,2364) node [anchor=north west][inner sep=0.75pt]    {$c$};	
		\end{tikzpicture}
		\caption{The $\operatorname{ind}(c)$ for $c \in \ell_1 \cap \ell_2$.} \label{2.10}
	\end{center}
\end{figure}
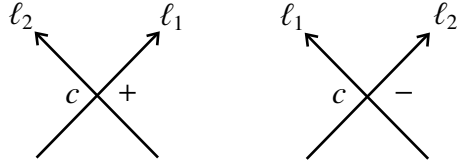

\noindent
\textbf{Rule 2.} For $e, f \in \operatorname{arr}(\alpha)$, let $e=(a,b)$ and $f=(c,d)$. Denote by $(ab)^\circ$ and $(cd)^\circ$ the interiors of the arcs $ab$ and $cd$, respectively. Define $ab \cdot cd$ as the number of arrows with tails in $(ab)^\circ$ and heads in $(cd)^\circ$ minus the number of arrows with tails in $(cd)^\circ$ and heads in $(ab)^\circ$. Figure Fig.~\ref{fig532} demonstrates an arrow $\zeta$ with tail in $(ab)^\circ$ and head in $(cd)^\circ$. It should be noted that an arc $ab$ (or $cd$) is allowed to be disconnected, in this case the tail and the head of $\alpha$ may lie between the endpoints of the arc $ab$ (or $cd$).

\begin{figure}[htbp]
	\begin{center}
		\tikzset{every picture/.style={line width=1.0pt}} 
		\begin{tikzpicture}[x=0.75pt,y=0.75pt,yscale=-1,xscale=1]
			\draw  [draw opacity=0] (148.74,582) .. controls (122.04,581.4) and (100.3,560.65) .. (98.12,534.36) -- (149.94,530.01) -- cycle ; \draw  [color={rgb, 255:red, 74; green, 144; blue, 226 }  ,draw opacity=1 ] (148.74,582) .. controls (122.04,581.4) and (100.3,560.65) .. (98.12,534.36) ;  
			\draw    (107.64,499.59) -- (193.79,501.9) ;
			\draw  [draw opacity=0] (193.22,501.18) .. controls (198.73,509.43) and (201.94,519.35) .. (201.94,530.01) .. controls (201.94,535.51) and (201.09,540.82) .. (199.5,545.79) -- (149.94,530.01) -- cycle ; \draw   (193.22,501.18) .. controls (198.73,509.43) and (201.94,519.35) .. (201.94,530.01) .. controls (201.94,535.51) and (201.09,540.82) .. (199.5,545.79) ;  
			\draw  [fill={rgb, 255:red, 0; green, 0; blue, 0 }  ,fill opacity=1 ] (196.47,545.79) .. controls (196.47,544.11) and (197.83,542.76) .. (199.5,542.76) .. controls (201.18,542.76) and (202.54,544.11) .. (202.54,545.79) .. controls (202.54,547.47) and (201.18,548.83) .. (199.5,548.83) .. controls (197.83,548.83) and (196.47,547.47) .. (196.47,545.79) -- cycle ;
			\draw  [fill={rgb, 255:red, 0; green, 0; blue, 0 }  ,fill opacity=1 ] (164.69,578.89) .. controls (164.69,577.21) and (166.05,575.85) .. (167.73,575.85) .. controls (169.41,575.85) and (170.77,577.21) .. (170.77,578.89) .. controls (170.77,580.57) and (169.41,581.93) .. (167.73,581.93) .. controls (166.05,581.93) and (164.69,580.57) .. (164.69,578.89) -- cycle ;
			\draw    (98.1,534.67) -- (148.31,581.5) ;
			\draw   (145.21,573.76) -- (148.27,581.47) -- (140.51,578.53) ;
			\draw   (115.89,503.33) -- (108.55,499.46) -- (116.36,496.65) ;
			\draw  [draw opacity=0] (108.12,499.11) .. controls (117.59,486.31) and (132.8,478.01) .. (149.94,478.01) .. controls (168.17,478.01) and (184.2,487.39) .. (193.49,501.58) -- (149.94,530.01) -- cycle ; 
			\draw  [color={rgb, 255:red, 248; green, 72; blue, 61 }  ,draw opacity=1 ] (108.12,499.11) .. controls (117.59,486.31) and (132.8,478.01) .. (149.94,478.01) .. controls (168.17,478.01) and (184.2,487.39) .. (193.49,501.58) ;  
			\draw  [draw opacity=0] (98.16,534.78) .. controls (98.01,533.21) and (97.94,531.62) .. (97.94,530.01) .. controls (97.94,518.46) and (101.71,507.79) .. (108.08,499.16) -- (149.94,530.01) -- cycle ; \draw   (98.16,534.78) .. controls (98.01,533.21) and (97.94,531.62) .. (97.94,530.01) .. controls (97.94,518.46) and (101.71,507.79) .. (108.08,499.16) ;  
			\draw    (115.99,569.13) -- (154.84,477.86) ;
			\draw   (122.21,563.1) -- (116.12,568.74) -- (116.06,560.44) ;
			\draw  [draw opacity=0] (167.73,578.89) .. controls (162.18,580.91) and (156.19,582.01) .. (149.94,582.01) .. controls (149.45,582.01) and (148.97,582.01) .. (148.48,581.99) -- (149.94,530.01) -- cycle ; \draw   (167.73,578.89) .. controls (162.18,580.91) and (156.19,582.01) .. (149.94,582.01) .. controls (149.45,582.01) and (148.97,582.01) .. (148.48,581.99) ;  
			\draw (140.91,582.85) node [anchor=north west][inner sep=0.75pt]  [font=\small]  {$d$};
			\draw (98.68,485.42) node [anchor=north west][inner sep=0.75pt]  [font=\small]  {$b$};
			\draw (192.71,486.93) node [anchor=north west][inner sep=0.75pt]  [font=\small]  {$a$};
			\draw (88.52,531.58) node [anchor=north west][inner sep=0.75pt]  [font=\small]  {$c$};
			\draw (135.5,461.75) node [anchor=north west][inner sep=0.75pt]  [font=\small,color={rgb, 255:red, 248; green, 72; blue, 61 }  ,opacity=1 ]  {$(ab)^\circ$};
			\draw (94.5,574.75) node [anchor=north west][inner sep=0.75pt]  [font=\small,color={rgb, 255:red, 74; green, 144; blue, 226 }  ,opacity=1 ]  {$(cd)^\circ$};
			\draw (136.81,517.24) node [anchor=north west][inner sep=0.75pt]  [font=\small]  {$\zeta $};		
		\end{tikzpicture}
		\caption{An arrow $\zeta$ with tail in $(ab)^\circ$ and head in $(cd)^\circ$.} \label{fig532}
	\end{center}
\end{figure}
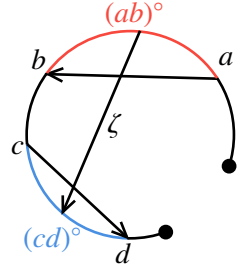

Let $\epsilon(e,f)\in\{-1,0,1\}$ be determined by whether arcs $e$ and $f$ are linked negatively, unlinked, or linked positively, see Fig.~\ref{fig7}. Then we set ${\bf b}(e,f) = ab \cdot cd + \epsilon(e,f)$.   
\begin{figure}[htbp]
	\begin{center}
		\tikzset{every picture/.style={line width=1.0pt}}  
		\begin{tikzpicture}[x=0.75pt,y=0.75pt,yscale=-1.0,xscale=1.0]  
			\draw  [draw opacity=0] (154.11,318.29) .. controls (149.85,319.86) and (145.24,320.72) .. (140.44,320.72) .. controls (118.51,320.72) and (100.73,302.83) .. (100.73,280.76) .. controls (100.73,258.69) and (118.51,240.8) .. (140.44,240.8) .. controls (162.38,240.8) and (180.16,258.69) .. (180.16,280.76) .. controls (180.16,286.17) and (179.09,291.33) .. (177.16,296.03) -- (140.44,280.76) -- cycle ; \draw   (154.11,318.29) .. controls (149.85,319.86) and (145.24,320.72) .. (140.44,320.72) .. controls (118.51,320.72) and (100.73,302.83) .. (100.73,280.76) .. controls (100.73,258.69) and (118.51,240.8) .. (140.44,240.8) .. controls (162.38,240.8) and (180.16,258.69) .. (180.16,280.76) .. controls (180.16,286.17) and (179.09,291.33) .. (177.16,296.03) ;  
			\draw    (100.81,279.99) -- (180.13,279.99) ;
			\draw    (139.43,241.13) -- (139.58,320.42) ;
			\draw   (173.93,277.02) -- (179.79,279.96) -- (173.93,282.91) ;
			\draw    (215.36,278.23) -- (294.68,278.23) ;
			\draw    (253.98,239.37) -- (254.13,318.66) ;
			\draw   (288.48,275.27) -- (294.34,278.21) -- (288.48,281.15) ;
			\draw   (257.13,312.74) -- (254.12,318.6) -- (251.27,312.66) ;
			\draw   (136.51,247.12) -- (139.42,241.21) -- (142.37,247.1) ;
			\draw    (369.03,320.47) -- (368.99,240.65) ;
			\draw    (359.25,241.32) -- (359.37,319.83) ;
			\draw   (366.04,247.48) -- (368.97,241.58) -- (371.9,247.48) ;
			\draw   (362.37,312.71) -- (359.36,318.57) -- (356.52,312.63) ;
			\draw  [fill={rgb, 255:red, 17; green, 16; blue, 16 }  ,fill opacity=1 ] (151.63,318.31) .. controls (151.65,319.68) and (152.76,320.78) .. (154.13,320.77) .. controls (155.49,320.76) and (156.59,319.64) .. (156.58,318.27) .. controls (156.57,316.9) and (155.45,315.8) .. (154.09,315.81) .. controls (152.72,315.82) and (151.62,316.94) .. (151.63,318.31) -- cycle ;
			\draw  [fill={rgb, 255:red, 17; green, 16; blue, 16 }  ,fill opacity=1 ] (174.68,296.05) .. controls (174.7,297.42) and (175.81,298.52) .. (177.18,298.51) .. controls (178.54,298.5) and (179.64,297.38) .. (179.63,296.01) .. controls (179.62,294.64) and (178.5,293.54) .. (177.14,293.55) .. controls (175.77,293.56) and (174.67,294.68) .. (174.68,296.05) -- cycle ;
			\draw  [draw opacity=0] (269.11,316.54) .. controls (264.85,318.11) and (260.24,318.97) .. (255.44,318.97) .. controls (233.51,318.97) and (215.73,301.08) .. (215.73,279.01) .. controls (215.73,256.94) and (233.51,239.05) .. (255.44,239.05) .. controls (277.38,239.05) and (295.16,256.94) .. (295.16,279.01) .. controls (295.16,284.42) and (294.09,289.58) .. (292.16,294.28) -- (255.44,279.01) -- cycle ; \draw   (269.11,316.54) .. controls (264.85,318.11) and (260.24,318.97) .. (255.44,318.97) .. controls (233.51,318.97) and (215.73,301.08) .. (215.73,279.01) .. controls (215.73,256.94) and (233.51,239.05) .. (255.44,239.05) .. controls (277.38,239.05) and (295.16,256.94) .. (295.16,279.01) .. controls (295.16,284.42) and (294.09,289.58) .. (292.16,294.28) ;  
			\draw  [fill={rgb, 255:red, 17; green, 16; blue, 16 }  ,fill opacity=1 ] (266.63,316.56) .. controls (266.65,317.93) and (267.76,319.03) .. (269.13,319.02) .. controls (270.49,319.01) and (271.59,317.89) .. (271.58,316.52) .. controls (271.57,315.15) and (270.45,314.05) .. (269.09,314.06) .. controls (267.72,314.07) and (266.62,315.19) .. (266.63,316.56) -- cycle ;
			\draw  [fill={rgb, 255:red, 17; green, 16; blue, 16 }  ,fill opacity=1 ] (289.68,294.3) .. controls (289.7,295.67) and (290.81,296.77) .. (292.18,296.76) .. controls (293.54,296.75) and (294.64,295.63) .. (294.63,294.26) .. controls (294.62,292.89) and (293.5,291.79) .. (292.14,291.8) .. controls (290.77,291.81) and (289.67,292.93) .. (289.68,294.3) -- cycle ;
			\draw  [draw opacity=0] (382.76,317.74) .. controls (378.5,319.31) and (373.9,320.17) .. (369.1,320.17) .. controls (347.16,320.17) and (329.38,302.28) .. (329.38,280.21) .. controls (329.38,258.14) and (347.16,240.25) .. (369.1,240.25) .. controls (391.03,240.25) and (408.81,258.14) .. (408.81,280.21) .. controls (408.81,285.62) and (407.74,290.78) .. (405.81,295.48) -- (369.1,280.21) -- cycle ; \draw   (382.76,317.74) .. controls (378.5,319.31) and (373.9,320.17) .. (369.1,320.17) .. controls (347.16,320.17) and (329.38,302.28) .. (329.38,280.21) .. controls (329.38,258.14) and (347.16,240.25) .. (369.1,240.25) .. controls (391.03,240.25) and (408.81,258.14) .. (408.81,280.21) .. controls (408.81,285.62) and (407.74,290.78) .. (405.81,295.48) ;  
			\draw  [fill={rgb, 255:red, 17; green, 16; blue, 16 }  ,fill opacity=1 ] (380.29,317.76) .. controls (380.3,319.13) and (381.41,320.23) .. (382.78,320.22) .. controls (384.14,320.21) and (385.24,319.09) .. (385.23,317.72) .. controls (385.22,316.35) and (384.1,315.25) .. (382.74,315.26) .. controls (381.37,315.27) and (380.28,316.39) .. (380.29,317.76) -- cycle ;
			\draw  [fill={rgb, 255:red, 17; green, 16; blue, 16 }  ,fill opacity=1 ] (403.34,295.5) .. controls (403.35,296.87) and (404.46,297.97) .. (405.83,297.96) .. controls (407.19,297.95) and (408.29,296.83) .. (408.28,295.46) .. controls (408.27,294.09) and (407.15,292.99) .. (405.79,293) .. controls (404.42,293.01) and (403.33,294.13) .. (403.34,295.5) -- cycle ;
			\draw (132.25,222.53) node [anchor=north west][inner sep=0.75pt]    {$f$};
			\draw (181.18,273.49) node [anchor=north west][inner sep=0.75pt]    {$e$};
			\draw (296.64,273.99) node [anchor=north west][inner sep=0.75pt]    {$e$};
			\draw (245.63,220.36) node [anchor=north west][inner sep=0.75pt]    {$f$};
			\draw (363.45,227.08) node [anchor=north west][inner sep=0.75pt]    {$e$};
			\draw (350.11,323.11) node [anchor=north west][inner sep=0.75pt]    {$f$};					
		\end{tikzpicture}
		\centerline{$\epsilon (e,f) = 1$ \quad\qquad $\epsilon(e,f) = -1$ \quad\qquad $\epsilon (e,f) = 0$} 
		\caption{Cases when arrows $f$ and $e$ are linked positively, negatively, or are unlinked.}
		\label{fig7}
	\end{center}
\end{figure}
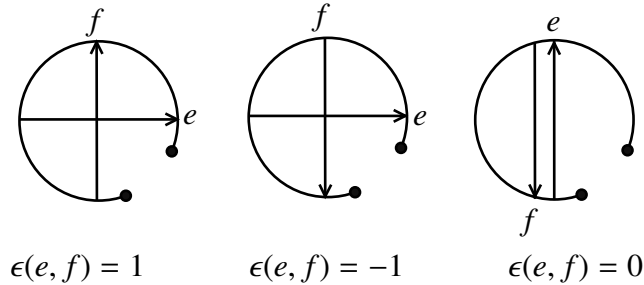

An element $g \in G - \{s, d\}$ is said to be \textit{annihilating} when ${\bf b}(g, h) = 0$ holds for all $h\in G$. An element $g\in G - \left\{s, d\right\}$ is termed a \textit{core element} if ${\bf b}(g, h)$ equals ${\bf b}(s, h)$ for each $h\in G$. Two elements $g_1, g_2\in G - \left\{s, d\right\}$ are said to be \textit{complementary} provided that ${\bf b}(g_1, h) + {\bf b}(g_2, h) = {\bf b}(s, h)$ for all $h\in G$. A distinguished element $g\in \left\{s, d\right\}$ is referred to as \textit{annihilating-like} if ${\bf b}(g, h) = 0$ for all $h\in G$. We term $d$ \textit{core-like} when ${\bf b}(d, h) = {\bf b}(s, h)$ for each $h\in G$.

Turaev first introduced elementary extension operations for based matrices in~\cite{Tur2}, and later Henrich defined similar operations $\widetilde{M}_1, \widetilde{M}_2, \widetilde{M}_3$ and $N$ in~\cite{Hen}. These operations are capable of transforming one SBM into another, and their specific definitions are as follows.

\begin{itemize}
	\item $\widetilde{M}_1$ transforms the SBM $(G, s, d, {\bf b} : G \times G \to H)$ into $(G_1 = G \sqcup \{g\}, s, d, {\bf b}_1 : G_1 \times G_1 \to H)$ for some $g$. Here, ${\bf b}_1$ is an extension of ${\bf b}$ such that  ${\bf b}_1(g, h) = 0$ for each $h \in G_1$. 
	\item $\widetilde{M}_2$ transforms $(G, s, d, {\bf b} : G \times G \to H)$ into $(G_2 = G \sqcup \{g\}, s, d, {\bf b}_2 : G_2 \times G_2 \to H)$ for some $g$. Here ${\bf b}_2$ is an extension ${\bf b}$ such that ${\bf b}_2(g, h) = {\bf b}_2(s, h)$ for all $h\in G_2.$ 
	\item $\widetilde{M}_3$ maps $(G, s, d, {\bf b} : G \times G \to H)$ to $(G_3 = G \sqcup \{g_i, g_j\}, s, d, {\bf b}_3 : G_3 \times G_3 \to H)$ for some $g_i$ and $g_j$. In this case, ${\bf b}_3$ is a skew-symmetric map that extends ${\bf b}$ and such that ${\bf b}_3(g_i, h) + {\bf b}_3(g_j, h) = {\bf b}_3(s, h)$ for all $h\in G_3$. 
	\item If there exists an element $g\in G$ such that ${\bf b}(g, h) + {\bf b}(d, h) = {\bf b}(s, h)$ for all $h\in G$ (i.e., $g$ and $d$ are complementary), then $N$ transforms $(G, s, d, {\bf b})$ into $(G, s, g, {\bf b})$. Essentially, this operation switches the roles of $g$ and $d$.
\end{itemize}
The operations $\widetilde{M}_1, \widetilde{M}_2, \widetilde{M}_3$ are known as \textit{elementary extensions}, while $N$ is known as a \textit{singularity switch}. For each $i=1,2,3$, we denote by $\widetilde{M}_i^{-1}$ the inverse operation for $\widetilde{M_i}$. 

\begin{definition} \cite[Def.~4.9]{Hen} \label{def:primitive} 
	{\rm 
		\begin{itemize}
			\item[(i)] Two SBMs $(G, s, d, {\bf b})$ and $(G', s', d', {\bf b}')$ are \textit{isomorphic} if there exists a bijection from $G$ to $G'$ that sends $s$ to $s'$, $d$ to $d'$, and transforms ${\bf b}$ into ${\bf b}'$. 
			\item[(ii)] An SBM $(G, s, d, {\bf b})$ is \textit{primitive} if it is impossible to obtain it from any other SBM via an elementary extension, even after performing the singularity switch operation. 
			\item[(iii)] Two SBMs are \textit{homologous} if one can be derived from the other through a finite sequence of moves, where each move is either $\widetilde{M}_1^{\pm 1}, \widetilde{M}_2^{\pm 1}, \widetilde{M}_3^{\pm 1}$ or $N$.
		\end{itemize}
	}
\end{definition}

\begin{theorem}\cite[Th.~4.12]{Hen}\label{th5.1}
	For any two homologous primitive SBMs, the second can be obtained from the first either by an isomorphism, or by combining an isomorphism with a single move of the type $\widetilde{M_1}^{-1} \circ N \circ \widetilde{M_2}$, $\widetilde{M_2}^{-1} \circ N \circ \widetilde{M_1}$, or $N$. 
\end{theorem} 

In~\cite{Pet}, it is shown that one can associate a based matrix to a long virtual string, and that two homotopic long virtual strings yield homologous based matrices. Naturally, we can associate a based matrix to a virtual open string, and two homotopic virtual open strings yield homologous based matrices. 

\begin{theorem} \cite{Pet} \label{th4.1} 
	If two singular virtual open strings are homotopic, then their associated SBMs are homologous.
\end{theorem}

%%%%%
\subsection{$\mathbf{P}_K(x,y)$ and the gluing invariant $\mathcal G(K)$ } \label{subsection4.2}

The gluing invariant was introduced by Henrich~\cite{Hen} as a universal Vassiliev invariant of order one for virtual knots. It was later extended to long virtual knots by Petit~\cite{Pet}, and thus naturally defines an invariant for virtual knotoids.

Let $K$ be a virtual knotoid with diagram $D$. For each classical crossing $c$ of $D$, let $D^{c}_{\operatorname{glue}}$ denote the diagram obtained by gluing the two strands at $c$, see Fig.~\ref{fig5}, and let $[\overline{D^{c}_{\operatorname{glue}}}]$ be the flat equivalence class of its shadow $\overline{D^{c}_{\operatorname{glue}}}$.
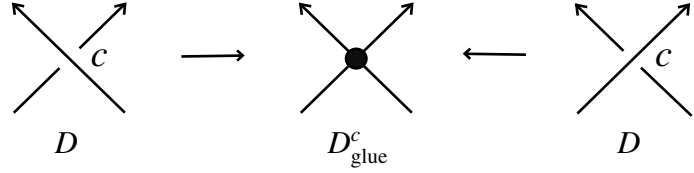
\begin{figure}[!ht]
	\begin{center}
		\tikzset{every picture/.style={line width=1.0pt}}  
		\begin{tikzpicture}[x=0.75pt,y=0.75pt,yscale=-1,xscale=1]
			\draw    (477.08,568.99) -- (421.47,623.44) ;
			\draw   (421.69,574.27) -- (420.99,568.92) -- (426.42,570) ;
			\draw   (471.93,570.14) -- (477.31,568.87) -- (476.81,574.24) ;
			\draw    (453.39,600.4) -- (478.8,624.6) ;
			\draw    (194.32,568.59) -- (171.68,590.72) ;
			\draw    (138.6,568.65) -- (194.43,622.98) ;
			\draw   (139.21,573.87) -- (138.51,568.53) -- (143.93,569.6) ;
			\draw   (189.17,569.74) -- (194.55,568.47) -- (194.05,573.84) ;
			\draw    (161.64,600.64) -- (138.71,623.04) ;
			\draw    (420.53,568.64) -- (444.61,591.98) ;
			\draw    (338.32,568.63) -- (282.71,623.08) ;
			\draw    (282.6,568.69) -- (338.43,623.02) ;
			\draw   (283.22,573.91) -- (282.51,568.57) -- (287.94,569.64) ;
			\draw   (333.17,569.78) -- (338.56,568.51) -- (338.06,573.88) ;
			\draw  [fill={rgb, 255:red, 14; green, 13; blue, 13 }  ,fill opacity=1 ] (305.42,595.86) .. controls (305.42,593.16) and (307.7,590.97) .. (310.52,590.97) .. controls (313.34,590.97) and (315.62,593.16) .. (315.62,595.86) .. controls (315.62,598.55) and (313.34,600.74) .. (310.52,600.74) .. controls (307.7,600.74) and (305.42,598.55) .. (305.42,595.86) -- cycle ;
			\draw    (395.67,593.89) -- (365.63,593.45) ;
			\draw   (367.79,590.85) -- (364.94,593.39) -- (367.72,596) ;
			\draw    (222.83,594.79) -- (252.88,594.35) ;
			\draw   (250.72,591.75) -- (253.56,594.3) -- (250.79,596.9) ;
			\draw (175.76,589.63) node [anchor=north west][inner sep=0.75pt]  [font=\large]  {$c$};
			\draw (157.23,631.27) node [anchor=north west][inner sep=0.75pt]    {$D$};
			\draw (439.6,631.27) node [anchor=north west][inner sep=0.75pt]    {$D$};
			\draw (294.28,631.04) node [anchor=north west][inner sep=0.75pt]    {$D^c_{\operatorname{glue}}$};
			\draw (458.76,589.63) node [anchor=north west][inner sep=0.75pt]  [font=\large]  {$c$};
		\end{tikzpicture}
		\caption{Gluing two strands in a classical crossing to form a singular crossing.} \label{fig5}
	\end{center}
\end{figure}

Let $D^{0}_{\operatorname{sing}}$ be the diagram obtained by introducing a kink via an $\Omega_1$-move and then replacing it with a singular crossing, see Fig.~\ref{fig6}. Denote by $[\overline{D^{0}_{\operatorname{sing}}}]$ the flat equivalence class of its shadow $\overline{D^{0}_{\operatorname{sing}}}$. This class is independent of the position of the kink.

\begin{figure}[!ht]
	\begin{center}
		\tikzset{every picture/.style={line width=1.0pt}}  
		\begin{tikzpicture}[x=0.75pt,y=0.75pt,yscale=-1,xscale=1]
			\draw    (336.04,178.26) .. controls (345.68,196.59) and (345.57,214.77) .. (335.97,232.89) ;
			\draw    (286.93,178.08) -- (335.97,232.89) ;
			\draw    (336.04,178.26) -- (314.94,201.05) ;
			\draw    (173.77,177.48) -- (173.77,232.88) ;
			\draw    (307.18,209.27) -- (286.37,231.73) ;
			\draw    (222.77,206.58) -- (247.85,206.22) ;
			\draw   (246.05,204.15) -- (248.42,206.18) -- (246.1,208.27) ;
			\draw    (481.46,177.46) .. controls (491.1,195.79) and (490.98,213.97) .. (481.39,232.09) ;
			\draw    (432.35,177.28) -- (481.39,232.09) ;
			\draw    (481.46,177.46) -- (460.36,200.25) ;
			\draw    (460.36,200.25) -- (431.79,230.93) ;
			\draw    (378.19,205.78) -- (403.26,205.42) ;
			\draw   (401.46,203.34) -- (403.84,205.38) -- (401.52,207.47) ;
			\draw   (170.93,181.1) -- (173.74,177.14) -- (176.69,180.91) ;
			\draw   (287.74,183.13) -- (287.12,178.37) -- (291.93,179.33) ;
			\draw   (432.55,181.63) -- (431.92,176.88) -- (436.73,177.83) ;
			\draw  [fill={rgb, 255:red, 14; green, 13; blue, 13 }  ,fill opacity=1 ] (451.77,204.68) .. controls (451.77,201.99) and (454.05,199.8) .. (456.87,199.8) .. controls (459.69,199.8) and (461.97,201.99) .. (461.97,204.68) .. controls (461.97,207.38) and (459.69,209.57) .. (456.87,209.57) .. controls (454.05,209.57) and (451.77,207.38) .. (451.77,204.68) -- cycle ;
			\draw (166.59,238.54) node [anchor=north west][inner sep=0.75pt]    {$D$};
			\draw (282.23,240.54) node [anchor=north west][inner sep=0.75pt]    {kink on $D$};
			\draw (445.4,236.04) node [anchor=north west][inner sep=0.75pt]    {$D^0_{\operatorname{sing}}$};
		\end{tikzpicture}
		\caption{Gluing in a kink crossing.} \label{fig6}
	\end{center}
\end{figure}
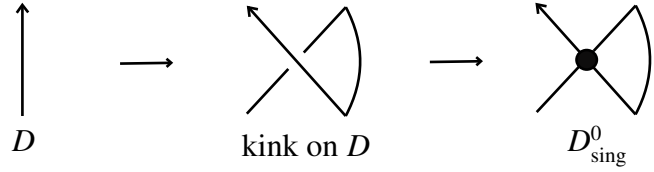 

\begin{definition}\cite[Def.~3.25]{Pet} {\rm 
		For a virtual knotoid diagram $D$, the \textit{gluing invariant} $\mathcal{G}(D)$ is defined by
		\begin{equation*}
			\mathcal{G} (D) = \sum_{c \in C(D)} \operatorname{sgn}(c) \, [\overline{D^c_{\operatorname{glue}}}] - \operatorname{wr}(D) \, [\overline{D^0_{\operatorname{sing}}}]
		\end{equation*}
		where the sum runs over the set C(D) of all classical crossings of $D$, and $\operatorname{wr}(D)$ is the writhe of $D$.
	}
\end{definition}	

A \textit{crossing change} is a local move that swaps the under- and over-strands at a given crossing in a diagram.
Two virtual knotoid diagrams are said to be \textit{homotopic} if they are related by a finite sequence of moves from $\textit{g}\mathcal{R}$ and crossing changes.

A Vassiliev invariant $\bold{U}_1$ of order one for virtual knotoids is called the \textit{universal invariant of order one} if, for any Vassiliev invariant $V$ of order one for virtual knotoids, $V$ can be recovered from $\bold{U}_1$ by using the first derivative of $V$ together with the values of $V$ on a set of representatives of the homotopy classes of virtual knotoids. 

\begin{theorem} \cite[Th.~3.26]{Pet} \label{th4.2}
	The gluing invariant $\mathcal{G}(D)$ is the universal Vassiliev invariant of order one for virtual knotoids.
\end{theorem}

\begin{theorem} \label{prop4.1}
	The gluing invariant $\mathcal{G}(K)$ is strictly stronger than $\mathbf{P}_K(x,y)$.
\end{theorem}

\begin{proof} 
	Firstly, since $\mathcal{G}$ is a universal Vassiliev invariant, by definition of universal, we directly obtain that if $\mathbf{P}_{K_1}(x,y) \neq \mathbf{P}_{K_2}(x,y)$ then $\mathcal{G}(K_1)\neq \mathcal{G}(K_2)$.
	
	Next, we demonstrate that there are two virtual knotoids $K_1$ and $K_2$ such that $\mathbf{P}_{K_1}(x,y)=\mathbf{P}_{K_2}(x,y)$, whereas $\mathcal{G}(K_1)\neq \mathcal{G}(K_2)$. Let $D_1$ and $D_2$ be diagrams of two oriented virtual knotoids $K_1$ and $K_2$ shown in Fig.~\ref{fig116}, differing by crossing changes at crossings $c_3$ and $c_4$. We show that $\mathbf{P}_{K_1}(x,y)=\mathbf{P}_{K_2}(x,y)$, whereas $\mathcal{G}(K_1)\neq \mathcal{G}(K_2)$.
	\begin{figure}[htbp]
		\begin{center}
			\tikzset{every picture/.style={line width=1.0pt}}  
			\begin{tikzpicture}[x=0.75pt,y=0.75pt,yscale=-0.95,xscale=0.95]
				\draw   (213.6,55.19) -- (220.48,59.01) -- (213.6,62.83) ;
				\draw   (492.62,57.74) -- (499.5,61.56) -- (492.62,65.38) ;
				\draw    (61.91,98.42) .. controls (71.12,56.05) and (89.79,120.05) .. (104.45,96.71) ;
				\draw    (132.3,114.92) .. controls (108.51,107.31) and (61.83,120.89) .. (81.12,97.38) ;
				\draw    (87.58,88.76) .. controls (109.12,61.38) and (108.79,127.71) .. (140.45,89.38) ;
				\draw    (104.45,96.71) .. controls (117.12,71.38) and (133.12,82.71) .. (135.45,87.05) ;
				\draw    (140.45,89.38) .. controls (176.45,47) and (238.45,57.67) .. (266.78,68.8) ;
				\draw    (142.58,95.76) .. controls (145.12,102.38) and (162.03,99.12) .. (163.36,100.45) ;
				\draw    (165,88.17) .. controls (195.71,71.88) and (203.37,88.05) .. (205.93,90.05) ;
				\draw    (205.93,90.05) .. controls (224.88,117.71) and (230.12,68.89) .. (249.12,85.56) ;
				\draw    (61.91,98.42) .. controls (54.04,135.05) and (140.71,170.17) .. (200.3,93.38) ;
				\draw    (209.52,86.05) .. controls (226.93,68.05) and (238.71,123.05) .. (253.05,88.38) ;
				\draw    (256.79,92.23) .. controls (264.79,98.89) and (264.45,104.89) .. (276.58,90.85) ;
				\draw    (253.05,88.38) .. controls (270.3,59.59) and (315.08,137.88) .. (220.79,130.17) ;
				\draw    (266.78,68.8) .. controls (279.94,75.03) and (284.23,83.87) .. (276.58,90.85) ;
				\draw    (132.3,114.92) .. controls (170.15,125.45) and (214.87,129.54) .. (220.79,130.17) ;
				\draw   (113.49,90.11) .. controls (113.54,93.1) and (111.16,95.57) .. (108.16,95.62) .. controls (105.17,95.67) and (102.7,93.29) .. (102.65,90.3) .. controls (102.6,87.3) and (104.98,84.84) .. (107.97,84.78) .. controls (110.97,84.73) and (113.43,87.12) .. (113.49,90.11) -- cycle ;
				\draw   (174.2,122.8) .. controls (174.26,125.79) and (171.87,128.26) .. (168.88,128.31) .. controls (165.89,128.36) and (163.42,125.98) .. (163.37,122.99) .. controls (163.32,120) and (165.7,117.53) .. (168.69,117.48) .. controls (171.68,117.42) and (174.15,119.81) .. (174.2,122.8) -- cycle ;
				\draw   (281.03,91.33) .. controls (281.08,94.32) and (278.7,96.79) .. (275.71,96.84) .. controls (272.71,96.89) and (270.25,94.51) .. (270.19,91.52) .. controls (270.14,88.52) and (272.53,86.06) .. (275.52,86) .. controls (278.51,85.95) and (280.98,88.34) .. (281.03,91.33) -- cycle ;
				\draw   (234.53,89.38) .. controls (234.59,92.38) and (232.2,94.84) .. (229.21,94.9) .. controls (226.22,94.95) and (223.75,92.56) .. (223.7,89.57) .. controls (223.65,86.58) and (226.03,84.11) .. (229.02,84.06) .. controls (232.01,84.01) and (234.48,86.39) .. (234.53,89.38) -- cycle ;
				\draw  [fill={rgb, 255:red, 13; green, 13; blue, 13 }  ,fill opacity=1 ] (161.71,100.45) .. controls (161.71,99.54) and (162.45,98.8) .. (163.36,98.8) .. controls (164.28,98.8) and (165.02,99.54) .. (165.02,100.45) .. controls (165.02,101.37) and (164.28,102.11) .. (163.36,102.11) .. controls (162.45,102.11) and (161.71,101.37) .. (161.71,100.45) -- cycle ;
				\draw  [fill={rgb, 255:red, 13; green, 13; blue, 13 }  ,fill opacity=1 ] (163.35,88.17) .. controls (163.35,87.25) and (164.09,86.51) .. (165,86.51) .. controls (165.92,86.51) and (166.66,87.25) .. (166.66,88.17) .. controls (166.66,89.08) and (165.92,89.83) .. (165,89.83) .. controls (164.09,89.83) and (163.35,89.08) .. (163.35,88.17) -- cycle ;
				\draw    (341.91,101.09) .. controls (351.12,58.71) and (369.79,122.71) .. (384.45,99.38) ;
				\draw    (412.3,117.59) .. controls (388.51,109.98) and (341.83,123.56) .. (361.12,100.05) ;
				\draw    (367.58,91.42) .. controls (389.12,64.05) and (388.79,130.38) .. (420.45,92.05) ;
				\draw    (384.45,99.38) .. controls (397.12,74.05) and (413.12,85.38) .. (415.45,89.71) ;
				\draw    (420.45,92.05) .. controls (456.45,49.67) and (518.45,60.33) .. (546.78,71.46) ;
				\draw    (422.58,98.42) .. controls (425.12,105.05) and (442.03,101.79) .. (443.36,103.12) ;
				\draw    (445,90.84) .. controls (467.12,75.95) and (479.89,86.95) .. (482.45,88.95) ;
				\draw    (488.79,96.29) .. controls (506.79,115.95) and (509.12,74.62) .. (524.81,85.71) ;
				\draw    (341.91,101.09) .. controls (334.04,137.71) and (425.2,171.41) .. (484.79,94.62) ;
				\draw    (484.79,94.62) .. controls (508.12,58.95) and (511.79,121.56) .. (529.12,96.23) ;
				\draw    (524.81,85.71) .. controls (536.79,93.95) and (539.12,115.95) .. (556.58,93.52) ;
				\draw    (535.45,88.56) .. controls (552.45,68.23) and (595.08,140.55) .. (500.79,132.84) ;
				\draw    (546.78,71.46) .. controls (559.94,77.69) and (564.23,86.53) .. (556.58,93.52) ;
				\draw    (412.3,117.59) .. controls (450.15,128.12) and (494.87,132.2) .. (500.79,132.84) ;
				\draw   (393.49,92.77) .. controls (393.54,95.77) and (391.16,98.23) .. (388.16,98.29) .. controls (385.17,98.34) and (382.7,95.96) .. (382.65,92.96) .. controls (382.6,89.97) and (384.98,87.5) .. (387.97,87.45) .. controls (390.97,87.4) and (393.43,89.78) .. (393.49,92.77) -- cycle ;
				\draw   (455.76,125.69) .. controls (455.81,128.68) and (453.43,131.15) .. (450.44,131.2) .. controls (447.44,131.25) and (444.98,128.87) .. (444.92,125.88) .. controls (444.87,122.88) and (447.25,120.42) .. (450.25,120.36) .. controls (453.24,120.31) and (455.71,122.7) .. (455.76,125.69) -- cycle ;
				\draw   (561.03,93.99) .. controls (561.08,96.99) and (558.7,99.45) .. (555.71,99.51) .. controls (552.71,99.56) and (550.25,97.18) .. (550.19,94.18) .. controls (550.14,91.19) and (552.53,88.72) .. (555.52,88.67) .. controls (558.51,88.62) and (560.98,91) .. (561.03,93.99) -- cycle ;
				\draw   (514.53,92.05) .. controls (514.59,95.04) and (512.2,97.51) .. (509.21,97.56) .. controls (506.22,97.61) and (503.75,95.23) .. (503.7,92.24) .. controls (503.65,89.25) and (506.03,86.78) .. (509.02,86.73) .. controls (512.01,86.67) and (514.48,89.06) .. (514.53,92.05) -- cycle ;
				\draw  [fill={rgb, 255:red, 13; green, 13; blue, 13 }  ,fill opacity=1 ] (441.71,103.12) .. controls (441.71,102.21) and (442.45,101.46) .. (443.36,101.46) .. controls (444.28,101.46) and (445.02,102.21) .. (445.02,103.12) .. controls (445.02,104.04) and (444.28,104.78) .. (443.36,104.78) .. controls (442.45,104.78) and (441.71,104.04) .. (441.71,103.12) -- cycle ;
				\draw  [fill={rgb, 255:red, 13; green, 13; blue, 13 }  ,fill opacity=1 ] (443.35,90.84) .. controls (443.35,89.92) and (444.09,89.18) .. (445,89.18) .. controls (445.92,89.18) and (446.66,89.92) .. (446.66,90.84) .. controls (446.66,91.75) and (445.92,92.49) .. (445,92.49) .. controls (444.09,92.49) and (443.35,91.75) .. (443.35,90.84) -- cycle ;
				\draw (75,71) node [anchor=north west][inner sep=0.75pt]    {$c_1$};
				\draw (133,71) node [anchor=north west][inner sep=0.75pt]    {$c_2$};
				\draw (199.33,71) node [anchor=north west][inner sep=0.75pt]    {$c_3$};
				\draw (245.67,71) node [anchor=north west][inner sep=0.75pt]    {$c_4$};
				\draw (356.67,71) node [anchor=north west][inner sep=0.75pt]    {$c_1$};
				\draw (413,71) node [anchor=north west][inner sep=0.75pt]    {$c_2$};
				\draw (479.33,71) node [anchor=north west][inner sep=0.75pt]    {$c_3$};
				\draw (525.67,71) node [anchor=north west][inner sep=0.75pt]    {$c_4$};
				\draw (159.11,141.27) node [anchor=north west][inner sep=0.75pt]    {$D_1$};
				\draw (441.33,140.38) node [anchor=north west][inner sep=0.75pt]    {$D_2$};
			\end{tikzpicture}
			\caption{Example illustrating that $\mathcal{G}$ is strictly stronger than $\mathbf{P}(x,y)$.} \label{fig116}
		\end{center}
	\end{figure}
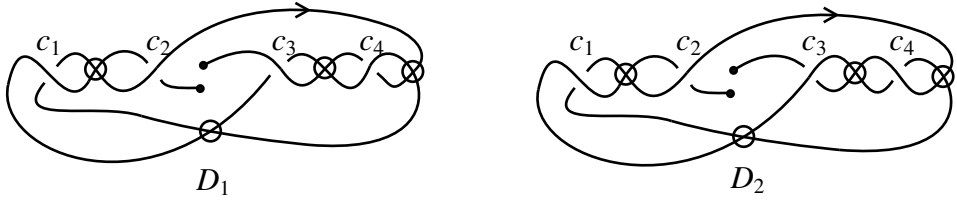 
	
	We first compute $\mathbf{P}_{K_1}(x,y)$. It follows from Fig.~\ref{fig116} that $\operatorname{sgn}(c_1)=\operatorname{sgn}(c_3)=-1$, $\operatorname{sgn}(c_2)=\operatorname{sgn}(c_4)=1$, then $\operatorname{wr}(D)=0$. And $i(c_1)=i(c_4)=-2$, $i(c_2)=i(c_3)=2$. 
	\begin{figure}[!ht]
		\begin{center}
			\tikzset{every picture/.style={line width=1.0pt}}  
			\begin{tikzpicture}[x=0.75pt,y=0.75pt,yscale=-0.95,xscale=0.95]	
				\draw  [color={rgb, 255:red, 248; green, 72; blue, 61 }  ,draw opacity=1 ]  (95.61,71.69) .. controls (93.04,77.69) and (100.75,88.26) .. (106.45,79.38) ;
				\draw    (134.3,97.59) .. controls (110.51,89.98) and (63.83,103.56) .. (83.12,80.05) ;
				\draw  [color={rgb, 255:red, 248; green, 72; blue, 61 }  ,draw opacity=1 ]   (95.61,71.69) .. controls (108.75,45.12) and (110.79,110.38) .. (142.45,72.05) ;
				\draw  [color={rgb, 255:red, 248; green, 72; blue, 61 }  ,draw opacity=1 ]  (106.45,79.38) .. controls (119.12,54.05) and (135.12,65.38) .. (137.45,69.71) ;
				\draw  [color={rgb, 255:red, 248; green, 72; blue, 61 }  ,draw opacity=1 ]   (142.45,72.05) .. controls (178.45,29.67) and (240.45,40.33) .. (268.78,51.46) ;
				\draw   [color={rgb, 255:red, 248; green, 72; blue, 61 }  ,draw opacity=1 ]   (144.58,78.42) .. controls (147.12,85.05) and (164.03,81.79) .. (165.36,83.12) ;
				\draw    [color={rgb, 255:red, 248; green, 72; blue, 61 }  ,draw opacity=1 ]  (167,70.84) .. controls (197.71,54.55) and (205.37,70.71) .. (207.93,72.71) ;
				\draw   [color={rgb, 255:red, 248; green, 72; blue, 61 }  ,draw opacity=1 ]   (207.93,72.71) .. controls (226.88,100.38) and (232.12,51.56) .. (251.12,68.23) ;
				\draw    (63.91,81.09) .. controls (56.04,117.71) and (142.71,152.84) .. (202.3,76.05) ;
				\draw    (211.52,68.71) .. controls (228.93,50.71) and (240.71,105.71) .. (255.05,71.05) ;
				\draw  [color={rgb, 255:red, 248; green, 72; blue, 61 }  ,draw opacity=1 ]   (258.79,74.89) .. controls (266.79,81.56) and (266.45,87.56) .. (278.58,73.52) ;
				\draw    (255.05,71.05) .. controls (272.3,42.25) and (317.08,120.55) .. (222.79,112.84) ;
				\draw   [color={rgb, 255:red, 248; green, 72; blue, 61 }  ,draw opacity=1 ]   (268.78,51.46) .. controls (281.94,57.69) and (286.23,66.53) .. (278.58,73.52) ;
				\draw    (134.3,97.59) .. controls (172.15,108.12) and (216.87,112.2) .. (222.79,112.84) ;
				\draw   [color={rgb, 255:red, 248; green, 72; blue, 61 }  ,draw opacity=1 ]  (115.49,72.77) .. controls (115.54,75.77) and (113.16,78.23) .. (110.16,78.29) .. controls (107.17,78.34) and (104.7,75.96) .. (104.65,72.96) .. controls (104.6,69.97) and (106.98,67.5) .. (109.97,67.45) .. controls (112.97,67.4) and (115.43,69.78) .. (115.49,72.77) -- cycle ;
				\draw   (176.2,105.47) .. controls (176.26,108.46) and (173.87,110.93) .. (170.88,110.98) .. controls (167.89,111.03) and (165.42,108.65) .. (165.37,105.65) .. controls (165.32,102.66) and (167.7,100.19) .. (170.69,100.14) .. controls (173.68,100.09) and (176.15,102.47) .. (176.2,105.47) -- cycle ;
				\draw   (283.03,73.99) .. controls (283.08,76.99) and (280.7,79.45) .. (277.71,79.51) .. controls (274.71,79.56) and (272.25,77.18) .. (272.19,74.18) .. controls (272.14,71.19) and (274.53,68.72) .. (277.52,68.67) .. controls (280.51,68.62) and (282.98,71) .. (283.03,73.99) -- cycle ;
				\draw   (236.53,72.05) .. controls (236.59,75.04) and (234.2,77.51) .. (231.21,77.56) .. controls (228.22,77.61) and (225.75,75.23) .. (225.7,72.24) .. controls (225.65,69.25) and (228.03,66.78) .. (231.02,66.73) .. controls (234.01,66.67) and (236.48,69.06) .. (236.53,72.05) -- cycle ;
				\draw  [color={rgb, 255:red, 248; green, 72; blue, 61 }  ,draw opacity=1 ][fill={rgb, 255:red, 248; green, 72; blue, 61 }  ,fill opacity=1 ] (163.71,83.12) .. controls (163.71,82.21) and (164.45,81.46) .. (165.36,81.46) .. controls (166.28,81.46) and (167.02,82.21) .. (167.02,83.12) .. controls (167.02,84.04) and (166.28,84.78) .. (165.36,84.78) .. controls (164.45,84.78) and (163.71,84.04) .. (163.71,83.12) -- cycle ;
				\draw  [color={rgb, 255:red, 248; green, 72; blue, 61 }  ,draw opacity=1 ][fill={rgb, 255:red, 248; green, 72; blue, 61 }  ,fill opacity=1 ] (165.35,70.84) .. controls (165.35,69.92) and (166.09,69.18) .. (167,69.18) .. controls (167.92,69.18) and (168.66,69.92) .. (168.66,70.84) .. controls (168.66,71.75) and (167.92,72.49) .. (167,72.49) .. controls (166.09,72.49) and (165.35,71.75) .. (165.35,70.84) -- cycle ;
				\draw   (215.6,37.85) -- (222.48,41.67) -- (215.6,45.49) ;
				\draw    (60.91,185.42) .. controls (70.12,143.05) and (88.79,207.05) .. (103.45,183.71) ;
				\draw [color={rgb, 255:red, 248; green, 72; blue, 61 }  ,draw opacity=1 ]   (131.3,201.92) .. controls (107.51,194.31) and (60.83,207.89) .. (80.12,184.38) ;
				\draw [color={rgb, 255:red, 248; green, 72; blue, 61 }  ,draw opacity=1 ]   (86.58,175.76) .. controls (108.12,148.38) and (107.79,214.71) .. (139.45,176.38) ;
				\draw    (103.45,183.71) .. controls (116.12,158.38) and (132.12,169.71) .. (134.45,174.05) ;
				\draw [color={rgb, 255:red, 248; green, 72; blue, 61 }  ,draw opacity=1 ]   (139.45,176.38) .. controls (175.45,134) and (237.45,144.67) .. (265.78,155.8) ;
				\draw    (141.58,182.76) .. controls (144.12,189.38) and (161.03,186.12) .. (162.36,187.45) ;
				\draw    (164,175.17) .. controls (194.71,158.88) and (196.96,175.58) .. (199.3,180.38) ;
				\draw [color={rgb, 255:red, 248; green, 72; blue, 61 }  ,draw opacity=1 ]   (212.68,174.73) .. controls (219.82,208.16) and (229.12,155.89) .. (248.12,172.56) ;
				\draw    (60.91,185.42) .. controls (57.53,242.75) and (217.25,235.01) .. (199.3,180.38) ;
				\draw [color={rgb, 255:red, 248; green, 72; blue, 61 }  ,draw opacity=1 ]   (212.68,174.73) .. controls (210.1,149.58) and (245.82,203.58) .. (252.05,175.38) ;
				\draw [color={rgb, 255:red, 248; green, 72; blue, 61 }  ,draw opacity=1 ]   (255.79,179.23) .. controls (263.79,185.89) and (263.45,191.89) .. (275.58,177.85) ;
				\draw [color={rgb, 255:red, 248; green, 72; blue, 61 }  ,draw opacity=1 ]   (252.05,175.38) .. controls (269.3,146.59) and (314.08,224.88) .. (219.79,217.17) ;
				\draw [color={rgb, 255:red, 248; green, 72; blue, 61 }  ,draw opacity=1 ]   (265.78,155.8) .. controls (278.94,162.03) and (283.23,170.87) .. (275.58,177.85) ;
				\draw [color={rgb, 255:red, 248; green, 72; blue, 61 }  ,draw opacity=1 ]   (131.3,201.92) .. controls (169.15,212.45) and (213.87,216.54) .. (219.79,217.17) ;
				\draw   (112.49,177.11) .. controls (112.54,180.1) and (110.16,182.57) .. (107.16,182.62) .. controls (104.17,182.67) and (101.7,180.29) .. (101.65,177.3) .. controls (101.6,174.3) and (103.98,171.84) .. (106.97,171.78) .. controls (109.97,171.73) and (112.43,174.12) .. (112.49,177.11) -- cycle ;
				\draw   (189.72,212.43) .. controls (189.77,215.42) and (187.39,217.89) .. (184.39,217.94) .. controls (181.4,217.99) and (178.93,215.61) .. (178.88,212.62) .. controls (178.83,209.62) and (181.21,207.16) .. (184.21,207.1) .. controls (187.2,207.05) and (189.67,209.44) .. (189.72,212.43) -- cycle ;
				\draw  [color={rgb, 255:red, 248; green, 72; blue, 61 }  ,draw opacity=1 ] (280.03,178.33) .. controls (280.08,181.32) and (277.7,183.79) .. (274.71,183.84) .. controls (271.71,183.89) and (269.25,181.51) .. (269.19,178.52) .. controls (269.14,175.52) and (271.53,173.06) .. (274.52,173) .. controls (277.51,172.95) and (279.98,175.34) .. (280.03,178.33) -- cycle ;
				\draw  [color={rgb, 255:red, 248; green, 72; blue, 61 }  ,draw opacity=1 ] (233.53,176.38) .. controls (233.59,179.38) and (231.2,181.84) .. (228.21,181.9) .. controls (225.22,181.95) and (222.75,179.56) .. (222.7,176.57) .. controls (222.65,173.58) and (225.03,171.11) .. (228.02,171.06) .. controls (231.01,171.01) and (233.48,173.39) .. (233.53,176.38) -- cycle ;
				\draw  [fill={rgb, 255:red, 13; green, 13; blue, 13 }  ,fill opacity=1 ] (160.71,187.45) .. controls (160.71,186.54) and (161.45,185.8) .. (162.36,185.8) .. controls (163.28,185.8) and (164.02,186.54) .. (164.02,187.45) .. controls (164.02,188.37) and (163.28,189.11) .. (162.36,189.11) .. controls (161.45,189.11) and (160.71,188.37) .. (160.71,187.45) -- cycle ;
				\draw  [fill={rgb, 255:red, 13; green, 13; blue, 13 }  ,fill opacity=1 ] (162.35,175.17) .. controls (162.35,174.25) and (163.09,173.51) .. (164,173.51) .. controls (164.92,173.51) and (165.66,174.25) .. (165.66,175.17) .. controls (165.66,176.08) and (164.92,176.83) .. (164,176.83) .. controls (163.09,176.83) and (162.35,176.08) .. (162.35,175.17) -- cycle ;
				\draw   (212.6,142.19) -- (219.48,146.01) -- (212.6,149.83) ;
				\draw [color={rgb, 255:red, 248; green, 72; blue, 61 }  ,draw opacity=1 ]   (363.7,81.23) .. controls (372.91,38.85) and (391.58,102.85) .. (406.24,79.52) ;
				\draw [color={rgb, 255:red, 248; green, 72; blue, 61 }  ,draw opacity=1 ]  (434.09,97.73) .. controls (410.29,90.12) and (363.62,103.69) .. (382.91,80.18) ;
				\draw [color={rgb, 255:red, 248; green, 72; blue, 61 }  ,draw opacity=1 ]   (389.36,71.56) .. controls (410.91,44.18) and (409.58,106.83) .. (437.24,69.85) ;
				\draw [color={rgb, 255:red, 248; green, 72; blue, 61 }  ,draw opacity=1 ]   (406.24,79.52) .. controls (418.91,54.18) and (444.44,61.69) .. (437.24,69.85) ;
				\draw    (447.58,72.26) .. controls (483.58,29.88) and (540.24,40.47) .. (568.57,51.6) ;
				\draw    (447.58,72.26) .. controls (440.64,81.76) and (463.82,81.93) .. (465.15,83.26) ;
				\draw    (466.79,70.97) .. controls (497.5,54.69) and (505.16,70.85) .. (507.72,72.85) ;
				\draw    (507.72,72.85) .. controls (526.67,100.52) and (531.91,51.7) .. (550.91,68.37) ;
				\draw [color={rgb, 255:red, 248; green, 72; blue, 61 }  ,draw opacity=1 ]   (363.7,81.23) .. controls (355.83,117.85) and (442.49,152.97) .. (502.09,76.18) ;
				\draw [color={rgb, 255:red, 248; green, 72; blue, 61 }  ,draw opacity=1 ]   (511.31,68.85) .. controls (528.72,50.85) and (540.5,105.85) .. (554.84,71.18) ;
				\draw    (558.58,75.03) .. controls (566.58,81.7) and (566.24,87.7) .. (578.37,73.65) ;
				\draw [color={rgb, 255:red, 248; green, 72; blue, 61 }  ,draw opacity=1 ]   (554.84,71.18) .. controls (572.09,42.39) and (616.87,120.69) .. (522.58,112.97) ;
				\draw    (568.57,51.6) .. controls (581.72,57.83) and (586.02,66.67) .. (578.37,73.65) ;
				\draw [color={rgb, 255:red, 248; green, 72; blue, 61 }  ,draw opacity=1 ]   (434.09,97.73) .. controls (471.94,108.26) and (516.66,112.34) .. (522.58,112.97) ;
				\draw  [color={rgb, 255:red, 248; green, 72; blue, 61 }  ,draw opacity=1 ] (415.27,72.91) .. controls (415.33,75.9) and (412.94,78.37) .. (409.95,78.42) .. controls (406.96,78.48) and (404.49,76.09) .. (404.44,73.1) .. controls (404.39,70.11) and (406.77,67.64) .. (409.76,67.59) .. controls (412.75,67.54) and (415.22,69.92) .. (415.27,72.91) -- cycle ;
				\draw  [color={rgb, 255:red, 248; green, 72; blue, 61 }  ,draw opacity=1 ] (475.99,105.6) .. controls (476.04,108.6) and (473.66,111.06) .. (470.67,111.12) .. controls (467.68,111.17) and (465.21,108.78) .. (465.16,105.79) .. controls (465.1,102.8) and (467.49,100.33) .. (470.48,100.28) .. controls (473.47,100.23) and (475.94,102.61) .. (475.99,105.6) -- cycle ;
				\draw   (582.82,74.13) .. controls (582.87,77.12) and (580.49,79.59) .. (577.49,79.64) .. controls (574.5,79.7) and (572.03,77.31) .. (571.98,74.32) .. controls (571.93,71.33) and (574.31,68.86) .. (577.31,68.81) .. controls (580.3,68.76) and (582.76,71.14) .. (582.82,74.13) -- cycle ;
				\draw   (536.32,72.19) .. controls (536.37,75.18) and (533.99,77.65) .. (531,77.7) .. controls (528.01,77.75) and (525.54,75.37) .. (525.49,72.38) .. controls (525.43,69.38) and (527.82,66.92) .. (530.81,66.86) .. controls (533.8,66.81) and (536.27,69.2) .. (536.32,72.19) -- cycle ;
				\draw  [fill={rgb, 255:red, 13; green, 13; blue, 13 }  ,fill opacity=1 ] (463.49,83.26) .. controls (463.49,82.34) and (464.24,81.6) .. (465.15,81.6) .. controls (466.07,81.6) and (466.81,82.34) .. (466.81,83.26) .. controls (466.81,84.17) and (466.07,84.92) .. (465.15,84.92) .. controls (464.24,84.92) and (463.49,84.17) .. (463.49,83.26) -- cycle ;
				\draw  [fill={rgb, 255:red, 13; green, 13; blue, 13 }  ,fill opacity=1 ] (465.13,70.97) .. controls (465.13,70.06) and (465.88,69.32) .. (466.79,69.32) .. controls (467.71,69.32) and (468.45,70.06) .. (468.45,70.97) .. controls (468.45,71.89) and (467.71,72.63) .. (466.79,72.63) .. controls (465.88,72.63) and (465.13,71.89) .. (465.13,70.97) -- cycle ;
				\draw   (515.39,37.99) -- (522.27,41.81) -- (515.39,45.63) ;
				\draw [color={rgb, 255:red, 248; green, 72; blue, 61 }  ,draw opacity=1 ]   (360.7,185.56) .. controls (369.91,143.18) and (388.58,207.18) .. (403.24,183.85) ;
				\draw    (431.09,202.06) .. controls (407.29,194.45) and (360.62,208.03) .. (379.91,184.52) ;
				\draw    (386.36,175.9) .. controls (407.91,148.52) and (407.58,214.85) .. (439.24,176.52) ;
				\draw [color={rgb, 255:red, 248; green, 72; blue, 61 }  ,draw opacity=1 ]   (403.24,183.85) .. controls (415.91,158.52) and (431.91,169.85) .. (434.24,174.18) ;
				\draw    (439.24,176.52) .. controls (475.24,134.14) and (537.24,144.8) .. (565.57,155.93) ;
				\draw [color={rgb, 255:red, 248; green, 72; blue, 61 }  ,draw opacity=1 ]   (441.36,182.9) .. controls (443.91,189.52) and (460.82,186.26) .. (462.15,187.59) ;
				\draw [color={rgb, 255:red, 248; green, 72; blue, 61 }  ,draw opacity=1 ]   (463.79,175.31) .. controls (494.5,159.02) and (502.16,175.18) .. (504.72,177.18) ;
				\draw [color={rgb, 255:red, 248; green, 72; blue, 61 }  ,draw opacity=1 ]   (504.72,177.18) .. controls (523.67,204.85) and (541.95,146.44) .. (547.95,175.58) ;
				\draw [color={rgb, 255:red, 248; green, 72; blue, 61 }  ,draw opacity=1 ]   (360.7,185.56) .. controls (352.83,222.18) and (439.49,257.31) .. (499.09,180.52) ;
				\draw [color={rgb, 255:red, 248; green, 72; blue, 61 }  ,draw opacity=1 ]   (508.31,173.18) .. controls (525.72,155.18) and (552.52,208.16) .. (547.95,175.58) ;
				\draw    (556.23,182.44) .. controls (562.23,186.44) and (566.52,186.16) .. (575.37,177.99) ;
				\draw    (555.38,171.01) .. controls (575.95,160.44) and (613.87,225.02) .. (519.58,217.31) ;
				\draw    (565.57,155.93) .. controls (578.72,162.16) and (583.02,171) .. (575.37,177.99) ;
				\draw    (431.09,202.06) .. controls (468.94,212.59) and (513.66,216.67) .. (519.58,217.31) ;
				\draw   (412.27,177.25) .. controls (412.33,180.24) and (409.94,182.71) .. (406.95,182.76) .. controls (403.96,182.81) and (401.49,180.43) .. (401.44,177.43) .. controls (401.39,174.44) and (403.77,171.97) .. (406.76,171.92) .. controls (409.75,171.87) and (412.22,174.25) .. (412.27,177.25) -- cycle ;
				\draw   (472.99,209.94) .. controls (473.04,212.93) and (470.66,215.4) .. (467.67,215.45) .. controls (464.68,215.5) and (462.21,213.12) .. (462.16,210.13) .. controls (462.1,207.13) and (464.49,204.67) .. (467.48,204.61) .. controls (470.47,204.56) and (472.94,206.94) .. (472.99,209.94) -- cycle ;
				\draw   (579.82,178.47) .. controls (579.87,181.46) and (577.49,183.93) .. (574.49,183.98) .. controls (571.5,184.03) and (569.03,181.65) .. (568.98,178.65) .. controls (568.93,175.66) and (571.31,173.19) .. (574.31,173.14) .. controls (577.3,173.09) and (579.76,175.47) .. (579.82,178.47) -- cycle ;
				\draw  [color={rgb, 255:red, 248; green, 72; blue, 61 }  ,draw opacity=1 ] (536.75,175.38) .. controls (536.8,178.37) and (534.42,180.84) .. (531.43,180.89) .. controls (528.44,180.89) and (525.97,178.56) .. (525.92,175.57) .. controls (525.86,172.58) and (528.25,170.11) .. (531.24,170.06) .. controls (534.23,170) and (536.7,172.39) .. (536.75,175.38) -- cycle ;
				\draw  [color={rgb, 255:red, 248; green, 72; blue, 61 }  ,draw opacity=1 ][fill={rgb, 255:red, 248; green, 72; blue, 61 }  ,fill opacity=1 ] (460.49,187.59) .. controls (460.49,186.68) and (461.24,185.94) .. (462.15,185.94) .. controls (463.07,185.94) and (463.81,186.68) .. (463.81,187.59) .. controls (463.81,188.51) and (463.07,189.25) .. (462.15,189.25) .. controls (461.24,189.25) and (460.49,188.51) .. (460.49,187.59) -- cycle ;
				\draw  [color={rgb, 255:red, 248; green, 72; blue, 61 }  ,draw opacity=1 ][fill={rgb, 255:red, 248; green, 72; blue, 61 }  ,fill opacity=1 ]  (462.13,175.31) .. controls (462.13,174.39) and (462.88,173.65) .. (463.79,173.65) .. controls (464.71,173.65) and (465.45,174.39) .. (465.45,175.31) .. controls (465.45,176.22) and (464.71,176.96) .. (463.79,176.96) .. controls (462.88,176.96) and (462.13,176.08) .. (462.13,175.31) -- cycle ;
				\draw   (512.39,142.32) -- (519.27,146.14) -- (512.39,149.96) ;
				\draw    (63.91,81.09) .. controls (67.04,47.69) and (95.9,61.69) .. (83.12,80.05) ;
				\draw   (117.03,117.85) -- (123.91,121.67) -- (117.03,125.49) ;
				\draw   (428.21,65.87) -- (420.57,64.02) -- (426.18,58.5) ;
				\draw   (196.96,186.89) -- (198.86,179.26) -- (204.34,184.91) ;
				\draw    (555.38,171.01) .. controls (552.81,173.3) and (553.09,179.58) .. (556.23,182.44) ;
				\draw  [color={rgb, 255:red, 248; green, 72; blue, 61 }  ,draw opacity=1 ] (496.52,173.97) -- (491.12,168.24) -- (498.84,166.69) ;
				\draw (94.4,44) node [anchor=north west][inner sep=0.75pt]  [font=\small]  {$D_{1_1}$};
				\draw (418.4,38.4) node [anchor=north west][inner sep=0.75pt]  [font=\small]  {$D_{1_1}$};
				\draw (216.8,186.2) node [anchor=north west][inner sep=0.75pt]  [font=\small]  {$D_{1_1}$};
				\draw (524.8,185.6) node [anchor=north west][inner sep=0.75pt]  [font=\small]  {$D_{1_1}$};
				\draw (66,43.2) node [anchor=north west][inner sep=0.75pt]  [font=\small]  {$D_{1_2}$};
				\draw (174.4,185.2) node [anchor=north west][inner sep=0.75pt]  [font=\small]  {$D_{1_2}$};
				\draw (444.4,38) node [anchor=north west][inner sep=0.75pt]  [font=\small]  {$D_{1_2}$};
				\draw (553.2,185.6) node [anchor=north west][inner sep=0.75pt]  [font=\small]  {$D_{1_2}$};
				\draw (200.8,77.6) node [anchor=north west][inner sep=0.75pt]  [font=\small]  {$-$};
				\draw (130.8,161.6) node [anchor=north west][inner sep=0.75pt]  [font=\small]  {$+$};
				\draw (430.8,156.8) node [anchor=north west][inner sep=0.75pt]  [font=\small]  {$-$};
				\draw (74.4,160.8) node [anchor=north west][inner sep=0.75pt]  [font=\small]  {$+$};
				\draw (373.6,158.4) node [anchor=north west][inner sep=0.75pt]  [font=\small]  {$-$};
				\draw (548,77.6) node [anchor=north west][inner sep=0.75pt]  [font=\small]  {$+$};
				\draw (247.6,77.6) node [anchor=north west][inner sep=0.75pt]  [font=\small]  {$-$};
				\draw (500,78) node [anchor=north west][inner sep=0.75pt]  [font=\small]  {$+$};	
				\draw (159,123.45) node [anchor=north west][inner sep=0.75pt]   [align=left] {(a)};
				\draw (464,123.45) node [anchor=north west][inner sep=0.75pt]   [align=left] {(b)};
				\draw (464,234.45) node [anchor=north west][inner sep=0.75pt]   [align=left] {(d)};
				\draw (159,234.45) node [anchor=north west][inner sep=0.75pt]   [align=left] {(c)};
			\end{tikzpicture}
			\caption{(a)--(d): Diagrams obtained by orientation-preserving smoothing at classical crossings $c_1$, $c_2$, $c_3$ and $c_4$ in $D_1$.} \label{fig18}
		\end{center}
	\end{figure}
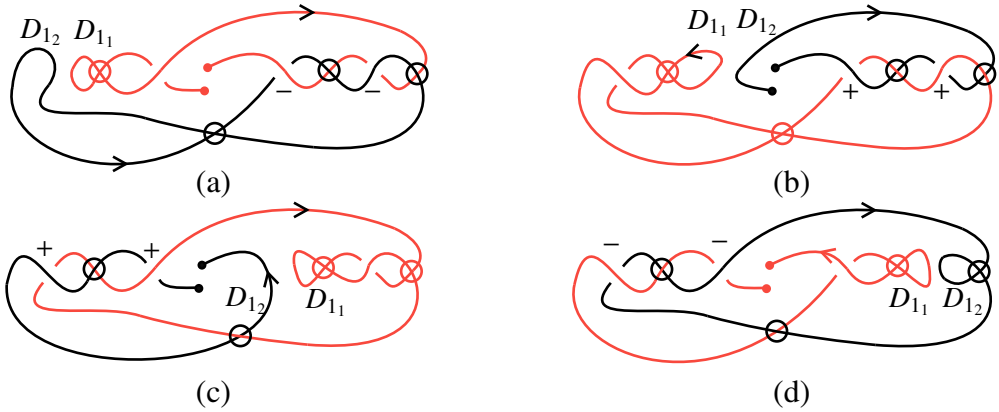 
	From Fig.~\ref{fig18}, we see that $I(c_1)=I(c_2)=I(c_3)=I(c_4)=2$, thus $c_1,c_2,c_3,c_4\in E(D)$. From the computation in the left panel of Fig.~\ref{fig19}, we obtain 
	$$
	\begin{gathered}
		W_D(c_1)= -(0-2)=2, \quad W_D(c_2)=-1-1=-2, \cr W_D(c_3)= -(2-0) =-2, \quad W_D(c_4)=1-(-1)=2.
	\end{gathered}
	$$ 
	Then
	\begin{eqnarray*}
		\mathbf{P}_{K_1}(x,y)&=&\operatorname{sgn}(c_1) i(c_1) x^{W_D(c_1)}+\operatorname{sgn}(c_2) i(c_2) x^{W_D(c_2)}+\operatorname{sgn}(c_3) i(c_3) x^{W_D(c_3)}\\ &&+\operatorname{sgn}(c_4) i(c_4) x^{W_D(c_4)} 
		=- (-2) x^2 + 2 x^{-2}- 2 x^{-2} + (-2) x^2=0
	\end{eqnarray*}
	For $K_2$, we see that $\operatorname{sgn}(c_1)=\operatorname{sgn}(c_4)=-1$, $\operatorname{sgn}(c_2)=\operatorname{sgn}(c_3)=1$ and $c_1, c_2, c_3, c_4 \in E(D)$ with $I(c_i) = 2$ for  $i=1,2,3,4$. And $i(c_1)=i(c_3)=-2$, $i(c_2)=i(c_4)=2$.  
	The corresponding weights are 
	$$
	\begin{gathered}
		W_D(c_1) = -(0-2) = 2, \quad W_D(c_2) = -1 - 1 = -2, \cr W_D(c_3) = 2 - 0 = 2, \quad W_D(c_4) = -(1-(-1)) = -2.
	\end{gathered}
	$$  
	Therefore, 
	$$
	\mathbf{P}_{K_2}(x,y) = - (-2) x^2 + 2 x^{-2} +  (-2) x^2 - 2 x^{-2}  = 0,
	$$
	which shows that $\mathbf{P}_{K_1}(x,y) = \mathbf{P}_{K_2}(x,y)$.
	\begin{figure}[!ht]
		\begin{center}
			\tikzset{every picture/.style={line width=1.0pt}}  
			\begin{tikzpicture}[x=0.75pt,y=0.75pt,yscale=-1,xscale=1]	
				\draw    (81.91,118.42) .. controls (91.12,76.05) and (109.79,140.05) .. (124.45,116.71) ;
				\draw    (152.3,134.92) .. controls (128.51,127.31) and (81.83,140.89) .. (101.12,117.38) ;
				\draw    (107.58,108.76) .. controls (129.12,81.38) and (128.79,147.71) .. (160.45,109.38) ;
				\draw    (124.45,116.71) .. controls (137.12,91.38) and (153.12,102.71) .. (155.45,107.05) ;
				\draw    (160.45,109.38) .. controls (196.45,67) and (258.45,77.67) .. (286.78,88.8) ;
				\draw    (162.58,115.76) .. controls (165.12,122.38) and (182.03,119.12) .. (183.36,120.45) ;
				\draw    (185,108.17) .. controls (215.71,91.88) and (223.37,108.05) .. (225.93,110.05) ;
				\draw    (225.93,110.05) .. controls (244.88,137.71) and (250.12,88.89) .. (269.12,105.56) ;
				\draw    (81.91,118.42) .. controls (74.04,155.05) and (160.71,190.17) .. (220.3,113.38) ;
				\draw    (229.52,106.05) .. controls (246.93,88.05) and (258.71,143.05) .. (273.05,108.38) ;
				\draw    (276.79,112.23) .. controls (284.79,118.89) and (284.45,124.89) .. (296.58,110.85) ;
				\draw    (273.05,108.38) .. controls (290.3,79.59) and (335.08,157.88) .. (240.79,150.17) ;
				\draw    (286.78,88.8) .. controls (299.94,95.03) and (304.23,103.87) .. (296.58,110.85) ;
				\draw    (152.3,134.92) .. controls (190.15,145.45) and (234.87,149.54) .. (240.79,150.17) ;
				\draw   (133.49,110.11) .. controls (133.54,113.1) and (131.16,115.57) .. (128.16,115.62) .. controls (125.17,115.67) and (122.7,113.29) .. (122.65,110.3) .. controls (122.6,107.3) and (124.98,104.84) .. (127.97,104.78) .. controls (130.97,104.73) and (133.43,107.12) .. (133.49,110.11) -- cycle ;
				\draw   (194.2,142.8) .. controls (194.26,145.79) and (191.87,148.26) .. (188.88,148.31) .. controls (185.89,148.36) and (183.42,145.98) .. (183.37,142.99) .. controls (183.32,140) and (185.7,137.53) .. (188.69,137.48) .. controls (191.68,137.42) and (194.15,139.81) .. (194.2,142.8) -- cycle ;
				\draw   (301.03,111.33) .. controls (301.08,114.32) and (298.7,116.79) .. (295.71,116.84) .. controls (292.71,116.89) and (290.25,114.51) .. (290.19,111.52) .. controls (290.14,108.52) and (292.53,106.06) .. (295.52,106) .. controls (298.51,105.95) and (300.98,108.34) .. (301.03,111.33) -- cycle ;
				\draw   (254.53,109.38) .. controls (254.59,112.38) and (252.2,114.84) .. (249.21,114.9) .. controls (246.22,114.95) and (243.75,112.56) .. (243.7,109.57) .. controls (243.65,106.58) and (246.03,104.11) .. (249.02,104.06) .. controls (252.01,104.01) and (254.48,106.39) .. (254.53,109.38) -- cycle ;
				\draw  [fill={rgb, 255:red, 13; green, 13; blue, 13 }  ,fill opacity=1 ] (181.71,120.45) .. controls (181.71,119.54) and (182.45,118.8) .. (183.36,118.8) .. controls (184.28,118.8) and (185.02,119.54) .. (185.02,120.45) .. controls (185.02,121.37) and (184.28,122.11) .. (183.36,122.11) .. controls (182.45,122.11) and (181.71,121.37) .. (181.71,120.45) -- cycle ;
				\draw  [fill={rgb, 255:red, 13; green, 13; blue, 13 }  ,fill opacity=1 ] (183.35,108.17) .. controls (183.35,107.25) and (184.09,106.51) .. (185,106.51) .. controls (185.92,106.51) and (186.66,107.25) .. (186.66,108.17) .. controls (186.66,109.08) and (185.92,109.83) .. (185,109.83) .. controls (184.09,109.83) and (183.35,109.08) .. (183.35,108.17) -- cycle ;
				\draw    (361.91,121.09) .. controls (371.12,78.71) and (389.79,142.71) .. (404.45,119.38) ;
				\draw    (432.3,137.59) .. controls (408.51,129.98) and (361.83,143.56) .. (381.12,120.05) ;
				\draw    (387.58,111.42) .. controls (409.12,84.05) and (408.79,150.38) .. (440.45,112.05) ;
				\draw    (404.45,119.38) .. controls (417.12,94.05) and (433.12,105.38) .. (435.45,109.71) ;
				\draw    (440.45,112.05) .. controls (476.45,69.67) and (538.45,80.33) .. (566.78,91.46) ;
				\draw    (442.58,118.42) .. controls (445.12,125.05) and (462.03,121.79) .. (463.36,123.12) ;
				\draw    (465,110.84) .. controls (487.12,95.95) and (499.89,106.95) .. (502.45,108.95) ;
				\draw    (508.79,116.29) .. controls (526.79,135.95) and (529.12,94.62) .. (544.81,105.71) ;
				\draw    (361.91,121.09) .. controls (354.04,157.71) and (445.2,191.41) .. (504.79,114.62) ;
				\draw    (504.79,114.62) .. controls (528.12,78.95) and (531.79,141.56) .. (549.12,116.23) ;
				\draw    (544.81,105.71) .. controls (556.79,113.95) and (559.12,135.95) .. (576.58,113.52) ;
				\draw    (555.45,108.56) .. controls (572.45,88.23) and (615.08,160.55) .. (520.79,152.84) ;
				\draw    (566.78,91.46) .. controls (579.94,97.69) and (584.23,106.53) .. (576.58,113.52) ;
				\draw    (432.3,137.59) .. controls (470.15,148.12) and (514.87,152.2) .. (520.79,152.84) ;
				\draw   (413.49,112.77) .. controls (413.54,115.77) and (411.16,118.23) .. (408.16,118.29) .. controls (405.17,118.34) and (402.7,115.96) .. (402.65,112.96) .. controls (402.6,109.97) and (404.98,107.5) .. (407.97,107.45) .. controls (410.97,107.4) and (413.43,109.78) .. (413.49,112.77) -- cycle ;
				\draw   (475.76,145.69) .. controls (475.81,148.68) and (473.43,151.15) .. (470.44,151.2) .. controls (467.44,151.25) and (464.98,148.87) .. (464.92,145.88) .. controls (464.87,142.88) and (467.25,140.42) .. (470.25,140.36) .. controls (473.24,140.31) and (475.71,142.7) .. (475.76,145.69) -- cycle ;
				\draw   (581.03,113.99) .. controls (581.08,116.99) and (578.7,119.45) .. (575.71,119.51) .. controls (572.71,119.56) and (570.25,117.18) .. (570.19,114.18) .. controls (570.14,111.19) and (572.53,108.72) .. (575.52,108.67) .. controls (578.51,108.62) and (580.98,111) .. (581.03,113.99) -- cycle ;
				\draw   (534.53,112.05) .. controls (534.59,115.04) and (532.2,117.51) .. (529.21,117.56) .. controls (526.22,117.61) and (523.75,115.23) .. (523.7,112.24) .. controls (523.65,109.25) and (526.03,106.78) .. (529.02,106.73) .. controls (532.01,106.67) and (534.48,109.06) .. (534.53,112.05) -- cycle ;
				\draw  [fill={rgb, 255:red, 13; green, 13; blue, 13 }  ,fill opacity=1 ] (461.71,123.12) .. controls (461.71,122.21) and (462.45,121.46) .. (463.36,121.46) .. controls (464.28,121.46) and (465.02,122.21) .. (465.02,123.12) .. controls (465.02,124.04) and (464.28,124.78) .. (463.36,124.78) .. controls (462.45,124.78) and (461.71,124.04) .. (461.71,123.12) -- cycle ;
				\draw  [fill={rgb, 255:red, 13; green, 13; blue, 13 }  ,fill opacity=1 ] (463.35,110.84) .. controls (463.35,109.92) and (464.09,109.18) .. (465,109.18) .. controls (465.92,109.18) and (466.66,109.92) .. (466.66,110.84) .. controls (466.66,111.75) and (465.92,112.49) .. (465,112.49) .. controls (464.09,112.49) and (463.35,111.75) .. (463.35,110.84) -- cycle ;
				\draw   (233.6,75.19) -- (240.48,79.01) -- (233.6,82.83) ;
				\draw   (512.62,77.74) -- (519.5,81.56) -- (512.62,85.38) ;
				\draw (179.11,161.27) node [anchor=north west][inner sep=0.75pt]    {$D_{1}$};
				\draw (461.33,160.38) node [anchor=north west][inner sep=0.75pt]    {$D_{2}$};
				\draw (167.27,121.79) node [anchor=north west][inner sep=0.75pt]  [font=\footnotesize]  {$0$};
				\draw (512.27,89.51) node [anchor=north west][inner sep=0.75pt]  [font=\footnotesize]  {$1$};
				\draw (367.82,150.08) node [anchor=north west][inner sep=0.75pt]  [font=\footnotesize]  {$2$};
				\draw (251.02,88.37) node [anchor=north west][inner sep=0.75pt]  [font=\footnotesize]  {$-1$};
				\draw (478.13,89.91) node [anchor=north west][inner sep=0.75pt]  [font=\footnotesize]  {$0$};
				\draw (418.77,88.7) node [anchor=north west][inner sep=0.75pt]  [font=\footnotesize]  {$1$};
				\draw (383.02,90.21) node [anchor=north west][inner sep=0.75pt]  [font=\footnotesize]  {$-1$};
				\draw (444.13,123.41) node [anchor=north west][inner sep=0.75pt]  [font=\footnotesize]  {$0$};
				\draw (229.84,88.13) node [anchor=north west][inner sep=0.75pt]  [font=\footnotesize]  {$1$};
				\draw (562.5,148.67) node [anchor=north west][inner sep=0.75pt]  [font=\footnotesize]  {$0$};
				\draw (527.52,89.21) node [anchor=north west][inner sep=0.75pt]  [font=\footnotesize]  {$-1$};
				\draw (566.69,79.03) node [anchor=north west][inner sep=0.75pt]  [font=\footnotesize]  {$-2$};
				\draw (285.56,144.41) node [anchor=north west][inner sep=0.75pt]  [font=\footnotesize]  {$0$};
				\draw (134.72,88.27) node [anchor=north west][inner sep=0.75pt]  [font=\footnotesize]  {$1$};
				\draw (86.94,146.41) node [anchor=north west][inner sep=0.75pt]  [font=\footnotesize]  {$2$};
				\draw (199.79,88.89) node [anchor=north west][inner sep=0.75pt]  [font=\footnotesize]  {$0$};
				\draw (101.88,88.42) node [anchor=north west][inner sep=0.75pt]  [font=\footnotesize]  {$-1$};
				\draw (280.19,73.25) node [anchor=north west][inner sep=0.75pt]  [font=\footnotesize]  {$-2$};
			\end{tikzpicture}
			\caption{Integer labeling for $D_1$ and $D_2$. } \label{fig19}
		\end{center}
	\end{figure}
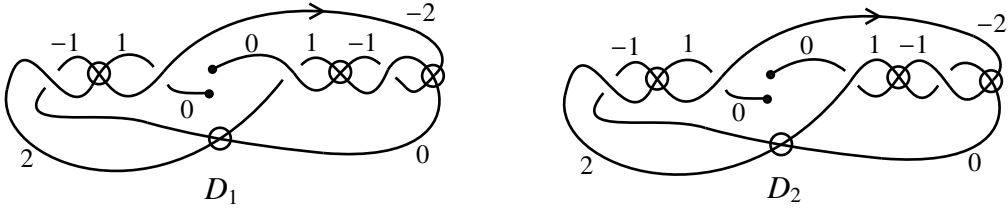 
	
	Now we consider $\mathcal{G}(K_{1})$, where $K_1$ has diagram $D_1$. Recall that 
	$$
	\mathcal{G} (D_1) = \sum_{k=1}^{4} w_{D_1} (c_k) \, [\overline{(D_1)^{c_k}_{\operatorname{glue}}}] - \operatorname{wr}(D_1) \, [\overline{(D_1)^0_{\operatorname{sing}}}]. 
	$$
	The crossings $c_3$ and $c_4$, contribute the terms $(-1)[\overline{(D_1)^{c_3}_{\operatorname{glue}}}]$ and $[\overline{(D_1)^{c_4}_{\operatorname{glue}}}]$, respectively.
	
	We now compute the SBMs $(G_{c_3}, s_{c_3}, d_{c_3}, {\bf b}_{c_3})$ and $(G_{c_4}, s_{c_4}, d_{c_4}, {\bf b}_{c_4})$ associated to these two singular virtual open strings as shown on Fig.~\ref{fig117}(a) and (b), respectively, to show that they  are distinct, that implies that their terms in $\mathcal{G}(K_{1})$ do not cancel.
	\begin{figure}[htbp]
		\begin{center}
			\tikzset{every picture/.style={line width=1.0pt}} 
			\begin{tikzpicture}[x=0.75pt,y=0.75pt,yscale=-0.95,xscale=0.95]
				\draw  [fill={rgb, 255:red, 17; green, 16; blue, 16 }  ,fill opacity=1 ] (222.73,301.94) .. controls (222.73,303.05) and (223.62,303.95) .. (224.7,303.94) .. controls (225.78,303.93) and (226.65,303.02) .. (226.64,301.9) .. controls (226.63,300.79) and (225.75,299.89) .. (224.67,299.9) .. controls (223.59,299.9) and (222.72,300.82) .. (222.73,301.94) -- cycle ;
				\draw  [fill={rgb, 255:red, 17; green, 16; blue, 16 }  ,fill opacity=1 ] (197.66,324.75) .. controls (197.67,325.86) and (198.55,326.76) .. (199.63,326.75) .. controls (200.72,326.74) and (201.59,325.83) .. (201.58,324.71) .. controls (201.57,323.6) and (200.68,322.7) .. (199.6,322.71) .. controls (198.52,322.71) and (197.65,323.63) .. (197.66,324.75) -- cycle ;
				\draw    (177.78,323.57) -- (181.6,238.51) ;
				\draw    (159.36,246.62) -- (156.36,316.49) ;
				\draw  [draw opacity=0] (199.62,324.73) .. controls (194.85,326.46) and (189.71,327.41) .. (184.35,327.41) .. controls (159.69,327.41) and (139.7,307.42) .. (139.7,282.76) .. controls (139.7,258.11) and (159.69,238.12) .. (184.35,238.12) .. controls (209,238.12) and (228.99,258.11) .. (228.99,282.76) .. controls (228.99,289.62) and (227.45,296.11) .. (224.68,301.92) -- (184.35,282.76) -- cycle ; \draw   (199.62,324.73) .. controls (194.85,326.46) and (189.71,327.41) .. (184.35,327.41) .. controls (159.69,327.41) and (139.7,307.42) .. (139.7,282.76) .. controls (139.7,258.11) and (159.69,238.12) .. (184.35,238.12) .. controls (209,238.12) and (228.99,258.11) .. (228.99,282.76) .. controls (228.99,289.62) and (227.45,296.11) .. (224.68,301.92) ;  
				\draw    (143.16,268.49) -- (209.96,246.89) ;
				\draw    (139.61,291.29) -- (222.81,264.49) ;
				\draw   (181.87,319.47) -- (178.76,326.02) -- (176.01,319.42) ;
				\draw   (202.7,246.26) -- (209.93,246.7) -- (204.85,251.72) ;
				\draw   (216.92,263.82) -- (224.15,264.25) -- (219.07,269.27) ;
				\draw   (159.43,310.39) -- (156.32,316.93) -- (153.57,310.34) ;
				\draw  [fill={rgb, 255:red, 17; green, 16; blue, 16 }  ,fill opacity=1 ] (401.39,301.94) .. controls (401.4,303.05) and (402.29,303.95) .. (403.37,303.94) .. controls (404.45,303.93) and (405.32,303.02) .. (405.31,301.9) .. controls (405.3,300.79) and (404.42,299.89) .. (403.34,299.9) .. controls (402.25,299.9) and (401.38,300.82) .. (401.39,301.94) -- cycle ;
				\draw  [fill={rgb, 255:red, 17; green, 16; blue, 16 }  ,fill opacity=1 ] (376.33,324.75) .. controls (376.34,325.86) and (377.22,326.76) .. (378.3,326.75) .. controls (379.38,326.74) and (380.25,325.83) .. (380.24,324.71) .. controls (380.24,323.6) and (379.35,322.7) .. (378.27,322.71) .. controls (377.19,322.71) and (376.32,323.63) .. (376.33,324.75) -- cycle ;
				\draw    (357.43,326.49) -- (361.38,238.07) ;
				\draw    (336.69,246.4) -- (333.88,314.01) ;
				\draw  [draw opacity=0] (378.29,324.73) .. controls (373.52,326.46) and (368.38,327.41) .. (363.01,327.41) .. controls (338.35,327.41) and (318.37,307.42) .. (318.37,282.76) .. controls (318.37,258.11) and (338.35,238.12) .. (363.01,238.12) .. controls (387.67,238.12) and (407.66,258.11) .. (407.66,282.76) .. controls (407.66,289.62) and (406.11,296.11) .. (403.35,301.92) -- (363.01,282.76) -- cycle ; \draw   (378.29,324.73) .. controls (373.52,326.46) and (368.38,327.41) .. (363.01,327.41) .. controls (338.35,327.41) and (318.37,307.42) .. (318.37,282.76) .. controls (318.37,258.11) and (338.35,238.12) .. (363.01,238.12) .. controls (387.67,238.12) and (407.66,258.11) .. (407.66,282.76) .. controls (407.66,289.62) and (406.11,296.11) .. (403.35,301.92) ;  
				\draw    (321.83,268.49) -- (388.63,246.89) ;
				\draw    (319.83,291.29) -- (403.03,264.49) ;
				\draw   (360.54,319.47) -- (357.43,326.02) -- (354.68,319.42) ;
				\draw   (381.36,246.26) -- (388.6,246.7) -- (383.51,251.72) ;
				\draw   (395.58,263.82) -- (402.82,264.25) -- (397.74,269.27) ;
				\draw    (180,323.34) -- (183.82,238.29) ;
				\draw    (338.91,245.29) -- (336.32,314.01) ;
				\draw   (338.1,310.39) -- (334.98,316.93) -- (332.23,310.34) ;
				\draw (209.6,234.57) node [anchor=north west][inner sep=0.75pt]  [font=\small]  {$c_1$};
				\draw (166.73,328.93) node [anchor=north west][inner sep=0.75pt]  [font=\small]  {$d_{c_3}=c_3$};
				\draw (146.96,319.71) node [anchor=north west][inner sep=0.75pt]  [font=\small]  {$c_4$};
				\draw (225.21,253.75) node [anchor=north west][inner sep=0.75pt]  [font=\small]  {$c_2$};
				\draw (388.26,234.57) node [anchor=north west][inner sep=0.75pt]  [font=\small]  {$c_1$};
				\draw (352.73,328.26) node [anchor=north west][inner sep=0.75pt]  [font=\small]  {$c_3$};
				\draw (295.63,318.71) node [anchor=north west][inner sep=0.75pt]  [font=\small]  {$d_{c_4}=c_4$};
				\draw (403.88,253.75) node [anchor=north west][inner sep=0.75pt]  [font=\small]  {$c_2$};
				\draw (168,348.44) node [anchor=north west][inner sep=0.75pt]    {(a)};
				\draw (352.8,348.04) node [anchor=north west][inner sep=0.75pt]    {(b)};
			\end{tikzpicture}
			\caption{Singular virtual open strings associated to gluing $D_1$ at crossings $c_3$ and $c_4$.} \label{fig117} 
		\end{center}
	\end{figure}
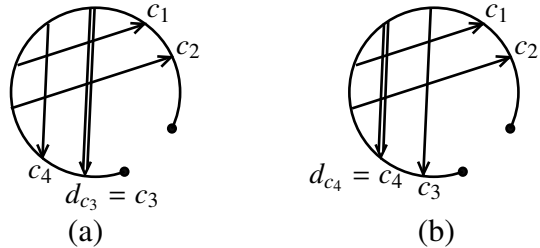
	
	\underline{SBM $(G_{c_3}, s_{c_3}$, $d_{c_3}, {\bf b}_{c_3})$.}
	Denote by ${\bf b}_{c_3}$ the skew-symmetric map in the SBM $(G_{c_3}, s_{c_3},$ $ d_{c_3},{\bf b}_{c_3})$ corresponding to the singular virtual open string associated with crossing $c_3$, shown in Fig.~\ref{fig117}(a). Equivalently, one can view this string as the flat singular virtual knotoid obtained by performing the gluing operation at crossing $c_3$ in Fig.~\ref{fig116}(left). 
	
	Following {\bf Rule~1}, we first compute ${\bf b}_{c_3}(e,s_{c_3})$. For each $e \in \{c_1,c_2,c_4,d_{c_3}\}$, the map ${\bf b}_{c_3}(e,s_{c_3})$ is defined as the sum of the signs of the flat crossings between the two components obtained by the $1$-smoothing at the crossing corresponding to $e$. We order the components such that the component on the right is $\ell_1$ and the component on the left is $\ell_2$, see Fig.~\ref{figXX}. Then we obtain:
	$$
	{\bf b}_{c_3}(c_1,s_{c_3})=-2, \quad {\bf b}_{c_3}(c_2,s_{c_3})=-2, \quad {\bf b}_{c_3}(c_4,s_{c_3})=2, \quad {\bf b}_{c_3}(d_{c_3},s_{c_3})=2.
	$$
	\begin{figure}[!ht]
		\begin{center}
			\tikzset{every picture/.style={line width=1.0pt}}  
			\begin{tikzpicture}[x=0.75pt,y=0.75pt,yscale=-0.95,xscale=0.95]	
				\draw [color={rgb, 255:red, 228; green, 43; blue, 43 }  ,draw opacity=1 ]   (88.61,81.69) .. controls (86.04,87.69) and (93.75,98.26) .. (99.45,89.38) ;
				\draw    (127.3,107.59) .. controls (103.51,99.98) and (56.83,113.56) .. (76.12,90.05) ;
				\draw [color={rgb, 255:red, 228; green, 43; blue, 43 }  ,draw opacity=1 ]   (88.61,81.69) .. controls (101.75,55.12) and (103.79,120.38) .. (135.45,82.05) ;
				\draw [color={rgb, 255:red, 228; green, 43; blue, 43 }  ,draw opacity=1 ]   (99.45,89.38) .. controls (112.12,64.05) and (126.85,73.2) .. (130.45,79.71) ;
				\draw [color={rgb, 255:red, 228; green, 43; blue, 43 }  ,draw opacity=1 ]   (135.45,82.05) .. controls (171.45,39.67) and (233.45,50.33) .. (261.78,61.46) ;
				\draw [color={rgb, 255:red, 228; green, 43; blue, 43 }  ,draw opacity=1 ]   (130.45,79.71) .. controls (139.18,96.53) and (157.03,91.79) .. (158.36,93.12) ;
				\draw [color={rgb, 255:red, 228; green, 43; blue, 43 }  ,draw opacity=1 ]   (160,80.84) .. controls (190.71,64.55) and (198.37,80.71) .. (200.93,82.71) ;
				\draw [color={rgb, 255:red, 248; green, 72; blue, 61 }  ,draw opacity=1 ]   (200.93,82.71) .. controls (219.88,110.38) and (225.12,61.56) .. (244.12,78.23) ;
				\draw    (56.91,91.09) .. controls (49.04,127.71) and (137.52,163.53) .. (199.18,83.53) ;
				\draw    (199.18,83.53) .. controls (222.85,55.53) and (233.71,115.71) .. (248.05,81.05) ;
				\draw [color={rgb, 255:red, 248; green, 72; blue, 61 }  ,draw opacity=1 ]   (244.12,78.23) .. controls (252.52,89.53) and (259.45,97.56) .. (271.58,83.52) ; 
				\draw    (248.05,81.05) .. controls (265.3,52.25) and (310.08,130.55) .. (215.79,122.84) ;
				\draw [color={rgb, 255:red, 248; green, 72; blue, 61 }  ,draw opacity=1 ]   (261.78,61.46) .. controls (274.94,67.69) and (279.23,76.53) .. (271.58,83.52) ;
				\draw    (127.3,107.59) .. controls (165.15,118.12) and (209.87,122.2) .. (215.79,122.84) ;
				\draw  [color={rgb, 255:red, 228; green, 43; blue, 43 }  ,draw opacity=1 ] (108.49,82.77) .. controls (108.54,85.77) and (106.16,88.23) .. (103.16,88.29) .. controls (100.17,88.34) and (97.7,85.96) .. (97.65,82.96) .. controls (97.6,79.97) and (99.98,77.5) .. (102.97,77.45) .. controls (105.97,77.4) and (108.43,79.78) .. (108.49,82.77) -- cycle ;
				\draw   (169.2,115.47) .. controls (169.26,118.46) and (166.87,120.93) .. (163.88,120.98) .. controls (160.89,121.03) and (158.42,118.65) .. (158.37,115.65) .. controls (158.32,112.66) and (160.7,110.19) .. (163.69,110.14) .. controls (166.68,110.09) and (169.15,112.47) .. (169.2,115.47) -- cycle ;
				\draw   (276.03,83.99) .. controls (276.08,86.99) and (273.7,89.45) .. (270.71,89.51) .. controls (267.71,89.56) and (265.25,87.18) .. (265.19,84.18) .. controls (265.14,81.19) and (267.53,78.72) .. (270.52,78.67) .. controls (273.51,78.62) and (275.98,81) .. (276.03,83.99) -- cycle ; 
				\draw   (229.53,82.05) .. controls (229.59,85.04) and (227.2,87.51) .. (224.21,87.56) .. controls (221.22,87.61) and (218.75,85.23) .. (218.7,82.24) .. controls (218.65,79.25) and (221.03,76.78) .. (224.02,76.73) .. controls (227.01,76.67) and (229.48,79.06) .. (229.53,82.05) -- cycle ;
				\draw  [color={rgb, 255:red, 228; green, 43; blue, 43 }  ,draw opacity=1 ][fill={rgb, 255:red, 248; green, 72; blue, 61 }  ,fill opacity=1 ] (156.71,93.12) .. controls (156.71,92.21) and (157.45,91.46) .. (158.36,91.46) .. controls (159.28,91.46) and (160.02,92.21) .. (160.02,93.12) .. controls (160.02,94.04) and (159.28,94.78) .. (158.36,94.78) .. controls (157.45,94.78) and (156.71,94.04) .. (156.71,93.12) -- cycle ; 
				\draw  [color={rgb, 255:red, 228; green, 43; blue, 43 }  ,draw opacity=1 ][fill={rgb, 255:red, 248; green, 72; blue, 61 }  ,fill opacity=1 ] (158.35,80.84) .. controls (158.35,79.92) and (159.09,79.18) .. (160,79.18) .. controls (160.92,79.18) and (161.66,79.92) .. (161.66,80.84) .. controls (161.66,81.75) and (160.92,82.49) .. (160,82.49) .. controls (159.09,82.49) and (158.35,81.75) .. (158.35,80.84) -- cycle ;
				\draw   (208.6,47.85) -- (215.48,51.67) -- (208.6,55.49) ;
				\draw    (53.91,195.42) .. controls (63.12,153.05) and (81.79,217.05) .. (96.45,193.71) ;
				\draw [color={rgb, 255:red, 248; green, 72; blue, 61 }  ,draw opacity=1 ]   (124.3,211.92) .. controls (100.51,204.31) and (53.74,213.56) .. (79.58,185.76) ;
				\draw [color={rgb, 255:red, 248; green, 72; blue, 61 }  ,draw opacity=1 ]   (79.58,185.76) .. controls (101.12,158.38) and (100.79,224.71) .. (132.45,186.38) ;
				\draw    (96.45,193.71) .. controls (109.12,168.38) and (125.12,179.71) .. (127.45,184.05) ;
				\draw [color={rgb, 255:red, 248; green, 72; blue, 61 }  ,draw opacity=1 ]   (132.45,186.38) .. controls (168.45,144) and (230.45,154.67) .. (258.78,165.8) ;
				\draw    (127.45,184.05) .. controls (138.03,200.13) and (154.03,196.12) .. (155.36,197.45) ;
				\draw    (157,185.17) .. controls (187.71,168.88) and (189.96,185.58) .. (192.3,190.38) ;
				\draw [color={rgb, 255:red, 248; green, 72; blue, 61 }  ,draw opacity=1 ]   (205.68,184.73) .. controls (212.82,218.16) and (222.12,165.89) .. (241.12,182.56) ;
				\draw    (53.91,195.42) .. controls (50.53,252.75) and (210.25,245.01) .. (192.3,190.38) ;
				\draw [color={rgb, 255:red, 248; green, 72; blue, 61 }  ,draw opacity=1 ]   (205.68,184.73) .. controls (203.1,159.58) and (232.88,212.13) .. (245.05,185.38) ; 
				\draw [color={rgb, 255:red, 248; green, 72; blue, 61 }  ,draw opacity=1 ]   (241.12,182.56) .. controls (250.31,194.7) and (256.45,201.89) .. (268.58,187.85) ;
				\draw [color={rgb, 255:red, 248; green, 72; blue, 61 }  ,draw opacity=1 ]   (245.05,185.38) .. controls (262.3,156.59) and (307.08,234.88) .. (212.79,227.17) ;
				\draw [color={rgb, 255:red, 248; green, 72; blue, 61 }  ,draw opacity=1 ]   (258.78,165.8) .. controls (271.94,172.03) and (276.23,180.87) .. (268.58,187.85) ;
				\draw [color={rgb, 255:red, 248; green, 72; blue, 61 }  ,draw opacity=1 ]   (124.3,211.92) .. controls (162.15,222.45) and (206.87,226.54) .. (212.79,227.17) ; 
				\draw   (105.49,187.11) .. controls (105.54,190.1) and (103.16,192.57) .. (100.16,192.62) .. controls (97.17,192.67) and (94.7,190.29) .. (94.65,187.3) .. controls (94.6,184.3) and (96.98,181.84) .. (99.97,181.78) .. controls (102.97,181.73) and (105.43,184.12) .. (105.49,187.11) -- cycle ; 
				\draw   (182.72,222.43) .. controls (182.77,225.42) and (180.39,227.89) .. (177.39,227.94) .. controls (174.4,227.99) and (171.93,225.61) .. (171.88,222.62) .. controls (171.83,219.62) and (174.21,217.16) .. (177.21,217.1) .. controls (180.2,217.05) and (182.67,219.44) .. (182.72,222.43) -- cycle ;
				\draw   (273.03,188.33) .. controls (273.08,191.32) and (270.7,193.79) .. (267.71,193.84) .. controls (264.71,193.89) and (262.25,191.51) .. (262.19,188.52) .. controls (262.14,185.52) and (264.53,183.06) .. (267.52,183) .. controls (270.51,182.95) and (272.98,185.34) .. (273.03,188.33) -- cycle ;
				\draw   (226.53,186.38) .. controls (226.59,189.38) and (224.2,191.84) .. (221.21,191.9) .. controls (218.22,191.95) and (215.75,189.56) .. (215.7,186.57) .. controls (215.65,183.58) and (218.03,181.11) .. (221.02,181.06) .. controls (224.01,181.01) and (226.48,183.39) .. (226.53,186.38) -- cycle ;
				\draw  [fill={rgb, 255:red, 13; green, 13; blue, 13 }  ,fill opacity=1 ] (153.71,197.45) .. controls (153.71,196.54) and (154.45,195.8) .. (155.36,195.8) .. controls (156.28,195.8) and (157.02,196.54) .. (157.02,197.45) .. controls (157.02,198.37) and (156.28,199.11) .. (155.36,199.11) .. controls (154.45,199.11) and (153.71,198.37) .. (153.71,197.45) -- cycle ;
				\draw  [fill={rgb, 255:red, 13; green, 13; blue, 13 }  ,fill opacity=1 ] (155.35,185.17) .. controls (155.35,184.25) and (156.09,183.51) .. (157,183.51) .. controls (157.92,183.51) and (158.66,184.25) .. (158.66,185.17) .. controls (158.66,186.08) and (157.92,186.83) .. (157,186.83) .. controls (156.09,186.83) and (155.35,186.08) .. (155.35,185.17) -- cycle ;
				\draw   (205.6,152.19) -- (212.48,156.01) -- (205.6,159.83) ;
				\draw    (356.7,91.23) .. controls (365.91,48.85) and (384.58,112.85) .. (399.24,89.52) ;
				\draw    (427.09,107.73) .. controls (403.29,100.12) and (356.94,121.22) .. (382.36,81.56) ;
				\draw    (382.36,81.56) .. controls (403.91,54.18) and (402.58,116.83) .. (430.24,79.85) ;
				\draw    (399.24,89.52) .. controls (411.91,64.18) and (437.44,71.69) .. (430.24,79.85) ;
				\draw [color={rgb, 255:red, 248; green, 72; blue, 61 }  ,draw opacity=1 ]   (440.58,82.26) .. controls (476.58,39.88) and (533.24,50.47) .. (561.57,61.6) ;
				\draw [color={rgb, 255:red, 248; green, 72; blue, 61 }  ,draw opacity=1 ]   (440.58,82.26) .. controls (433.64,91.76) and (456.82,91.93) .. (458.15,93.26) ;
				\draw [color={rgb, 255:red, 248; green, 72; blue, 61 }  ,draw opacity=1 ]   (459.79,80.97) .. controls (490.5,64.69) and (498.16,80.85) .. (500.72,82.85) ;
				\draw [color={rgb, 255:red, 248; green, 72; blue, 61 }  ,draw opacity=1 ]   (500.72,82.85) .. controls (519.67,110.52) and (524.91,61.7) .. (543.91,78.37) ;
				\draw    (356.7,91.23) .. controls (348.83,127.85) and (440.36,163.79) .. (504.31,78.85) ;
				\draw    (504.31,78.85) .. controls (521.72,60.85) and (533.5,115.85) .. (547.84,81.18) ; 
				\draw [color={rgb, 255:red, 248; green, 72; blue, 61 }  ,draw opacity=1 ]   (543.91,78.37) .. controls (554.36,88.94) and (559.24,97.7) .. (571.37,83.65) ;
				\draw    (547.84,81.18) .. controls (565.09,52.39) and (609.87,130.69) .. (515.58,122.97) ;
				\draw [color={rgb, 255:red, 248; green, 72; blue, 61 }  ,draw opacity=1 ]   (561.57,61.6) .. controls (574.72,67.83) and (579.02,76.67) .. (571.37,83.65) ;
				\draw    (427.09,107.73) .. controls (464.94,118.26) and (509.66,122.34) .. (515.58,122.97) ; 
				\draw   (408.27,82.91) .. controls (408.33,85.9) and (405.94,88.37) .. (402.95,88.42) .. controls (399.96,88.48) and (397.49,86.09) .. (397.44,83.1) .. controls (397.39,80.11) and (399.77,77.64) .. (402.76,77.59) .. controls (405.75,77.54) and (408.22,79.92) .. (408.27,82.91) -- cycle ;
				\draw   (468.99,115.6) .. controls (469.04,118.6) and (466.66,121.06) .. (463.67,121.12) .. controls (460.68,121.17) and (458.21,118.78) .. (458.16,115.79) .. controls (458.1,112.8) and (460.49,110.33) .. (463.48,110.28) .. controls (466.47,110.23) and (468.94,112.61) .. (468.99,115.6) -- cycle ;
				\draw   (575.82,84.13) .. controls (575.87,87.12) and (573.49,89.59) .. (570.49,89.64) .. controls (567.5,89.7) and (565.03,87.31) .. (564.98,84.32) .. controls (564.93,81.33) and (567.31,78.86) .. (570.31,78.81) .. controls (573.3,78.76) and (575.76,81.14) .. (575.82,84.13) -- cycle ; 
				\draw   (529.32,82.19) .. controls (529.37,85.18) and (526.99,87.65) .. (524,87.7) .. controls (521.01,87.75) and (518.54,85.37) .. (518.49,82.38) .. controls (518.43,79.38) and (520.82,76.92) .. (523.81,76.86) .. controls (526.8,76.81) and (529.27,79.2) .. (529.32,82.19) -- cycle ;
				\draw  [color={rgb, 255:red, 248; green, 72; blue, 61 }  ,draw opacity=1 ][fill={rgb, 255:red, 248; green, 72; blue, 61 }  ,fill opacity=1 ] (456.49,93.26) .. controls (456.49,92.34) and (457.24,91.6) .. (458.15,91.6) .. controls (459.07,91.6) and (459.81,92.34) .. (459.81,93.26) .. controls (459.81,94.17) and (459.07,94.92) .. (458.15,94.92) .. controls (457.24,94.92) and (456.49,94.17) .. (456.49,93.26) -- cycle ;
				\draw  [color={rgb, 255:red, 248; green, 72; blue, 61 }  ,draw opacity=1 ][fill={rgb, 255:red, 248; green, 72; blue, 61 }  ,fill opacity=1 ] (458.13,80.97) .. controls (458.13,80.06) and (458.88,79.32) .. (459.79,79.32) .. controls (460.71,79.32) and (461.45,80.06) .. (461.45,80.97) .. controls (461.45,81.89) and (460.71,82.63) .. (459.79,82.63) .. controls (458.88,82.63) and (458.13,81.89) .. (458.13,80.97) -- cycle ;
				\draw   (508.39,47.99) -- (515.27,51.81) -- (508.39,55.63) ;
				\draw    (353.7,195.56) .. controls (362.91,153.18) and (381.58,217.18) .. (396.24,193.85) ;
				\draw [color={rgb, 255:red, 248; green, 72; blue, 61 }  ,draw opacity=1 ]   (424.09,212.06) .. controls (400.29,204.45) and (353.79,211.9) .. (379.36,185.9) ;
				\draw [color={rgb, 255:red, 248; green, 72; blue, 61 }  ,draw opacity=1 ]   (379.36,185.9) .. controls (400.91,158.52) and (400.58,224.85) .. (432.24,186.52) ;
				\draw    (396.24,193.85) .. controls (408.91,168.52) and (424.91,179.85) .. (427.24,184.18) ;
				\draw [color={rgb, 255:red, 248; green, 72; blue, 61 }  ,draw opacity=1 ]   (432.24,186.52) .. controls (468.24,144.14) and (530.24,154.8) .. (558.57,165.93) ;
				\draw    (427.24,184.18) .. controls (437.22,196.47) and (453.82,196.26) .. (455.15,197.59) ;
				\draw    (456.79,185.31) .. controls (487.5,169.02) and (495.16,185.18) .. (497.72,187.18) ;
				\draw    (497.72,187.18) .. controls (516.67,214.85) and (534.95,156.44) .. (540.95,185.58) ;
				\draw    (353.7,195.56) .. controls (345.83,232.18) and (458.08,255.84) .. (501.31,183.18) ;
				\draw    (501.31,183.18) .. controls (518.72,165.18) and (545.52,218.16) .. (540.95,185.58) ;
				\draw [color={rgb, 255:red, 248; green, 72; blue, 61 }  ,draw opacity=1 ]   (549.23,192.44) .. controls (555.23,196.44) and (559.52,196.16) .. (568.37,187.99) ; 
				\draw [color={rgb, 255:red, 248; green, 72; blue, 61 }  ,draw opacity=1 ]   (548.38,181.01) .. controls (568.95,170.44) and (606.87,235.02) .. (512.58,227.31) ;
				\draw [color={rgb, 255:red, 248; green, 72; blue, 61 }  ,draw opacity=1 ]   (558.57,165.93) .. controls (571.72,172.16) and (576.02,181) .. (568.37,187.99) ;
				\draw [color={rgb, 255:red, 248; green, 72; blue, 61 }  ,draw opacity=1 ]   (424.09,212.06) .. controls (461.94,222.59) and (506.66,226.67) .. (512.58,227.31) ;
				\draw   (405.27,187.25) .. controls (405.33,190.24) and (402.94,192.71) .. (399.95,192.76) .. controls (396.96,192.81) and (394.49,190.43) .. (394.44,187.43) .. controls (394.39,184.44) and (396.77,181.97) .. (399.76,181.92) .. controls (402.75,181.87) and (405.22,184.25) .. (405.27,187.25) -- cycle ;
				\draw   (465.99,219.94) .. controls (466.04,222.93) and (463.66,225.4) .. (460.67,225.45) .. controls (457.68,225.5) and (455.21,223.12) .. (455.16,220.13) .. controls (455.1,217.13) and (457.49,214.67) .. (460.48,214.61) .. controls (463.47,214.56) and (465.94,216.94) .. (465.99,219.94) -- cycle ;
				\draw   (572.82,188.47) .. controls (572.87,191.46) and (570.49,193.93) .. (567.49,193.98) .. controls (564.5,194.03) and (562.03,191.65) .. (561.98,188.65) .. controls (561.93,185.66) and (564.31,183.19) .. (567.31,183.14) .. controls (570.3,183.09) and (572.76,185.47) .. (572.82,188.47) -- cycle ;
				\draw   (529.75,185.38) .. controls (529.8,188.37) and (527.42,190.84) .. (524.43,190.89) .. controls (521.44,190.94) and (518.97,188.56) .. (518.92,185.57) .. controls (518.86,182.58) and (521.25,180.11) .. (524.24,180.06) .. controls (527.23,180) and (529.7,182.39) .. (529.75,185.38) -- cycle ;
				\draw  [fill={rgb, 255:red, 13; green, 13; blue, 13 }  ,fill opacity=1 ] (453.49,197.59) .. controls (453.49,196.68) and (454.24,195.94) .. (455.15,195.94) .. controls (456.07,195.94) and (456.81,196.68) .. (456.81,197.59) .. controls (456.81,198.51) and (456.07,199.25) .. (455.15,199.25) .. controls (454.24,199.25) and (453.49,198.51) .. (453.49,197.59) -- cycle ;
				\draw  [fill={rgb, 255:red, 13; green, 13; blue, 13 }  ,fill opacity=1 ] (455.13,185.31) .. controls (455.13,184.39) and (455.88,183.65) .. (456.79,183.65) .. controls (457.71,183.65) and (458.45,184.39) .. (458.45,185.31) .. controls (458.45,186.22) and (457.71,186.96) .. (456.79,186.96) .. controls (455.88,186.96) and (455.13,186.22) .. (455.13,185.31) -- cycle ;
				\draw   (505.39,152.32) -- (512.27,156.14) -- (505.39,159.96) ;
				\draw    (56.91,91.09) .. controls (60.04,57.69) and (88.9,71.69) .. (76.12,90.05) ;
				\draw   (110.03,127.85) -- (116.91,131.67) -- (110.03,135.49) ;
				\draw   (421.21,75.87) -- (413.57,74.02) -- (419.18,68.5) ;
				\draw   (189.96,196.89) -- (191.86,189.26) -- (197.34,194.91) ;
				\draw [color={rgb, 255:red, 248; green, 72; blue, 61 }  ,draw opacity=1 ]   (548.38,181.01) .. controls (545.81,183.3) and (546.09,189.58) .. (549.23,192.44) ;
				\draw   (489.52,183.97) -- (484.12,178.24) -- (491.84,176.69) ;
				\draw (98.21,53.07) node [anchor=north west][inner sep=0.75pt]  [color={rgb, 255:red, 248; green, 72; blue, 61 }  ,opacity=1 ]  {$\ell _{1}$};
				\draw (50.06,53) node [anchor=north west][inner sep=0.75pt]    {$\ell _{2}$};
				\draw (442.87,53) node [anchor=north west][inner sep=0.75pt]  [color={rgb, 255:red, 248; green, 72; blue, 61 }  ,opacity=1 ]  {$\ell _{1}$};
				\draw (394.73,53) node [anchor=north west][inner sep=0.75pt]    {$\ell _{2}$};
				\draw (216.99,193.58) node [anchor=north west][inner sep=0.75pt]  [color={rgb, 255:red, 248; green, 72; blue, 61 }  ,opacity=1 ]  {$\ell _{1}$};
				\draw (170.85,192.51) node [anchor=north west][inner sep=0.75pt]    {$\ell _{2}$};
				\draw (552.3,195.91) node [anchor=north west][inner sep=0.75pt]  [color={rgb, 255:red, 248; green, 72; blue, 61 }  ,opacity=1 ]  {$\ell _{1}$};
				\draw (504.15,195.84) node [anchor=north west][inner sep=0.75pt]    {$\ell _{2}$};
				\draw (148,128.45) node [anchor=north west][inner sep=0.75pt]   [align=left] {(a)};
				\draw (457,128.45) node [anchor=north west][inner sep=0.75pt]   [align=left] {(b)};
				\draw (457,238.45) node [anchor=north west][inner sep=0.75pt]   [align=left] {(d)};
				\draw (148,238.45) node [anchor=north west][inner sep=0.75pt]   [align=left] {(c)};
			\end{tikzpicture}
			\caption{(a)--(d): Diagrams obtained by 1-smoothing at flat crossings $c_1$, $c_2$, $c_3$ and $c_4$ in the shadow of $D_1$.} \label{figXX}
		\end{center}
	\end{figure}
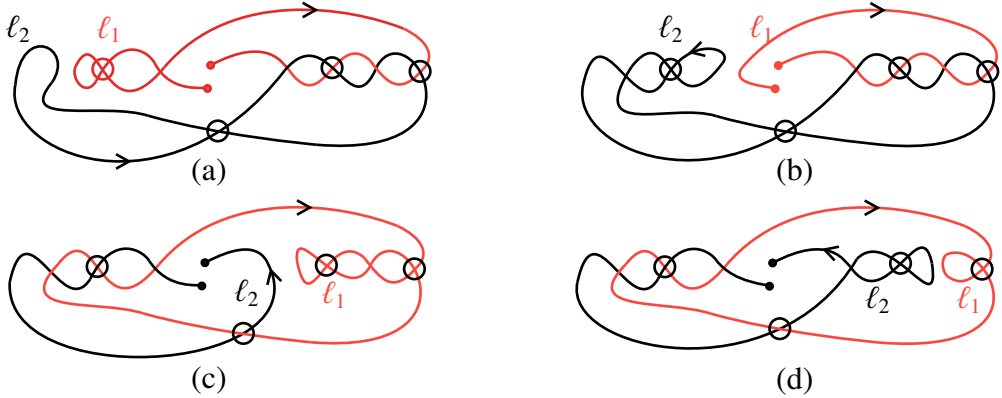 
	Next, by {\bf Rule~2}, for any pair of classical crossings $e,f \in \{c_1,c_2,c_4,d_{c_3}\}$, denote $e=(a_e,b_e)$, $f=(a_f,b_f)$ and compute
	$a_eb_e \cdot a_fb_f$, which equals to the number of arrows with tails in  $(a_e,b_e)^\circ$ and heads in $(a_f,b_f)^\circ$ minus the number of arrows with tails in $(a_f,b_f)^\circ$ and heads in $(a_e,b_e)^\circ$.
	It follows from Fig.~\ref{fig117}(a) that $\epsilon(c_1,c_2)=0$, $\epsilon(c_1,c_4)=-1$, $\epsilon(c_1,d_{c_3})=-1$, $\epsilon(c_2,c_4)=-1$, $\epsilon(c_2,d_{c_3})=-1$, $\epsilon(c_4,d_{c_3})=0$. We obtain
	$$
	\begin{gathered} 
		{\bf b}_{c_3}(c_1,c_2)=0-0+0=0, \quad {\bf b}_{c_3}(c_1,c_4)=0-1-1=-2,  \cr {\bf b}_{c_3}(c_1,d_{c_3})=0-2-1=-3,  \quad 
		{\bf b}_{c_3}(c_2,c_4)=0-0-1=-1, \cr {\bf b}_{c_3}(c_2,d_{c_3})=0-1-1=-2, \quad {\bf b}_{c_3}(c_4,d_{c_3})=0-0-0=0.
	\end{gathered}
	$$
	All remaining entries are obtained by skew-symmetry and the diagonal entries ${\bf b}_{c_3}(e,e)=0$, giving the full matrix $(G_{c_3}, s_{c_3}, d_{c_3}, {\bf b}_{c_3})$, written as $M_{{\bf b}_{c_3}}$. We represent the obtained values into the following matrix:
	$$	
	M_{{\bf b}_{c_3}} = \qquad \bordermatrix{
		&     s_{c_3} &  c_1 & c_2 & c_4 & d_{c_3}  \cr
		s_{c_3}   & 0 &  2 & 2 &-2 &-2  \cr
		c_1   &-2 &  0 & 0 &-2 &-3  \cr
		c_2   &-2 &  0 & 0 &-1 &-2  \cr
		c_4   & 2 &  2 & 1 & 0 & 0  \cr
		d_{c_3}   & 2 &  3 & 2 & 0 & 0  \cr	 
	}
	$$

	We claim that the matrix $M_{{\bf b}_{c_3}}$ is primitive. By the definition, a matrix is \textit{primitive} if it cannot be obtained from another SBM by elementary extensions, even after applying singularity switch operations.
	
	We verify that $M_{{\bf b}_{c_3}}$ cannot be obtained by any of the operations $\widetilde{M}_1$, $\widetilde{M}_2$, $\widetilde{M}_3$, or $N$. Indeed, it is not obtained from $\widetilde{M}_1$ since no row is identically zero, not from $\widetilde{M}_2$ since no row coincides with $s_{c_3} = [0, 2, 2, -2, -2]$, not from $\widetilde{M}_3$ since no two rows sum to $s_{c_3}$, and not from $N$, since
	\[
	\begin{aligned}
		\operatorname{row}(c_1)+\operatorname{row}(d_{c_3})
		&=(0,3,2,-2,-3),\\
		\operatorname{row}(c_2)+\operatorname{row}(d_{c_3})
		&=(0,3,2,-1,-2),\\
		\operatorname{row}(c_4)+\operatorname{row}(d_{c_3})
		&=(4,5,3,0,0),
	\end{aligned}
	\]
	none of which is equal to
	\[
	\operatorname{row}(s_{c_3})=(0,2,2,-2,-2).
	\]
	Therefore, $M_{{\bf b}_{c_3}}$ is primitive.
	
	Moreover, $d_{c_3} \in {\bf b}_{c_3}$ is neither annihilating-like nor core-like. Indeed, $d_{c_3}$ is not annihilating-like since there exists $h \in G$ such that ${\bf b}_{c_3}(d_{c_3},h) \neq 0$. It is not core-like since there exists $h \in G$ such that ${\bf b}_{c_3}(d_{c_3},h) \neq {\bf b}_{c_3}(s_{c_3},h).$
	
	\underline{SBM $(G_{c_4}, s_{c_4}, d_{c_4}, {\bf b}_{c_4})$.}
	Denote by ${\bf b}_{c_4}$ the skew-symmetric map in the SBM $(G_{c_4}, s_{c_4}, $ $d_{c_4},{\bf b}_{c_4})$ corresponding to the singular virtual open string associated with crossing $c_4$, shown in Fig.~\ref{fig117}(b). Equivalently, one can view this string as the flat singular virtual knotoid obtained by performing the gluing operation at crossing $c_4$ in Fig.~\ref{fig116}(right).
	
	The same procedure applies to ${\bf b}_{c_4}$: first compute ${\bf b}_{c_4}(e,s_{c_4})$ via Rule~1 with the same ordering of components, and then compute ${\bf b}_{c_4}(e,f)$ for all pairs $e,f \in \{c_1,c_2,c_3,d_{c_4}\}$ via Rule~2. The remaining entries are filled using skew-symmetry and zeros on the diagonal, producing the full matrix $(G_{c_4}, s_{c_4}, d_{c_4}, {\bf b}_{c_4})$, written as $M_{{\bf b}_{c_4}}$. We organize the computed result in a matrix as follows: 
	$$	
	M_{{\bf b}_{c_4}} = \qquad \bordermatrix{
		&     s_{c_4} &  c_1 & c_2 & c_3 & d_{c_4}  \cr
		s_{c_4}   & 0 &  2 & 2 &-2 &-2  \cr
		c_1   &-2 &  0 & 0 &-3 &-2  \cr
		c_2   &-2 &  0 & 0 &-2 &-1  \cr
		c_3   & 2 &  3 & 2 & 0 & 0  \cr
		d_{c_4}  & 2 &  2 & 1 & 0 & 0  \cr	 
	} 
	$$	
	
	We claim that the matrix $M_{{\bf b}_{c_4}}$ is primitive. We need to verify that $M_{{\bf b}_{c_3}}$ cannot be obtained by any of the operations $\widetilde{M}_1$, $\widetilde{M}_2$, $\widetilde{M}_3$, or $N$. Indeed, it is not obtained from $\widetilde{M}_1$ since no row is identically zero, not from $\widetilde{M}_2$ since no row coincides with $s_{c_4} = [0, 2, 2, -2, -2]$, not from $\widetilde{M}_3$ since no two rows sum to $s_{c_4}$, and not from $N$, since
	\[
	\begin{aligned}
		\operatorname{row}(c_1)+\operatorname{row}(d_{c_4})
		&=(0,2,1,-3,-2),\\
		\operatorname{row}(c_2)+\operatorname{row}(d_{c_4})
		&=(0,2,1,-2,-1),\\
		\operatorname{row}(c_4)+\operatorname{row}(d_{c_4})
		&=(4,5,3,0,0),
	\end{aligned}
	\]
	none of which is equal to
	\[
	\operatorname{row}(s_{c_4})=(0,2,2,-2,-2).
	\]
	Therefore, $M_{{\bf b}_{c_4}}$ is primitive.
	
	By the same arguments as for $M_{{\bf b}_{c_3}}$, we conclude that $d_{c_4} \in {\bf b}_{c_4}$ is neither annihilating-like nor core-like.
	
	\underline{Comparing $M_{{\bf b}_{c_3}}$ and $M_{{\bf b}_{c_4}}$.}
	We now determine whether $M_{{\bf b}_{c_3}}$ and $M_{{\bf b}_{c_4}}$ are related by an isomorphism or by a composition of an isomorphism with a single move: $\widetilde{M}_1^{-1} \circ N \circ \widetilde{M}_2$, $\widetilde{M}_2^{-1} \circ N \circ \widetilde{M}_1$, or $N$.
	
	First, $M_{{\bf b}_{c_3}}$ and $M_{{\bf b}_{c_4}}$ are not isomorphic follows directly from the proof of Theroem 3.11 in\cite{Hen}. Indeed, $M_{{\bf b}_{c_3}}$ has elements $\{s_{c_3},c_1,c_2,c_4,d_{c_3}\}$, while $M_{{\bf b}_{c_4}}$ has elements $\{s_{c_4},c_1,c_2,c_3,d_{c_4}\}$. And there is no bijection $G_{c_3} \to G_{c_4}$ sending $s_{c_3}$ to $s_{c_4}$, $d_{c_3}$ to $d_{c_4}$ and transforms ${\bf b}_{c_3}$ into ${\bf b}_{c_4}$.
	
	Moreover, $M_{{\bf b}_{c_3}}$ and $M_{{\bf b}_{c_4}}$ cannot be related by any composition of an isomorphism with a single move among $\widetilde{M}_1^{-1} \circ N \circ \widetilde{M}_2$, $\widetilde{M}_2^{-1} \circ N \circ \widetilde{M}_1$, or $N$. Since both $M_{{\bf b}_{c_3}}$ and $M_{{\bf b}_{c_4}}$ are primitive, the inverse operations $\widetilde{M}_1^{-1}$, $\widetilde{M}_2^{-1}$, and $N^{-1}$ are not applicable.
	
	Therefore, by Theorem~\ref{th4.1}, $M_{{\bf b}_{c_3}}$ and $M_{{\bf b}_{c_4}}$ are not homologous. It then follows from Theorem~\ref{th2.1} that the two singular virtual open strings in Fig.~\ref{fig117} are not homotopic. Consequently, the terms in $\mathcal{G}(K_1)$ corresponding to crossings $c_3$ and $c_4$ do not cancel.
	
	On the other hand, We see that the terms in $\mathcal{G}(K_2)$ corresponding to crossings $3$ and $4$ in $K_2$ are the same as in $\mathcal{G}(K_1)$, except with opposite sign. Thus, $\mathcal{G}(K_1)-\mathcal{G}(K_2)$ is non-zero. Indeed this difference has two terms, one with coefficient $+2$ and one with coefficient $-2$. By the analysis above, these terms do not cancel, hence $\mathcal{G}(K_1) \neq \mathcal{G}(K_2)$. This completes the proof of Theorem~\ref{prop4.1}. 
\end{proof}

\section{Conclusion}

In this paper, we introduce a two-variable parity polynomial $\mathbf{P}_K(x,y)$ for oriented virtual knotoids and prove that it is an invariant of virtual knotoids. We study its basic properties, including its behavior under inversion and mirror reflection, and show that it is a Vassiliev invariant of order one. We also give an example showing that $\mathbf{P}_K(x,y)$ distinguishes virtual knotoids that are not distinguished by the odd writhe, and prove that it is equivalent to the affine index polynomial. Finally, we compare $\mathbf{P}_K(x,y)$ with the Petit gluing invariant and prove that the latter is strictly stronger.

For future work, it would be interesting to know which properties of a virtual knotoid $K$ can be detected by $P_K(x,y)$. Another natural problem is to understand the relationship between $P_K(x,y)$ and the three-variable invariant in \cite{FLV}.

Recently, knot data analysis (KDA) incorporated knot theory into topological data analysis by introducing the multiscale Gauss link integral and combining it with machine learning and deep learning~\cite{KDA}. The resulting framework has been applied to protein flexibility analysis, protein--ligand interactions, human Ether-à-go-go-Related Gene potassium channel blockade screening, and quantitative toxicity assessment. This suggests a possible direction for studying applications of the polynomial invariant $\mathbf{P}_K(x,y)$ in knot data analysis.

\end{document}